\documentclass[a4paper, 11pt, final]{amsart}
\usepackage{multicol}
\usepackage{multirow}
\usepackage[pdftex]{graphicx}
\usepackage{amsmath, amsfonts, amssymb, amsthm, float}
\usepackage{mathtools}
\usepackage{mathrsfs}
\usepackage{showlabels}
\usepackage{enumitem}
\setlist{nolistsep}
\usepackage[a4paper, bottom=2cm, right=2.3cm, top=2cm]{geometry}
\usepackage{rotating}
\usepackage{array}
\usepackage{lmodern}
\usepackage{xcolor}
\usepackage{url}
\usepackage{tikz-cd}
\usepackage{bm}
\usepackage{tabularx}
\usepackage{etoolbox}
\usepackage[colorlinks=true,linkcolor=black,citecolor=blue!80!black]{hyperref}
\usepackage{xcolor}
\theoremstyle{plain}
\theoremstyle{definition}
\usepackage{appendix}
\usepackage{pdflscape}
\usepackage[capitalize]{cleveref}
\newtheorem{definition}{Definition}[section]

\theoremstyle{remark}
\newtheorem{remark}[definition]{Remark}

\newtheorem{corollary}[definition]{Corollary}

\newtheorem{lemma}[definition]{Lemma}

\newtheorem{theorem}[definition]{Theorem}
\newtheorem{proposition}[definition]{Proposition}

\newcounter{claim}[definition]

\theoremstyle{remark}

\theoremstyle{remark}
\newtheorem{notation}[definition]{Notation}

\theoremstyle{definition}

\renewcommand{\epsilon}{\varepsilon}
\renewcommand{\bar}{\overline}
\renewcommand{\hat}{\widehat}
\renewcommand{\leq}{\leqslant}
\renewcommand{\geq}{\geqslant}

\newcommand{\normaleq}{\trianglelefteq}

\newcommand{\divides}{\bigm|}
\newcommand{\B}{\mathcal{B}}
\newcommand{\E}{\mathcal{E}}
\newcommand{\F}{\mathcal{F}}
\newcommand{\G}{\mathcal{G}}

\newcommand{\N}{\mathbb{N}}

\newcommand{\KK}{\mathbb{K}}
\newcommand{\FF}{\mathbb{F}}
\newcommand{\SL}{\operatorname{SL}}

\newcommand{\syl}{\mathrm{Syl}}
\newcommand{\Syl}{\operatorname{Syl}}
\newcommand{\GL}{\mathrm{GL}}
\newcommand{\Sp}{\mathrm{Sp}}

\newcommand{\0}{\emptyset}

\newcommand{\PSL}{\mathrm{PSL}}
\newcommand{\PSp}{\mathrm{PSp}}

\newcommand{\Sym}{\mathrm{Sym}}
\newcommand{\Alt}{\mathrm{Alt}}

\newcommand{\Aut}{\mathrm{Aut}}
\newcommand{\Out}{\mathrm{Out}}
\newcommand{\Inn}{\mathrm{Inn}}
\newcommand {\End}{\mathrm{End}}
\newcommand {\Gal}{\mathrm{Gal}}

\newcommand{\Hom}{\mathrm{Hom}}
\newcommand{\Iso}{\mathrm{Iso}}

\newcommand{\diag}{\mathrm{diag}}

\def \ov {\overline}
\newcommand{\gen}[1]{\langle #1 \rangle}

\renewcommand{\hat}{\widehat}

\numberwithin{equation}{section}
\makeatletter
\let\c@theorem\c@equation
\makeatother

\begin{document}

\author{Valentina Grazian}
\address{Dipartimento di Matematica, Universit\`{a} di Padova, 35121, Italia}
\email{valentina.grazian@math.unipd.it}

\author{Chris Parker}
\address{School of Mathematics, University of Birmingham, B15 2TT,
  United Kingdom}
\email{c.w.parker@bham.ac.uk}

\author{Jason Semeraro}
\address{Department of Mathematics, Loughborough University, LE11 3TT,
  United Kingdom}
\email{j.p.semeraro@lboro.ac.uk}

\author{Martin van Beek}
\address{Department of Mathematics, University of Manchester, Manchester, M13 9PL, United Kingdom}
\email{martin.vanbeek@manchester.ac.uk}

\keywords{Exotic fusion systems; groups of Lie type; $p$-groups; abelian essential}
\subjclass[2020]{20D20; 20D05; 20E42; 20C33}

\begin{abstract}
Let $q$ be a power of a fixed prime $p$. We classify up to isomorphism all simple saturated fusion systems on a certain class of $p$-groups constructed  from the polynomial representations of $\SL_2(q)$, which includes the Sylow $p$-subgroups of $\GL_3(q)$ and $\Sp_4(q)$ as special cases. The resulting list includes all Clelland--Parker fusion systems, a simple exotic fusion system discovered by Henke--Shpectorov, and a new infinite family of exotic examples.
\end{abstract}

\title{Fusion systems related to polynomial representations of $\SL_2(q)$}

\maketitle
\section{Introduction}
For a prime $p$, the analysis of $p$-local structure in finite groups and their modular representation theory is naturally expressed in the language of saturated fusion systems on $p$-groups. Since $p$-groups of a given order are overwhelmingly dominated by those of nilpotency class $2$, the work of the fourth author~\cite{vB24} implies that almost all $p$-groups admit no reduced saturated fusion system. In fact, the proportion of $p$-groups supporting a reduced saturated fusion system appears to be vanishingly small; this expectation is reinforced by computational evidence (see, for example, \cite{parkersemerarocomputing}). Consequently, recognizing the $p$-groups that  do support reduced fusion systems, and classifying the fusion systems that occur on them, is an interesting and difficult problem. The results of this paper contribute to this programme (see \cite[Questions III.7.4, IV.7.1]{AKO}).

A saturated fusion system on a finite $p$–group $S$ is said to be \emph{exotic} if it does not arise as the fusion system of any finite group with Sylow $p$-subgroup $S$. As we shall see, the fusion systems studied in this paper are predominantly exotic.

One source of $p$-groups which support an  exotic fusion system are the Sylow $p$-subgroups of non-abelian simple groups.  In particular, a substantial body of work has examined the saturated fusion systems on a $p$-group $S$ which is the Sylow $p$-subgroup of a Lie type group of characteristic $p$ and rank $2$ \cite{RV,Clelland,G2p,U43Fusion,RaulU4p,U4qG2q,HS}, culminating in the classification obtained in \cite{Rank2Sylows}. This direction of research has produced several families of exotic fusion systems (see, for example, \cite{RV,G2p,HS}). For $q$ a power of $p$, the $p$-groups considered here, which we christen \emph{polynomial $p$-groups}, are generalisations of the Sylow $p$-subgroups of $\SL_3(q)$ and $\Sp_4(q)$ and have their origins in polynomial representations of $\SL_2(q)$. Thus they lie directly in the rank $2$ Lie type landscape. Moreover, polynomial $p$-groups are known to support reduced fusion systems and therefore provide a natural source of $p$-groups  for the general programme described above. From a different perspective, the polynomial $p$-groups denoted $S_n(q)$ below are remarkably  like ``$q$-versions" of maximal class groups (see Lemma~\ref{Somnibus}). Thus this work also complements the work in \cite{GPMaxClass}.

In this article, we provide a complete classification of saturated fusion systems $\F$ on polynomial $p$-groups with $O_p(\F)=1$ (core-free) and thus a determination of all the reduced saturated fusion systems on such groups. Our key results, Theorems \ref{thm: main} and \ref{thm: upVMain}, form an important step toward a broader project carried out with collaborators in which we determine all reduced saturated fusion systems $\F$ on a $p$-group $S$ that possess an $\F$-essential subgroup that is both abelian and normal in $S$. This project extends the work initiated in \cite{p.index1,p.index2,p.index3}, and the classification obtained here supplies input required for the principal families of fusion systems arising in that extension.

Let $\KK$ be a finite field of characteristic $p$ and order $q$, and $G=\GL_2(q)$. For $n \ge 1$, define $V_n(q)$  to be the set of homogeneous polynomials over $\KK$ of degree $n$ in two commuting variables. Then $V_n(q)$ becomes a $\KK G$-module of dimension $n+1$ with respect to the natural polynomial action of $\GL_2(q)$. The modules $V_n(q)$ are irreducible precisely when $0\leq n \leq p-1$. Writing $\Lambda(q)$ for the dual module of $V_p(q)$, define: $$P_n(q) = G\ltimes V_n(q), \hspace{2mm} S_n(q) \in \syl_p(P_n(q) ), \hspace{2mm} P_\Lambda(q) = G\ltimes \Lambda(q), \hspace{2mm}  \mbox{ and } \hspace{2mm} S_{\Lambda}(q) \in \Syl_p(P_\Lambda(q)).$$  Note that $$|S_n(q)|=q|V_n(q)|=q^{n+2}$$ and $$|S_\Lambda(q)|=q|\Lambda(q)|=q^{p+2}.$$ 
In this paper, we refer to the $p$-groups $S_n(q)$ and $S_{\Lambda}(q)$ as \emph{polynomial $p$-groups}. The primary objective of this work is to determine up to isomorphism all core-free saturated fusion systems $\F$   on $S_\Lambda(q)$ and $S_n(q)$ when $1\leq n \leq p-1$.

In \cite{ClellandParker2010}, Clelland and the second author  introduced two infinite families of saturated fusion systems on the groups $S_n(q)$ with $1\leq n \leq p-1$ which can be viewed as generalisations of the $p$-fusion systems of $\SL_3(q)$ and, when $q$ is odd, $\Sp_4(q)$. Indeed, these groups  have  parabolic subgroups which are closely related to $P_1(q)$ and $P_2(q)$ respectively, and the Clelland--Parker fusion systems possess an essential subgroup normaliser of the form $P_n(q)$. Other than in some very small cases, the fusion systems exhibited in \cite{ClellandParker2010} are exotic (see \cite[Theorems 5.1 and 5.2]{ClellandParker2010}). In Section \ref{sec:poly}, we describe these  systems and variants given by adding ``field automorphisms" and by removing essential subgroups via a process known as \emph{pruning} (see \cite[Lemma 6.4]{parkersemerarocomputing}). Some of these examples  had previously been considered in \cite{henke2023punctured}. Collectively, we refer to all of these fusion systems as \emph{polynomial fusion systems.}

Suppose $S=S_n(q)$. When $n=1$, $S$ is a Sylow $p$-subgroup of $\SL_3(q)$ and core-free saturated fusion systems on  $S$ were classified by Clelland in his PhD thesis \cite{Clelland},  generalising the pioneering work of Ruiz and Viruel \cite{RV} when $q=p$ (see Proposition \ref{prop:murray}). When $q$ is odd and $n=2$, $S$ is a Sylow $p$-subgroup of $\Sp_4(q)$, and core-free fusion systems on $S_2(q)$ were studied by Henke--Shpectorov in the unpublished manuscript \cite{HS}. Apart from the polynomial fusion systems, they discover two exotic fusion systems when $q=9$ related to a certain $6$-dimensional $\FF_3$-module for $2^.\PSL_3(4)$ which we refer to as \textit{Henke--Shpectorov fusion systems}. Our first main theorem is a classification of core-free fusion systems on $S_n(q)$ for all $1 \leq n \leq p-1$.

\begin{theorem}\label{thm: main}
Suppose that $p$ is a prime, $q$ is a power of  $p$ and $1\leq n\leq p-1$. If $\F$ is a saturated fusion system on $S=S_n(q)$ with $O_p(\F)=1$ then either $\F$ is a polynomial fusion system or one of the following holds: \vspace{-4mm} \begin{itemize}\item[(1)] $q=p$ and $S$ has an elementary abelian subgroup of index $p$.
\item[(2)] $\F$ is a Henke--Shpectorov fusion system on $S_2(9)$.
\end{itemize}  \vspace{-4mm}
In particular, all  reduced  saturated fusion systems on $S_n(q)$ are known.
\end{theorem}

This gives a characterisation of the Henke--Shpectorov systems as the only exceptional core-free exotic fusion systems on polynomial $p$-groups which occur when $q > p$. Theorem \ref{thm: main} subsumes the results in \cite{Clelland} and \cite{HS} mentioned above, but we rely on the former since the structure of $\Aut(S_n(q))$ when $n=1$ differs from the case $n \ge 2$, and this  precludes from a uniform analysis of this group. We do not rely on the results in \cite{HS} but it is notable that, in contrast to \cite{HS}, we make more liberal use of the Classification of Finite Simple Groups in our analysis  to gain insight into the groups which possess a strongly $p$-embedded subgroup (in \cite{HS} this result is only used to prove that certain fusion systems are exotic). For the conclusion that $\F$ is known in Theorem \ref{thm: main}(1), we rely on the results of Oliver, and Craven--Oliver-Semeraro \cite{p.index1, p.index2}. To sum up, for the proof of Theorem \ref{thm: main} we may assume that $q>p$ is odd and a broad brush overview of the strategy we employ is as follows. We first determine the groups $X$ with Sylow $p$-subgroup $S_n(q)$ which have the property that $O_p(X)$ is isomorphic to $V_n(q)$ as an $S_n(q)$-module. This turns out to be a rather delicate calculation which takes up a good part of Section 5. After that, in Section 6, we determine up to $S_n(q)$-conjugacy the candidates for $\F$-essential subgroups in a saturated fusion system on $S$, before assembling those which form compatible systems in Section 7.

In a later paper the authors, along with other collaborators, will demonstrate that the $p$-groups $S_n(q)$ do not support core-free saturated fusion systems whenever $n>p$. In addition, in this article we show in Proposition \ref{prop: V_pCon1} that, provided $q>p$, all saturated fusion systems on $S_p(q)$ are constrained. Nonetheless, it turns out that the $p$-groups $S_\Lambda(q)$ \textit{do} support a generic family of core-free saturated fusion systems, which we also regard as \textit{polynomial}.

\begin{theorem}\label{thm: NewExotics}
Suppose that $p$ is an odd prime, and $q=p^n>p$. Then $S_{\Lambda}(q)$ supports a simple, exotic fusion system $\F_{\Lambda}(q)$.
\end{theorem}

The system $\F_{\Lambda}(q)$ is  described in Section \ref{sec : NewExotics}, where Theorem \ref{thm: NewExotics} is a consequence of   Theorem \ref{t:newexotics}. In Remark \ref{NotIso}, it is proved that these examples do not coincide with currently known exotic families whenever $q>p$.  We include the case $q=p$ in our construction in Section \ref{sec : NewExotics}, but note that these examples were already considered in \cite{p.index2}.  The systems described in Theorem \ref{thm: NewExotics} arise as $p'$-index subsystems of a certain fixed fusion system $\F_{\Lambda}^*(q)$, also described in Section \ref{sec : NewExotics}. Indeed, these subsystems are parameterised by $p'$-subgroups of $\Aut(\KK)$, and so we obtain several new exotic systems this way. The reader is referred to Section \ref{sec:poly} for a fuller discussion of this correspondence.  For $n>p$, as for the groups $S_n(q)$, we will show in a later paper that the analogous $p$-group associated to the dual module of $V_n(q)$ also does not support core-free saturated fusion systems. We thus focus our attention on the group $S_{\Lambda}(q)$ when $n=p$. Our third main result classifies all core-free fusion systems on this group.

\begin{theorem}\label{thm: upVMain}
Suppose that $p$ is a prime, and $q$ is a power of $p$. If $\F$ is a saturated fusion system on  $S=S_{\Lambda}(q)$ with $O_p(\F)=1$ then either: 
\vspace{-4mm} \begin{itemize}
\item[(1)] $p=2$ and $O^{2'}(\F)$ is isomorphic to the $2$-fusion system of $\PSp_4(q)$; or
\item[(2)] $p$ is odd and either: \begin{enumerate} \item [(a)]$q> p$ and $\F$ is a polynomial fusion system with $\Lambda(q) $ not normal in $\F$; or \item [(b)] $q=p$ and $S$ has an elementary abelian subgroup of index $p$.\end{enumerate}
\end{itemize} \vspace{-4mm}
In particular, all  reduced  saturated fusion systems on $S_\Lambda(q)$ are known.
\end{theorem}

If $p=2$ then $S$ has nilpotency class two and Theorem \ref{thm: upVMain}(1) follows from \cite[Proposition 4.2]{vB24}. There it is shown, in addition, that $\F$ is realised by $\PSp_4(2) \cong \Sym(6)$ when $q=p$. As above, if $q=p$ is odd then $\F$ is known by \cite{p.index1, p.index2}. Thus we may assume that $q>p$ is odd and the strategy for the proof of Theorem \ref{thm: upVMain} is similar to that described above for Theorem \ref{thm: main}.

In Section \ref{sec:poly}, we explore various relationships between polynomial fusion systems. For example, suppose $\F$ is such a system on $S_n(q)$  with two essential subgroups $V_n(q)$ and $Y$. If  $n \ge 2$ and  $|Y|=q^2$, then pruning $V_n(q)$ from $\F$ results in another (core-free) polynomial fusion system. On the other hand, if $n \ge 3$ and $|Y|=q^3$, then pruning $V_n(q)$ yields a fusion system in which $Z(S)$ is normal, but in this case it turns out that $\F/Z(S)$ is again polynomial with exactly one class of essential subgroups of order $q^2$. We also show in Proposition \ref{normalizertower} that polynomial fusion systems with a single class of essential subgroups are ``closed under taking subsystems" in that restricting to all terms of the \emph{normaliser tower} of the essential subgroup delivers fusion systems already documented by Theorem \ref{thm: main}.

As pointed out earlier, the polynomial fusion systems that are realisable are generically  the fusion systems on $S_1(q)$ realised by $\PSL_3(q)\le H \le \Aut(\PSL_3(q))$, and those in $S_2(q)$ realised by $\PSp_4(q)\le H\le \Aut(\PSp_4(q))$.  Our constructions do not use the finiteness of the field at any crucial points and so it is plausible
that one obtains the fusion systems of $\PSL_3(\overline{\FF_p})$ and  $\PSp_4(\overline{\FF_p})$ by replacing $\KK$ with $\overline{\FF_p}$ in the construction of $S_n(q)$. For $n>2$, extending the constructions to $\overline{\FF_p}$, one might further obtain fusion systems which are exotic in the sense that they cannot be realised by any \textit{algebraic group}. One possible context in which to view these examples is within the theory of discrete localities developed by Chermak (see \cite[Appendix A]{chermak2017discrete}), but we do not investigate that here.

The paper is structured as follows. In Section \ref{sec:structureS} we define the polynomial  $p$-groups, determine properties of their automorphism groups and prove a variety of results about them to be used in later arguments. In Section \ref{sec : NewExotics} we introduce the new exotic fusion systems on $S_\Lambda(q)$ following the methodology employed by  Clelland--Parker in \cite{ClellandParker2010}. In Section \ref{sec:poly} we document the polynomial fusion systems which arise, explain how the examples in \cite{ClellandParker2010} and Section 3 can be extended by $p'$-automorphisms and list various properties concerning ``closure'' with respect to taking certain subsystems and quotients as mentioned above. Our classification begins in earnest in Section \ref{sec:reg} and consists of a selection of results to be used in recognising $\KK$-structure and $\SL_2(q)$-module structure among potential $\F$-essential subgroups; these latter subgroups are determined in Section \ref{sec:pot}. We complete the classification in Section \ref{sec:tsfs}.

Our notation for groups is mostly standard and follows \cite{AschbacherFG}. For $A, B \le G$, we define the $n$-fold commutator $[A,B;n]$ by first setting $[A,B;1] := [A,B]$ and then inductively defining $[A,B;n]= [[A,B;n-1],B]$ for $n> 1$. Similarly we mostly adopt standard conventions for fusion systems as found, for example, in \cite{AKO}. One exception is our use of $\mathcal{E}(\F)$ for the set of $\F$-essential subgroups of a saturated fusion system $\F$. Throughout, we will often investigate actions of a $p$-group on one of its abelian sections. For this reason, we will tend to use multiplicative notation for such actions when viewing these sections as part of the $p$-group, and additive notation when viewing them as vector spaces.

\textit{Acknowledgements:} We are grateful to David Craven for sharing his knowledge of indecomposable $\SL_2(q)$-modules, and in particular for his help with Lemmas \ref{lem:NoExt1} and \ref{lem:NoExt}. We also thank Ellen Henke, Justin Lynd and Bob Oliver for helpful discussions. We wish to the thank the anonymous referee whose comments greatly improved the exposition of this work.

This work resulted from the workshop ``Patterns in Exotic Fusion Systems" funded by the Heilbronn Institute for Mathematical Research. The first author is a member of the GNSAGA INdAM research group and thankfully acknowledges its support. The third author gratefully acknowledges funding from the UK Research Council EPSRC for the project EP/W028794/1. The fourth author is supported by the Heilbronn Institute for Mathematical Research.
\vspace{-5mm}
\section{The structure of polynomial $p$-groups and their automorphism group}\label{sec:structureS}

Throughout this work, we let $p$ be an odd prime, although several of the constructions and results in this section are also applicable when $p=2$. Let $\KK$ be a field of characteristic $p$ with group of units $\KK^*$ and order $q=p^m$. Set $A:= \KK[x,y]$ to be the ring of polynomials in two commuting indeterminates $x,y$ with coefficients in $\KK $.

Let $\phi\in \Aut(\KK)$. Then $\phi$ acts on $A$ by acting on the coefficients of $f(x,y)\in A$ yielding the polynomial $f^\phi(x,y)$. This is an $\FF_p$-linear map on $A$ and so $A$ is an $\FF_p\Aut(\KK)$-module.

Define $D = \KK ^*\times \GL_2(\KK )$ and make $A$ into a $\KK D$-module as follows: for $g =\left(\begin{smallmatrix} a & b \\ c & d\end{smallmatrix}\right) \in \GL_2(\KK )$ define $x\cdot g=a x + b y$ and $y\cdot g= c x + d y$ and then, for $(\lambda,g) \in D$ and $f(x,y) \in  A$, define \[f(x,y)\cdot(\lambda,g)= \lambda f(x\cdot g, y\cdot g).\] This action makes $A$ into a $\KK D$-module.

Finally, regarding $A$ as a $\FF_p$-space, both $D$ and $\Aut(\KK)$ act on $A$ and we can check that, for $\phi \in \Aut(\KK)$ and $(\lambda, \left(\begin{smallmatrix} a & b \\ c & d\end{smallmatrix}\right))\in D$ we have $(\lambda, \left(\begin{smallmatrix} a & b \\ c & d\end{smallmatrix}\right) )^\phi =  (\lambda^\phi, \left(\begin{smallmatrix} a^\phi & b ^\phi \\ c^\phi & d^\phi\end{smallmatrix}\right))\in D$. In particular, we can form the semidirect product  $\Aut(\KK)\ltimes D$ and this group acts $\FF_p$-linearly on $A$ by acting on the coefficients  of the polynomials.

\begin{notation}\label{n:DGroups}
Set \[D := \KK ^*\times \GL_2(\KK ),\]
\[D^* := O^{p}(\Aut(\KK))\ltimes D = O_{p'}(\Aut(\KK))\ltimes D, \mbox{ and }\]
\[D^\dagger:=\Aut(\KK)\ltimes D.\]
\end{notation}

A typical element of $D^\dagger$ will be written as \[(\phi,\lambda, \left(\begin{smallmatrix} a & b \\ c & d\end{smallmatrix}\right))\] where $\phi \in \Aut(\KK)$, $\lambda \in \KK^*$ and $\left(\begin{smallmatrix} a & b \\ c & d\end{smallmatrix}\right)\in \GL_2(\KK)$. The reader will have to remember that this is an element of a semidirect product of $\Aut(\KK)$ and $D$.

Consider $A$ as an $\FF_pD^\dagger$-module and observe that the action of $D^\dagger$ on $A$ maps polynomials of degree $k$ to polynomials of degree $k$. For $n \ge 0$, let $V_n(q)$ be the $\FF_p D^\dagger$-submodule  of $A$ consisting of the homogeneous polynomials of degree $n$. We have $\dim_\KK V_n(q)= n+1$ and provided that $1 \leq n \leq p-1$ these submodules are irreducible self-dual $\FF_p D^\dagger$-modules (see \cite[p588]{BrauerNesbitt}).  We introduce notation to describe particular subgroups of $D^*$. 

\begin{notation}\label{n:subsofd*}
Define \[U = \{\left(1, 1,\left(\begin{smallmatrix} 1 & 0 \\c& 1 \end{smallmatrix} \right)\right)\mid c\in \KK \},\]
\[B= \{\left(\phi,\lambda,\left(\begin{smallmatrix} a & 0\\c& d \end{smallmatrix}  \right)\right)\mid a,d,\lambda\in \KK ^*, c \in \KK,\phi\in O^p(\Aut(\KK)) \}, \mbox{ and }\]
\[\Sigma=\left\{(\phi,\lambda, \begin{pmatrix} a&0\\0&d \end{pmatrix})\mid a,d,\lambda \in \KK^*, \phi \in O^p(\Aut(\KK))\right\}.\]
\end{notation}
Then $U\in \Syl_p(D^{*})$, $U$ is elementary abelian of order $q$, and $B=N_{D^*}(U)$ has order $q(q-1)^3m_{p'}$ where $m_{p'}=|O^p(\Aut(\KK))|$ is the $p'$-part of $m$ for $q=p^m$, and $B=U\Sigma$. The following groups will be of primary importance: 
 
\begin{notation}\label{n:snandpn}
Fix $n \in \mathbb{N}$. With the action of $D$ and $D^*$ as described above, define \[P_n(q)=D\ltimes V_n(q), \hspace{2mm} P_n^*(q)= D^*\ltimes V_n(q ), \hspace{2mm} \mbox{and} \hspace{2mm} S_n(q)= U\ltimes V_n(q).\]
\end{notation}

Thus $S_n(q)$ is a Sylow $p$-subgroup of $P_n(q)$ and $P_n^*(q)$. We identify $D$, $B$, $\Sigma$, $U$ and $V$ as subgroups of $P_n(q)^*$. This means that we use multiplicative notation when considering $V$ as a $p$-subgroup of $S$, but additive notation when calculating inside $V$ as a vector space. Hence we have $S_n(q)= UV_n(q)$ and $O^{p'}(P_n^*(q)/V_n(q))=O^{p'}(P_n(q)/V_n(q))=O^{p'}(P_n(q))/V_n(q)\cong \SL_2(\KK)$.

We will be concerned with $V_n(q)$ only when $n\leq p$. As intimated above, $V_n(q)$ is irreducible when $n\leq p-1$, but it will be informative to understand the composition factors of $V_p(q)$. Throughout, a subscript $\KK$ indicates that we must take the $\KK$-subspace.

\begin{lemma}\label{Vpqstruct}
Regard $V=V_p(q)$ as a $\KK \SL_2(\KK)$-module. The $\KK$-linear map
\begin{eqnarray*}
\psi: V_p(q) &\rightarrow &V_{p-2}(q)\\ x^iy^j &\mapsto& ix^{i-1}y^{j-1}
\end{eqnarray*}
is a surjective $\KK\SL_2(\KK)$-module homomorphism with kernel $\langle x^p,y^p\rangle_{\KK}$.
\end{lemma}
\begin{proof}
Set $X=\SL_2(\KK)$ and observe that

\[X=\left\langle \left( \begin{smallmatrix}
1 & 0 \\ 1 & 1
\end{smallmatrix} \right), \left(\begin{smallmatrix}
0 & \theta \\ -\theta^{-1} & 0
\end{smallmatrix} \right) \mid \theta \in \mathbb\KK^*\right\rangle.\]

For $i+j=p$, we have \[i=-j \pmod p,\] \[ix^{i-1}y^{j-1}=-jx^{i-1}y^{j-1} \pmod p,\,\,\text{and}\] \[i  \binom{j-1}{k} = (k+i) \binom{j}{k} \pmod p.\]
Hence
\begin{eqnarray*}
 (x^iy^j \cdot \left( \begin{smallmatrix}
1 & 0 \\ 1 & 1
\end{smallmatrix} \right))\psi &=& (x^i(x+y)^j)\psi= \left(\sum_{k=0}^j \binom{j}{k} x^{k+i}y^{j-k}\right)\psi \\&=& \sum_{k=0}^{j-1} \binom{j}{k} (k+i)x^{k+i-1}y^{j-k-1}= x^{i-1}\sum_{k=0}^{j-1} (k+i)\binom{j}{k} x^ky^{j-k-1}\\&=&
i x^{i-1}\sum_{k=0}^{j-1} \binom{j-1}{k} x^{k}y^{j-1-k}=
 ix^{i-1}(x+y)^{j-1} =(x^iy^j)\psi\cdot\left(\begin{smallmatrix}
1 & 0 \\ 1 & 1
\end{smallmatrix} \right).
\end{eqnarray*}

Similarly, for $\theta \in \KK^*$,
\begin{eqnarray*} ( x^iy^j \cdot \left(\begin{smallmatrix}
0 & \theta \\ -\theta^{-1} & 0
\end{smallmatrix} \right))\psi & = & ((-\theta^{-1}x)^j(\theta y)^i)\psi= (x^j y^i \theta^{i-j} (-1)^j)\psi \\ &= &j x^{j-1}y^{i-1} (-1)^j \theta^{i-j}= i(\theta y)^{i-1} (-\theta^{-1} x)^{j-1}\\&=& ix^{i-1}y^{j-1} \cdot \left(\begin{smallmatrix}
0 & \theta \\ -\theta^{-1} & 0
\end{smallmatrix} \right)= (x^iy^i)\psi\cdot\left(\begin{smallmatrix}
0 & \theta \\ -\theta^{-1} & 0
\end{smallmatrix} \right).
\end{eqnarray*}
Therefore $\psi$ is $X$-linear and clearly surjective. It follows that $\ker \psi= \langle x^p,y^p\rangle_{\KK}$ and this proves the result.
\end{proof}

Next, consider $V_p(q)$ over $\KK$ and let $\Lambda(q)=\Hom_\KK(V_p(q),\KK)$ be the $\KK$-dual of $V_p(q)$. Fix the standard $\KK$-basis \[\B=\{x^p, x^{p-1}y, x^{p-2}y^2,\ldots, xy^{p-1}, y^p\}\] for $V_p(q)$ and for $v,w \in \B$ write $\overline{v} \in \Hom_\KK(V_p,\KK)$ for the linear extension of the map \[(w)\overline{v}=\begin{cases} 1 & \mbox{ if  $w=v$} \\ 0 & \mbox{ otherwise}  \end{cases}.\]

Set $\overline{\B}=\{\overline{v} \mid v \in \B\}$. Then $\overline{\B}$ is a basis for $\Lambda(q)$. As usual, the action of $D^*$ on $\Lambda(q)$ is given by \[(w)(\eta g) = (wg^{-1})\eta\] where $g\in D^*$, $w \in V_n(q)$ and $\eta \in \Lambda(q)$.

\begin{notation}\label{n:slandpl}
Define: \[P_\Lambda^*(q)= D^*\ltimes \Lambda \mbox{ and } S_{\Lambda}(q):= U\ltimes \Lambda(q).\]
\end{notation}
Thus $S_{\Lambda}(q)$ is a Sylow $p$-subgroup of $P_\Lambda^*(q)$. Moreover, from now on, we write $\Lambda= \Lambda(q)$ and identify $\Lambda$ as a subgroup of $P_\Lambda^*(q)$.  

The $p$-groups $S_n(q)$, $n \ge 1$ and $S_\Lambda(q)$ are collectively called \emph{polynomial $p$-groups.}

Let $\rho_{V}$ and $\rho_\Lambda$ be the corresponding $\KK  D^*$-representations for $V_p(q)$ and $\Lambda(q)$ respectively. Then $\rho_{V}$ and $\rho_\Lambda$ are related via $\rho_V(g)=(\rho_\Lambda(g)^{-1})^T$ for $g\in D^*$ where we assume that $\rho_V(g)$ and $\rho_\Lambda(g)$ act (on row vectors) on the right with respect to the ordered bases $\B$ and $\overline{\B}$. We have: \[\rho_V((1, \theta,\diag(\lambda,1)))=\theta\diag(\lambda^p,\lambda^{p-1}\,\ldots, \ldots, \lambda,1)\] and \[\rho_V((1, \theta,\diag(\lambda,\lambda^{-1})))=\theta\diag(\lambda^p,\lambda^{p-2}\,\ldots, \lambda,\lambda^{-1},\ldots, \lambda^{-p})\] while the last diagonal $3 \times 3$ submatrix of $\rho_V\left(1, \theta,\left(\begin{matrix}1 & 0 \\ \lambda & 1 \end{matrix}\right)\right)$ is \[\theta\left(\begin{matrix} 1 & 0 &0 \\ -\lambda & 1 & 0 \\ 0 & 0 & 1 \end{matrix}\right).\] Thus, \[\rho_\Lambda((1, \theta,\diag(\lambda,1)))=\theta^{-1}\diag(\lambda^{-p},\lambda^{1-p}\,\ldots, \ldots, \lambda^{-1},1), \]
\[\rho_\Lambda((1, \theta,\diag(\lambda,\lambda^{-1})))=\theta^{-1}\diag(\lambda^{-p},\lambda^{2-p}\,\ldots, \lambda^{-1},\lambda,\ldots, \lambda^{p})\]

and the last diagonal $3 \times 3$ submatrix of $\rho_\Lambda\left(1, \theta,\left(\begin{matrix} 1 & 0 \\ \lambda & 1 \end{matrix}\right)\right)$ is

\[\theta^{-1}\left(\begin{matrix} 1 & \lambda &0 \\ 0 & 1 & 0 \\ 0 & 0 & 1 \end{matrix}\right).\]

Using the above information, via a simple change of variables,  the space spanned by
$\mathcal W=\{\overline{x^2y^{p-2}},\overline{xy^{p-1}},\overline{y^p} \}$  is invariant under $\left(1, \theta,\left(\begin{matrix}
a & 0 \\ c & b
\end{matrix}\right)\right)$ and, with respect to $\mathcal W$,
\[\rho_\Lambda\left(1, \theta,\left(\begin{matrix}
a & 0 \\ c & b
\end{matrix}\right)\right)=\theta^{-1}\left(\begin{matrix} a^{-2}b^{2-p} & ca^{-2}b^{1-p} & 0 \\ 0 & a^{-1}b^{1-p} & 0 \\ 0 & 0 & b^{-p} \end{matrix}\right).\]

We present some results which are modest extensions of lemmas from \cite{ClellandParker2010}. They will be given without proof.

\begin{lemma}\label{Gammacentraliser}
For $1\leq n\leq p$, the set of elements of $D^*$ which act trivially on $V=V_n(q)$ is given by
\[ C_{D^*}(V) =C_{D}(V) = \left\{ \left(1,a^{-n},
\begin{pmatrix}
a & 0 \\  0 & a
\end{pmatrix}
\right) \mid  a  \in \KK^{*}   \right\}\]
and has order $ q-1 $. In addition,  $C_{D^*}(\Lambda(q))= C_{D^*}(V_p(q))$.
\end{lemma}
\begin{proof}
See \cite[Lemma 4.1]{ClellandParker2010}. The last statement follows by duality.
\end{proof}

We pause here to record the following technical result which will be needed in Section \ref{sec:tsfs}.

\begin{lemma}\label{l:centslem}
If $X:=\SL_2(q)$ with $X \le \Aut(V_n(q))$ and $T\in\syl_p(X)$ then $|C_X(C_{V_n(q)}(T))|=q.(n, q-1)$.
\end{lemma}
\begin{proof}
Observe that $C_X(C_{V_n(q)}(T))\le N_X(T)$ and $|N_X(T)|=q.(q-1)$. Clearly, $T$ centralises $C_{V_n(q)}(T)$. Let $K$ be a Hall $p'$-subgroup of $N_X(T)$ so that $K$ is cyclic of order $q-1$ and let $k$ be a generator for $K$. Then we can arrange that $k$ acts as $\lambda^n$ on $C_{V_n(q)}(T)$ for some $\lambda\in \KK$ of multiplicative order $q-1$. Thus every element of $K$ acts as $\mu^n$ on $C_{V_n(q)}(T)$ for some distinct $\mu\in \KK^*$. Hence, the elements of $K$ which centralise $C_{V_n(q)}(T)$ are in bijection with the $n$th roots of unity in $\KK$, which have order $(n, q-1)$ and the claim holds.
\end{proof}

For $n\leq p-1$ and $0 \neq f = \displaystyle\sum_{i=0}^{n} a_i x^{n-i}y^i \in V$ define $\mathrm{wt}(f) = \mathrm{max}\{ i \mid a_i \neq 0 \}$ to be the \textit{weight} of $f$. Also define $\mathrm{wt}(0) = -1$.  For $-1 \leq i \leq n$, set $ C_i = \{ f \in V \mid \mathrm{wt}(f) \leq i \}.$ For each $-1\leq i \leq n$,  $C_i$ is a subgroup of $V_n(q)$ of order $q^{i+1}$. In particular, $C_{-1}=1$ and $C_n =V$.

\begin{lemma} \label{Somnibus}
Suppose that $1\leq n\leq p-1$ and $S=S_n(q)$ for $q=p^m$.
\begin{enumerate}
\item For $0 \leq i \leq n$, $[C_i,S] = C_{i-1}$. In particular, for $i \leq n-1$, $C_i$ is the $(n-i)$th term of the lower central series of $S$.
\item The upper and lower central series of $S$ consist of the same subgroups. In particular, $C_0= C_V(S)$.
\item If $z\in S\setminus V$, then $C_V(z)= C_V(S)$ and $C_{V/C_0}(S)= C_{V/C_0}(z)=C_1/C_0$.
\item If $n \ge 2$, then $V$ is the unique abelian subgroup of maximal order in $S$. In particular, $V$ is characteristic in $S$.
\item Both $V$ and $[V,S]U$ are normalised by $B$.
\end{enumerate}
\end{lemma}
\begin{proof}
See \cite[Lemma 4.2]{ClellandParker2010}.
\end{proof}

To aid in understanding the automorphism group of $S_n(q)$, it will often be enough for our purposes to understand the automorphism group of $S_1(q)$.

\begin{proposition}\label{prop: L3Q}
Let $S:=S_1(q)$ for $q=p^m$ and $p$ odd. Set $C:=C_{\Aut(S)}(S/\Phi(S))$. Then $C$ is a normal $p$-subgroup of $\Aut(S)$ and there is $H\le \Aut(S)$ with $\Aut(S)=CH$ and $H\cong \Gamma \textrm{L}_2(q)$.

In particular, if $X\le \Aut(S)$ is such that $X$ normalises an elementary abelian subgroup of $S$ of order $q^2$, then $X$ is solvable and $|XC/C|\divides q.(q-1)^2.m_{p'}$,  where $m_{p'}$ denotes the $p'$-part of $m$.
\end{proposition}
\begin{proof}
This follows from \cite[Proposition 5.3]{ParkerRowleyWeakBN}.
\end{proof}

In the next lemma we exploit the following general fact about $\KK G$-modules. Suppose that $g\in G$. Then the commutator map delivers a $\KK C_G(g)$-module homomorphism $V \rightarrow [V,g]$ with kernel $C_V(g)$.

\begin{lemma}\label{Snauto}
Let $2\leq n \leq p-1$ and $q=p^m$. The following hold:
\begin{enumerate}
\item \label{sn1} $S_1(q)$ is isomorphic to a Sylow $p$-subgroup of $\SL_3(q)$ and $S_2(q)$ is isomorphic to a Sylow $p$-subgroup of $\Sp_4(q)$;
\item \label{sn2} $S'=\Phi(S)=[V,S]$ has index $q^2$ in $S=S_n(q)$;
\item \label{sn3} for $2\leq i\leq n$ and $S=S_n(q)$ we have that $S/Z_{n+1-i}(S)=S/[V,S; i] \cong S_{i-1}(q)$;
\item \label{sn4} $\Aut(S_n(q))$ is solvable and has a Hall $p'$-subgroup of order $(q-1)^2.m_{p'}$, where $m_{p'}$ denotes the $p'$-part of $m$.
\end{enumerate}
Furthermore, $\Aut_{P_n^*(q)}(S_n(q))$ contains a Hall $p'$-subgroup of $\Aut(S_n(q))$.
\end{lemma}
\begin{proof}
Part (\ref{sn1}) follows from \cite[Lemma 4.6]{ClellandParker2010}. Part (\ref{sn2}) follows from Lemma \ref{Somnibus} (1).

Define the map $\gamma: S_n(q)\to S_{n-1}(q)$ such that \[((1,1,\begin{pmatrix} 1&0\\\lambda&1\end{pmatrix}), \sum_{i=0}^n\lambda_ix^{n-i}y^i)\mapsto ((1,1,\begin{pmatrix} 1&0\\\lambda&1\end{pmatrix}), \sum_{i=1}^n i\lambda_ix^{n-i}y^{i-1}).\]

That is, $\gamma$ preserves $U$ and acts as differentiation with respect to $y$ on $V_n(q)$, using that $x^n\gamma=0$. Setting $\kappa_i=(i+1)\lambda_{i+1}$ we have that \[\sum_{i=1}^n i\lambda_ix^{n-i}y^{i-1}=\sum_{i=0}^{n-1} (i+1)\lambda_{i+1}x^{(n-1)-i}y^i=\sum_{i=0}^{n-1} \kappa_ix^{(n-1)-i}y^i\] and so clearly $\gamma$ maps $S_n(q)\to S_{n-1}(q)$. We record the identity
\begin{eqnarray*}
&((1,1,\left(\begin{smallmatrix} 1&0\\\lambda&1\end{smallmatrix}\right)), \displaystyle\sum_{i=0}^n\lambda_ix^{n-i}y^i)\cdot((1,1,\left(\begin{smallmatrix} 1&0\\\mu&1\end{smallmatrix}\right)),
\sum_{i=0}^n\mu_ix^{n-i}y^i)\\
=&((1, 1,\left(\begin{smallmatrix} 1&0\\\lambda+\mu&1\end{smallmatrix}\right)),\displaystyle\sum_{i=0}^n\lambda_ix^{n-i}(\mu x+y)^i + \mu_ix^{n-i}y^i)\\
=&((1, 1,\left(\begin{smallmatrix} 1&0\\\lambda+\mu&1\end{smallmatrix}\right)),\displaystyle\sum_{i=0}^n\lambda_ix^{n-i}\left(\sum_{j=0}^i\binom{i}{j}(\mu x)^jy^{i-j}\right) + \mu_ix^{n-i}y^i)\\
\end{eqnarray*} and leave the verification that $\gamma$ is a homomorphism to the reader. Assume that \begin{align*}
&\left((1,1,\begin{pmatrix} 1&0\\\lambda&1\end{pmatrix}), \sum_{i=0}^n\lambda_ix^{n-i}y^i\right)\gamma\\
=&\left((1,1,\begin{pmatrix} 1&0\\\lambda&1\end{pmatrix}), \sum_{i=1}^n i\lambda_ix^{n-i}y^{i-1}\right)\\
=&\left((1,1,\begin{pmatrix} 1&0\\0&1\end{pmatrix}), 0\right).
\end{align*} Since $n<p$, the only possible solution is that $\lambda=\lambda_i=0$ for all $i\in \{1,\dots, n\}$. Hence, elements of $\ker \gamma$ have the form $((1,1,\begin{pmatrix} 1&0\\0&1\end{pmatrix}), \lambda_0 x^n)$ for $\lambda_0\in \KK$, and so $\ker \gamma =C_{V_n(q)}(S_n(q))=Z(S_n(q))$. A consideration of orders yields that $\gamma$ is surjective and so $S_n(q)/Z(S_n(q))\cong S_{n-1}(q)$. Then \[S_n(q)/Z_i(S_n(q))\cong (S_n(q)/Z(S_n(q)))/(Z_{i-1}(S_n(q)/Z(S_n(q))))\cong S_{n-1}(q)/Z_{i-1}(S_{n-1}(q))\] and by induction we see that $S_n(q)/Z_i(S_n(q))\cong S_{n-i}(q)$, as desired. This proves (\ref{sn3}).

Let $S=S_n(q)$. By Lemma \ref{Somnibus} (4), $V=V_n(q)$ is a characteristic subgroup of $S$ and therefore so is $[V,S,S]$. We observe by (2) that $S/[V, S, S]\cong S_1(q)$, which itself is isomorphic to a Sylow $p$-subgroup of $\SL_3(q)$ by (\ref{sn1}). Notice that $C_{\Aut(S)}(S/[V, S, S])\le C_{\Aut(S)}(S/\Phi(S))$ is a normal $p$-subgroup of $\Aut(S)$ and that $\Aut(S)/C_{\Aut(S)}(S/[V, S, S])$ acts faithfully on $S/[V, S, S]$. Since $\Aut(S)$ leaves $V$ invariant by Lemma \ref{Somnibus} (4), we see that the action of $\Aut(S)$ on $S/[V, S, S]$ preserves an elementary abelian subgroup of order $q^2$ in $S/[V, S, S]$. Hence $\Aut(S)$ is solvable by Proposition \ref{prop: L3Q}. Moreover, the Hall $p'$-subgroups of $\Aut(S)$ have order dividing $(q-1)^2.m_{p'}$, where $m_{p'}$ denotes the $p'$-part of $m$.

We have $\Aut_{P_n^*(q)}(S)=\Aut_{BV}(S)\cong BV/Z(S)C_{D^*}(V)$ which has order $q^{n+1}(q-1)^2.m_{p'}$ by Lemma \ref{Gammacentraliser}. Thus the upper bound of $(q-1)^2.m_{p'}$ for the order of a Hall $p'$-subgroup of $\Aut(S)$ is attained by a subgroup of $\Aut_{P_n^*(q)}(S)$. This proves (\ref{sn4}). The final statement follows from Lemma \ref{Gammacentraliser} and the fact that $|N_{D^*}(S_n(q))/S_n(q)|=  (q-1)^3.m_{p'}$.
\end{proof}

\begin{lemma}\label{action on centre}
Suppose that  $2\leq n \leq  p-1$ and set $S=S_n(q)$, $V= V_n(q)$ and $L=\Aut_{P_n^*(q)}(S_n(q))$. Let $L_0=C_L(V /[ V,S])$. Then $L_0$ acts  and  irreducibly on $C_V(S)$ considered as a $\FF_pL_0$-module. Furthermore, the action is trivial if and only if $n=p-1= q-1$.
\end{lemma}
\begin{proof} We calculate that \[C_{D^*}(V/[V, S])=\left\{ \left(1,d^{-n},
\begin{pmatrix}
a & 0  \\ c & d
\end{pmatrix}
\right) \mid  a, d \in \KK^{*}, c\in \KK \right\}.\]
Hence Lemma \ref{Gammacentraliser}, implies
\[L_0=C_L(V /[ V,S])=\left\{c_z
 \mid z=  \left(1,1,
\begin{pmatrix}
ad^{-1} & 0 \\ c & 1
\end{pmatrix}
\right)  \in C_{D^*}(V/[V, S])\right\}.\]
By Lemma \ref{Somnibus}(2), $C_V(S)=\langle x^n \rangle_\KK$ and, writing $Z=C_V(S)$, we see $c_z: Z\rightarrow Z$ is multiplication by $\lambda^n$ for $\lambda \in \KK^*$. Hence $L_0$ induces a group of order $(q-1)/\gcd(q-1,n)$ on $Z$. It follows that $Z$   is irreducible  as a $\FF_p L_0$-module and is the trivial module if and only if $n=p-1=q-1$.
\end{proof}

We saw one way that the collection of groups $S_n(q)$ is closed under quotients in Lemma \ref{Snauto}(3). We now prove that this family of groups is closed upon taking certain subgroups.

\begin{lemma}\label{subgroup similarity}
Suppose that $S=S_n(q)$ with $q=p^m$ and $1 \leq n \leq p-1$, and $V=V_n(q) \le S$. Then, for $2\leq i\leq n-2$, we have $Z_{n-i+1}(S) =[V,S;i]\cong V_{n-i}(q)$ as an $\FF_pU$-module. In particular, $[V,S;i]U \cong S_{n-i}(q)$.
\end{lemma}
\begin{proof}
Let $z \in U^\#$. Then using Lemma \ref{Somnibus}(1) and (3), $C_V(z)= Z(S)$ and, as $|V/[V,z]|=q$, we deduce that $[V,S]=[V,z]$. By Lemma \ref{Snauto} (3), $S/Z(S) \cong S_{n-1}(q)$. Thus \[V/C_V(S)\cong V_{n-1}(q)\] as $\FF_pU$-modules.

Now, recall that
 $V/C_V(z) \cong [V,z]$ as $\KK C_D(z)$-modules and thus as $\FF_pU$-modules. Hence
 \[[V,S]=[V,z]\cong V/C_V(z)=V/C_V(S) \cong  V_{n-1}(q)\]
as $\FF_pU$-modules. It follows that $[V,S]U \cong S_{n-1}(q)$. Now iterate this argument to obtain the  result.
\end{proof}

One particular case of Lemma \ref{subgroup similarity} is of particular importance to us and so we single it out.

\begin{corollary}\label{prop: Q iso}
Let $S=S_n(q)$ for $2\leq n\leq p-1$. Then $UZ_2(S)\cong S_1(q)$.
\end{corollary}
\begin{proof}
By Lemma \ref{Somnibus}(2), we have $Z_2(S)= [V,S;n-2]$ and so the result follows from Lemma \ref{subgroup similarity}.
\end{proof}

\begin{lemma}\label{lem: com full}
Let $S=S_n(q)$ and $V= V_n(q)$ with $1\leq n\leq p-1$. Put $Z_i= Z_i(S)$ for $1\leq i \leq n$ and $Z_{n+1}=V $. Then  for $1\leq i\leq n$ for $z \in Z_{i+1} \setminus Z_i $, we have $[z,S]= Z_{i} $.
\end{lemma}
\begin{proof} Suppose first that $n=1$. Then $z \in V\setminus Z_1$ and the result follows from an elementary calculation.  Now suppose that $n > 1$. Set $T= Z_{i+1}U$. Then $T \cong S_i(q)$ by Lemma \ref{subgroup similarity}. If $i+1<n$, the result then follows by induction. Thus we may assume that $T=S$ and $z\in V\setminus Z_n$. By induction, and using Lemma \ref{Snauto} (\ref{sn3}), $Z_{n}= [z,S]Z_1$. Notice that $Z_2 \le [z,S]Z_1$ and so there exists  $y \in ([z,S] \cap Z_2)\setminus Z_1$. Since $[z,S]$ is normalised by $U$, $[z,S]\ge [y,U] = [y,Z_2U] = Z_1$ by Lemma \ref{subgroup similarity} and induction as $Z_2U \cong S_1(q)$. Hence  $Z_{n}= [z,S]Z_1=[z,S]$ and this proves the claim.
\end{proof}

By analogy with Lemma \ref{Somnibus} and Lemma \ref{Snauto}, we provide the following structural results for $S_{p}(q)$ and $S_{\Lambda}(q)$.

\begin{lemma}\label{(p+1)CVS}
Let $S=S_p(q)$ and $V=V_p(q)$. Set $W=\langle x^p,y^p\rangle_\KK$  and let $z\in S\setminus V$. Then
\begin{enumerate}
\item \label{cvs1} $C_W(S)=C_V(S)=\langle x^p\rangle_{\KK}$ has order $q$ if $q>p$, and otherwise $C_V(S)= \langle x^p, x^{p-1}y-y^p\rangle_{\KK}$ has order $p^2$;
\item \label{cvs2} $Z_2(S)=\langle x^p, y^p, x^{p-1}y\rangle_{\KK}$ has order $q^3$, $Z_p(S)=S$ and for $2<i<p$ we have that $|Z_i(S)/Z_{i-1}(S)|=q$;
\item \label{cvs3} $[V, z]=[V,S]= \langle x^iy^{p-i}\mid 2\leq i\leq p\rangle_{\KK}$, and in particular, $|V/[V,S]|=q^2$;
\item \label{cvs4} $C_W(z)=C_{[V, S]}(z)$ has order $q$ and $C_V(z)$ has order $q^2$;
\item \label{cvs5} $Z_2(S)/C_V(z)=WC_V(z)/C_V(z)=C_{V/C_V(z)}(z)$ and for $2< i\leq p$, $Z_i(S)/Z_{i-1}(S)=C_{S/Z_{i-1}(S)}(z)$;
\item \label{cvs6} for $1\leq i <p$ we have that $[V, z; i]=[V, S; i]$, $[V, S; p]=1$ and $|[V, S; i]/[V, S; i+1]|=q$; and
\item \label{cvs7} $V$ is the unique abelian subgroup of maximal order in $S$. In particular, $V$ is characteristic in $S$.
\end{enumerate}
\end{lemma}
\begin{proof}
By Lemma \ref{Vpqstruct}, $V/W \cong V_{p-2}(q)$ as $\KK \SL_2(q)$-modules. Let $g_\lambda\in S\setminus V$ be represented by the element $\left(1,1 \begin{pmatrix} 1&0\\\lambda&1\end{pmatrix}\right)\in \Aut_S(V)$, for $\lambda \in \mathbb K^*$. We know
\[C_V(S)W/W \le C_{V/W}(S)= \langle x^{p-1}y\rangle_{\KK} W/W.\]
Since $C_V(S) \cap W= C_W(S)=\langle x^p\rangle_{\KK}$, if $\dim_\KK C_V(S) \ne 1$, there exists $\gamma\in \mathbb K$ such that $x^{p-1}y+ \gamma y^p\in C_V(S)$. Thus
\begin{eqnarray*}
0=[x^{p-1}y+ \gamma y^p,g_\lambda]&=& x^{p-1}(\lambda x+ y)+ \gamma(\lambda^px^p+y^p)- (x^{p-1}y+ \gamma y^p)\\
&=& \lambda x^{p}+ \gamma\lambda^px^p= \lambda (1+\gamma \lambda^{p-1})x^p,
\end{eqnarray*}
from which we see that $\lambda^{p-1}=-\gamma^{-1}$ which only has $p-1$ solutions in $\mathbb K$. Hence $\dim_{\mathbb K} C_V(S)\ne 1$ if and only if $\mathbb K$ has $p$ elements. Applying a similar reasoning to before, when $q=p$ we witness that $x^{p-1}y-y^p\in C_V(S)$. This proves (\ref{cvs1}).

It is clear from the action of $S$ that $\langle x^p, y^p, x^{p-1}y\rangle_{\KK}\le Z_2(S)$. Moreover, from the structure of $V/W$ we must have that $Z_2(S)/W\le \langle x^{p-1}y, x^{p-2}y^2\rangle_{\KK} W/W$. Since $Z_2(S)$ is a $\mathbb K$-space, to show that $Z_2(S)=\langle x^p, y^p, x^{p-1}y\rangle_{\KK}$ we need only verify that $x^{p-2}y^2\not\in Z_2(S)$. As before, recognise $g_\lambda$ with $\left(1,1 \begin{pmatrix} 1&0\\\lambda&1\end{pmatrix}\right)$, and then $[x^{p-2}y^2, g_\lambda]=\lambda^2x^p+2\lambda x^{p-1}y$ which is not contained in $C_V(S)$ for any $q$ by (\ref{cvs1}). Hence, $Z_2(S)$ is as described. Since $Z_2(S)/W=C_{V/W}(S)$, the observation that $Z_p(S)=S$ and $|Z_i(S)/Z_{i-1}(S)|=q$ whenever $2<i<p$ follow from the structure of $V/W$ as witnessed in Lemma \ref{Somnibus}. This proves (\ref{cvs2}).

For $z\in S\setminus V$, commutation by $z$ induces a homomorphism from $V$ to $[V, z]$ with kernel $C_V(z)$. By (\ref{cvs1}), we see that $|C_V(z)|=|V/[V, z]|\leq q^2$. We claim that $[V, S]=[V, z]$ has index $q^2$ in $V$. Observe that $[V, S]$ is generated by $\{[x^iy^j,g_\lambda]\mid \lambda \in \mathbb K\}$. We will show that $[V,S]=\langle x^iy^{p-i}\mid i \ge 2\rangle_{\KK}$, which has dimension $p-1$. First note that $[y^p,g_\lambda] = \lambda^p x^p \in [V,S]$.  Then, for $1\leq k \leq p$,
\begin{eqnarray*}
[x^ky^{p-k},g_\lambda]&=& x^k(\lambda x+y)^{p-k}-x^ky^{p-k}= \sum_{j=1}^{p-k} \binom{p-k}{j}\lambda^j x^{j+k}y^{p-k-j}\in [V,S]
\end{eqnarray*}
and the claim easily follows. Thus, (\ref{cvs3}) holds and so $|C_V(z)|=q^2$. Since $W$ is a faithful $2$-dimensional $\mathbb K\SL_2(q)$-submodule of $V$, we must have that $|C_W(z)|=q$. A computation as in (\ref{cvs1}) shows that $\langle x^{p-1}y\rangle_{\KK} \cap C_V(z)=1$ and so $C_{[V, z]}(z)=[V, z]\cap C_V(z)=\langle x^p\rangle_{\KK}=C_W(z)$ and (\ref{cvs4}) holds.

We have that $Z_2(S)/C_V(z)=WC_V(z)/C_V(z)\le C_{V/C_V(z)}(z)$. As in the proof of (\ref{cvs2}), using the structure of $V/W$, to prove that $WC_V(z)/C_V(z)= C_{V/C_V(z)}(z)$ it suffices to show that $[x^{p-2}y^2, z]\not\le C_V(z)$. This is the same calculation as in the proof of (\ref{cvs2}), which yields the result. Since $W\le Z_2(S)$, the remainder of (\ref{cvs5}) follows upon applying the results in Lemma \ref{Somnibus} to $V/W$.

We have that $[V, z]=[V, S]$ and $|C_{[V, z]}(z)|=q$. Then, for $2\leq i\leq p$, commutation by $z$ yields a homomorphism from $[V, z; i]$ to $[V, z; i+1]$ from which we deduce that $|[V, z; i]/[V, z; i+1]|=q$. Since $S$ acts nilpotently on $V$, and $[V, S]$ is a $\mathbb K$ space, we ascertain that $[V, z; i]=[V, S; i]$ for $1\leq i\leq p$ and so (\ref{cvs6}) holds.

We note that $|V|=q^{p+1}$ and that if $A\le S$ is abelian with $A\not\le V$, then $|A|\leq |S/V||C_V(A)|\leq q^3$ by (\ref{cvs4}). Since $p$ is odd, $|V|>q^3$ and so $V$ is as described in (\ref{cvs7}).
\end{proof}

In the next lemma, we examine the action of $S_\Lambda(q)$ on $\Lambda(q)$. By Lemma \ref{Vpqstruct}, $\langle x^p,y^p\rangle_\KK$ is a $\KK \SL_2(q)$-submodule of  $V_p(q)$. In $\Lambda(q)$, this corresponds to the subset of all linear functionals with $\gen{x^p,y^p}_{\KK}$ in the kernel. Thus $\Lambda(q)$ has a submodule $W$ of codimension $2$ and $W=\gen{\ov{x^{p-1}y},\dots,\ov{xy^{p-1}}}_{\KK}$.

\begin{lemma}\label{(p+1)CUpS}
Let $S=S_{\Lambda}(q)$ and $V=\Lambda(q)$. Set $W$ to be the $(p-1)$-dimensional $\KK D$-submodule of $V$, when $V$ is regarded as a $\KK D$-module, and let $z\in S\setminus V$. Then
\begin{enumerate}
\item \label{cups1} $C_W(S)=C_W(z)=\gen{\ov{xy^{p-1}}}_{\KK}$ has order $q$ and $Z(S)=C_V(S)=C_V(z)=\gen{\ov{xy^{p-1}}, \ov{y^p}}_{\KK}$ has order $q^2$;
\item \label{cups2} $Z_2(S)$ has order $q^3$, $Z_{p}(S)=S$ and for $2\leq i\leq p-1$ we have that $|Z_i(S)/Z_{i-1}(S)|=q$;
\item \label{cups3} if $q>p$ then we have that $|V/[V, S]|=q$, $|V/[V, z]|=q^2$, $[V, z]\cap W=[W, z]$ and $[V, S]=WC_V(S)=[V, z]C_V(S)=[V, z]W$;
\item \label{cups4} if $q=p$ then $[V, S]=[V, z]$ has index $p^2$ in $V$ and $[V, z]\cap W=[W, z]$;
\item \label{cups5} $C_W(S)=C_{[V, z]}(z)$ has order $q$;
\item \label{cups6} $Z_2(S)/C_V(z)=C_{WC_V(z)/C_V(z)}(S)=C_{V/C_V(z)}(z)$ and for $2\leq i\leq p$, we have that $Z_i(S)/Z_{i-1}(S)=C_{S/Z_{i-1}(S)}(z)$;
\item \label{cups7} for $2\leq i <p$ we have that $[V, z; i]=[V, S; i]$, $[V, S; p]=1$ and $|[V, S; i]/[V, S; i+1]|=q$; and
\item \label{cups8} $V$ is the unique abelian subgroup of maximal order in $S$. In particular, $V$ is characteristic in $S$.
\end{enumerate}
\end{lemma}
\begin{proof}
This proof follows from Lemma \ref{(p+1)CVS} and dualising actions (see \cite[VII, Lemma 8.3]{HuppertBlackburnVolume2}).
\end{proof}

\begin{lemma}\label{lem: CharSub}
Let $S=S_{\Lambda}(q)$ where $q=p^m$ and $m>1$. Then $V=\Lambda(q)$ and $U[V, S]$ are the unique normal subgroups of $S$ of index $q$ which have exponent $p$. In particular, $U[V, S]$ is characteristic in $S$.
\end{lemma}
\begin{proof}
Observe that $V$ is characteristic in $S$ by Lemma \ref{(p+1)CUpS}. Let $W\le V$ be such that $|V/W|=q^2$ and $W\normaleq P_\Lambda^*(q)$. Thus, $|V/[V, S]|=|[V, S]/W|=q$. Let $A\normaleq S$ be such that $|S/A|=q$, $V\ne A$ and $A$ has exponent $p$.

Let $a\in A\setminus (A\cap V)$. By \cite[Theorem 9.7]{Huppert}, for all $v\in A\cap V$ and $a\in A$ we have that $[v, a; p-1]\in [V, S; p]=1$. By Lemma \ref{(p+1)CUpS}, we have that $[[V, S], a; p-1]=1$ so that $[[V, S](A\cap V), a;p-1]=1$ and again considering Lemma \ref{(p+1)CUpS}, we see that $A\cap V\le [V, S]$. Since $|S/A|=q$ we must have that $[V, S]=A\cap V$, $A$ has nilpotency class $p-1$ and $S=AV$. Another application of \cite[Theorem 9.7]{Huppert} reveals that $S$ has exponent $p^2$.

Now, $A/W$ is elementary abelian of order $q^2$, as are both $U[V, S]/W$ and $V/W$. Since $S/W\cong S_1(q)$, $S/W$ has $q+1$ elementary abelian subgroups of order $q^2$ which contain $[V, S]/W$, and the preimages of these subgroups cover all elements of $S$. There is an element of order $q-1$ in $P_\Lambda^*(q)$ which normalises $U[V, S]$, $V$, $[V, S]$ and $W$ and permutes the remaining $q-1$ elementary abelian subgroups of $S/W$. Since $S$ has exponent $p^2$, the preimages in $S$ of the permuted subgroups all contain elements of order $p^2$, and finally, since $A\not\le V$, we conclude that $A=U[V, S]$. That is, $U[V, S]$ is the unique normal subgroup of $S$ of exponent $p$ and index $q$ which is not equal to $V$. Since $V$ is characteristic in $S$, so too is $U[V, S]$.
\end{proof}

\begin{lemma}\label{lem: LambdaIso}
Let $Z:=\gen{\overline{y^p}}_{\KK}\le \Lambda(q)$. Then $\Lambda(q)/Z  \cong V_{p-1}(q)$ as a $\KK U$-module. In particular $S_\Lambda(q)/Z \cong S_{p-1}(q).$
\end{lemma}
\begin{proof}
Observe that $Z=M^\perp$, where $M=\langle x^ay^{p-a} \mid a > 0 \rangle_{\KK}.$ We use \cite[Lemma VII.8.3]{HuppertBlackburnVolume2} and adopt the notation from there. Then $\Lambda(q)/Z$ is isomorphic to $M^*$ as a $\KK U$-module. We calculate that the map from $M$ to $V_{p-1}(q)$ given by $x^ay^{p-a} \mapsto x^{a-1}y^{p-a}$ is a $\KK U$-module isomorphism. Since $M$ and $M^*$ are isomorphic as $\KK U$-modules we obtain $\Lambda(q)/Z \cong V_{p-1}(q)$ as $\KK U$-modules. It thus follows that $S_\Lambda(q)/Z \cong S_{p-1}(q)$.
\end{proof}

\begin{lemma}\label{Supauto}
Let $p$ be an odd prime and $S=S_{\Lambda}(q)$. Then $\Aut(S)$ is solvable and has a Hall $p'$-subgroup of order $(q-1)^2.m_{p'}$, where $m_{p'}$ denotes the $p'$-part of $m$.
\end{lemma}
\begin{proof}
Applying Lemma \ref{(p+1)CUpS}, we have that $Z(S)=\gen{\ov{xy^{p-1}}, \ov{y^{p}}}_{\KK}$ and applying Lemma \ref{lem: LambdaIso} alongside Lemma \ref{Snauto} (\ref{sn3}) implies that $S/Z(S)\cong S_{p-2}(q)$. Now let $t$ be a $p'$-element of $\Aut(S)$ which acts trivially on $S/Z(S)$. Then by the three subgroups lemma, $t$ acts trivially on $S'$ and as $Z(S)\le S'$, coprime action implies that $t=1$. Hence, $C_{\Aut(S)}(S/Z(S))$ is a normal $p$-subgroup of $\Aut(S)$ and $\Aut(S)/C_{\Aut(S)}(S/Z(S))$ embeds as a subgroup of $\Aut(S_{p-2}(q))$. Applying Lemma \ref{Snauto}, we see that $\Aut(S)$ is solvable and a Hall $p'$-subgroup of $\Aut(S)$ has order at most $(q-1)^2.m_{p'}$. Since we witness a subgroup of $\Aut_{P_\Lambda^*(q)}(S)$ of order $(q-1)^2.m_{p'}$, the lemma holds.
\end{proof}

\begin{lemma}\label{action on centreLambda}
Suppose that $p$ is an odd prime and $q=p^m>p$. Set $S=S_\Lambda(q)$, $V= \Lambda(q)$ and $L=\Aut_{P_\Lambda^*(q)}(S)$. Let $K_1$ be a cyclic $p'$-subgroup of $C_L(C_{[V, S,S]}(S))$ and $K_2$ be a cyclic $p'$-subgroup of $C_L(C_V(S)[V, S,S]/[V, S, S])$ such that $|K_1|\geq \frac{q-1}{2}\leq |K_2|$. Then, for $i\in \{1,2\}$, we have that $|C_V(K_i)|=q$.
\end{lemma}
\begin{proof}
We observe by Lemma \ref{(p+1)CUpS} that $C_{[V, S,S]}(S)=\langle \bar{xy^{p-1}}\rangle_{\KK}$. We calculate that \[C_{D^*}(C_{[V, S,S]}(S))=\left\{ \left(1,a^{-1}d^{-p},
\begin{pmatrix}
a & 0  \\ c & d
\end{pmatrix}
\right) \mid  a, d \in \KK^{*}, c\in \KK \right\}.\]
Hence Lemma \ref{Gammacentraliser} implies
\[C_L(C_{[V, S,S]}(S))=\left\{c_z
 \mid z=  \left(1,a^{-1}d,
\begin{pmatrix}
ad^{-1} & 0 \\ c & 1
\end{pmatrix}
\right)  \in C_{D^*}(C_{[V, S,S]}(S))\right\}.\]
Since $|C_V(K_1^g)|=|C_V(K_1)|$ for any $g\in L$, we may assume that elements of $K_1$ correspond to elements of $C_L(C_{[V, S,S]}(S))$ of the form \[\left\{c_z
 \mid z=  \left(1,a^{-1}d,
\begin{pmatrix}
ad^{-1} & 0 \\ 0 & 1
\end{pmatrix}
\right)  \in C_{D^*}(C_{[V, S,S]}(S))\right\}.\]

Write \[v=\sum_{i=0}^p \lambda_i\bar{x^{p-i}y^{i}}\] so that \[v\cdot k=\sum_{i=0}^p \lambda_i(ad^{-1})^{i-p-1}\bar{x^{p-i}y^{i}}\] for $k\in K_1$. Note that $ad^{-1}$ has multiplicative order at least $\frac{q-1}{2}$ in $\KK$. Choose $k$  such that $K_1=\langle k \rangle$. By linearity $v\cdot k=v$ implies that $\lambda_i=0$ unless $i=p-1$  (since  $q>p$, $(ad^{-1})^{i-p-1}=1$ only if $i=p-1$). Hence, $C_V(K_1)=\langle \bar{xy^{p-1}}\rangle_{\KK}$, as desired.

Applying a conjugation argument as above, in order to prove the result for $K_2$ it suffices to consider the action of elements of the form \[\left\{c_z
 \mid z=  \left(1,\theta,
\begin{pmatrix}
a & 0 \\ 0 & d
\end{pmatrix}
\right)  \in D^*\right\}\] on the subspace $\langle \bar{y^p}\rangle_\KK$. Then the proof that $|C_V(K_2)|=q$ follows the same method as the proof for $K_1$.
\end{proof}

 \begin{notation}\label{n:bsandcs}
For $S\in \{S_n(q), S_{\Lambda}(q)\}$ where $1\leq n\leq p-1$, we set $d=1$ if $S=S_n(q)$ and $d=q$ if $S=S_{\Lambda}(q)$. Set $R:=UZ(S)$ and $Q:=UZ_2(S)$.

Define \[\begin{array}{rcl}\mathcal{B}(S)&=&\{B_0 \le S \mid \mbox{$B_0$ is elementary abelian, } |B_0|=q^2d \mbox{ and } S=B_0V\} \mbox{; and}\\
\mathcal{C}(S)&=&\{C_0 \le S \mid \mbox{$C_0$ is class 2 and exponent $p$, } |C_0|=q^3d \mbox{ and } S=C_0V\}.
\end{array}\]
Then $R\in\mathcal{B}(S)$ and $Q\in\mathcal{C}(S)$.
\end{notation} 

Observe that $Z(S) \le B_0$ for $B_0 \in \mathcal{B}(S)$ and $Z_2(S)\le C_0$ for $C_0 \in \mathcal{C}(S)$. Note also that $\mathcal{C}(S)$ only makes sense for $S=S_n(q)$ when $n\geq 2$.

\begin{lemma}\label{lem:intersec1}
Let $q=p^m$ and $S\in \{S_n(q), S_{\Lambda}(q)\}$. Then for $B_0\in\mathcal{B}(S)$ and $C_0\in\mathcal{C}(S)$, the following hold:
\begin{enumerate}
\item either $B_0\le C_0$ or $B_0\cap C_0=Z(S)$;
\item if $B_0, B_1\in\mathcal{B}(S)$ are such that $B_0 \ne B_1$ then $B_0\cap B_1=Z(S)$; and
\item if $C_0, C_1\in\mathcal{C}(S)$ are such that $C_0 \ne C_1$ then $C_0\cap C_1=Z_2(S)$.
\end{enumerate}
\end{lemma}
\begin{proof}
Let $B_0, B_1, C_0, C_1$ be as in the statement of the lemma. By Lemmas \ref{Somnibus} and \ref{(p+1)CUpS},  we have that $B_0=C_S(b)$ for $b\in B_0\setminus Z(S)$ and $|C_{C_0}(c)|=|C_S(c)|=q^2d$ for $c\in C_0\setminus Z_2(S)$. Since $Z(S)\le B_0\cap C_0$, if there is $b\in (B_0\cap C_0)\setminus Z(S)$, then $b\in C_0\setminus Z_2(S)$ and so $B_0\le C_S(b)=C_{C_0}(b)$. Comparing orders we obtain that $B_0=C_{C_0}(b)$, proving (1). Next note that any $b\in B_0\cap B_1$ is centralised by both $B_0$ and $B_1$ and so if $B_0\ne B_1$ then necessarily $B_0\cap B_1=Z(S)$. This proves (2). Finally to prove (3) note that $Z_2(S)\le C_0\cap C_1$ and if $c\in (C_0\cap C_1)\setminus Z_2(S)$, we must have $C_0=C_S(c)Z_2(S)=C_{C_1}(c)Z_2(S)=C_1$, a contradiction.
\end{proof}

\begin{lemma}\label{lem:S-conj}
Let $q=p^m$ and $S\in \{S_n(q), S_{\Lambda}(q)\}$. If $A\in \mathcal{B}(S)\cup \mathcal{C}(S)$, then every element of $A[V,S]\setminus [V,S]$ is $S$-conjugate to an element of $A$.
\end{lemma}
\begin{proof}
Suppose that $A\in \mathcal{B}(S)\cup \mathcal{C}(S)$ and let $a=1$ if $A\in\mathcal{B}(S)$ and $a=2$ if $A\in\mathcal{C}(S)$ so that $|A|=q^{a+d+1}$. It follows from Lemma \ref{Somnibus} and Lemma \ref{(p+1)CUpS} that $N_S(A)=AZ_{a+1}(S)$ has order $q|A|$. Observe that $S$-conjugates of $A$ are contained in $A[V, S]\normaleq S$ so that $S$-conjugates of elements in $A\setminus V$ are contained in $A[V, S]\setminus [V, S]$.

By Lemma \ref{lem:intersec1}, for each $s\in S\setminus N_S(A)$ we have $A\cap A^s=Z_a(S)$ has order $q^{a+d}$. If $S=S_{\Lambda}(q)$ then let $n=p$. For any $S$ we have that $|S|=q^{n+2}$ and $|A^S|=|S:N_S(A)|=q^{n-a-d}$. There are $|A^S||A\setminus Z_a(S)|$ elements of $S$ which are $S$-conjugate to some element of $A\setminus Z_a(S)$. We deduce that
\[|(A\setminus V)^S|=|A^S||A\setminus Z_a(S)|=q^{n-a-d}(q^{a+d+1}-q^{a+d})=q^{n+1}-q^n=|A[V, S]\setminus [V, S]|.\]

Therefore, every element of $A[V, S]\setminus [V, S]$ is $S$-conjugate to an element of $A$.
\end{proof}

\begin{lemma}\label{lem:S-conj2}
Let $q=p^m$ and $S\in \{S_n(q), S_{\Lambda}(q)\}$. Suppose that $A, A' \in \mathcal{B}(S)\cup \mathcal{C}(S)$ with $|A|=|A'|$. If $A[V,S] \cap A'[V,S] > [V,S]$, then $A$ and $A'$ are $S$-conjugate and $A[V,S]=A'[V,S]$.
\end{lemma}
\begin{proof}
Suppose that $A[V,S] \cap A'[V,S] > [V,S]$. Then, by the Dedekind Modular Law, $A' \cap A[V,S] \not \le [V,S]$. Choose $a \in (A'\cap A[V,S])\setminus [V,S]$. Then, by Lemma \ref{lem:S-conj}, there exists $s\in S$ such that $a^s \in A \setminus V$. Thus $A'^s \cap A \not\le V$. Observe that $A'^s\in\mathcal{B}(S)\cup \mathcal{C}(S)$ and $|A'^s|=|A|$.
Hence, by Lemma \ref{lem:intersec1} we conclude that $A=A'^s$ so $A$ and $A'$ are conjugate in $S$.
\end{proof}

To close this section, we focus our attention on the groups $S=S_n(q)$ for $1\leq n\leq p-1$.

 \begin{notation}\label{n:deltamap}
For $X \unlhd S$, define
\[\Aut(S,X)=\{\alpha|_X\mid \alpha \in \Aut(S) \text{ with } X\alpha = X\}\subseteq \Aut(X).\]
For $Y \unlhd X$ and $\alpha \in  \Aut(S,X)$ with $Y\alpha = Y$,  define $\alpha|_{X/Y}\in \Aut(X/Y)$ by $(xY)\alpha|_{X/Y}=x\alpha Y$. Then  put
\[\Aut(S,X/Y) = \{ \alpha|_{X/Y} \mid \alpha \in  \Aut(S,X)  \text{ with } Y\alpha=Y\}\subseteq \Aut(X/Y).\]
We shall be concerned with
\[\Delta=\Aut(S,S/V) \times \Aut(S,Z(S))\cong \Gamma_1(\KK)\times \Gamma_1(\KK)\]
where we regard $\Gamma_1(q) =\Aut(\KK)\ltimes \KK^*$.

We define a map
\begin{eqnarray*}
\delta: \Aut(S) & \rightarrow &\Delta\\ \phi&\mapsto&(\phi|_{S/V},\phi|_{Z(S)}).
\end{eqnarray*}
\end{notation} 

\begin{lemma}\label{lem: kerneldescription}
The kernel of the map $\delta$ is a $p$-group containing $\Inn(S)$. In particular, $\delta$ is injective upon restricting to a Hall $p'$-subgroup of $\Aut(S)$.
\end{lemma}
\begin{proof}
Suppose $\phi\in \Aut(S)$ has $p'$-order and $(\phi)\delta=(1,1)$. Then $\phi$ centralises $S/V$  and $Z(S)= C_V(S)$. The Thompson A$\times$B-Lemma \cite[(24.2)]{AschbacherFG} now implies that $\phi$ centralises $V$ and coprime action yields $\phi=1$. Hence $\ker \delta$ is a $p$-group. Since $\Inn(S)$ centralises $S/V$ and $Z(S)$, it follows that $\Inn(S)\le \ker(\delta)$.
\end{proof}

Notice that $\Aut_B(S)\delta = \Aut_\Sigma(S)\delta$ has order $(q-1)^2.m_{p'}$, where $m_{p'}$ denotes the $p'$-part of $m$. Recall the group $\Sigma$ from \ref{n:subsofd*}.  Identifying $S/V$ and $Z(S)$ with $\KK$, taking $t=(\phi,\lambda, \begin{pmatrix} \mu&0\\0&\nu \end{pmatrix})\in \Sigma$ and letting $c_t\in \Aut(S)$ be the automorphism of $S$ induced by conjugation by $t$, we obtain \begin{eqnarray}\label{eq:delta-sigma}
c_{t }\delta &=&((\phi,\mu \nu^{-1}), (\phi,\lambda\mu^n)).
\end{eqnarray}

\section{A new family of exotic fusion systems}\label{sec : NewExotics}

Adopting the methodology in \cite{ClellandParker2010}, the objective of this section is to describe exotic fusion systems supported on $S_{\Lambda}(q)$ for various $q$.  To prove these examples are saturated, we will apply the following result from \cite{G2p}:

\begin{theorem}\label{t:satamalg}
Let $\mathcal A= (G_1 \ge G_{12} \le G_2)$ be an amalgam of finite groups, assume that $\Syl_p(G_{12}) \subseteq \Syl_p(G_2)$ and fix $S_i \in \Syl_p(G_i)$ with $S_2 \le S_1$.  Assume that $G=G_1 *_{G_{12}} G_2$ is the   universal completion of $\mathcal A$ and write $\Gamma=\Gamma(G,G_1,G_2,G_{12})$ for   the coset graph. Suppose that:
\begin{enumerate}
\item[(a)] for all $\F_{S_1}(G)$-centric subgroups $P$ of $S_1$, $\Gamma^P$ is finite; and
\item[(b)] each $\F_{S_i}(G_i)$-essential subgroup is $\F_{S_1}(G)$-centric.
\end{enumerate}
Then $\F_{S_1}(G)$ is saturated.
\end{theorem}
\begin{proof}
See \cite[Theorem 2.5]{G2p}.
\end{proof}

  Let $\Gamma$ be the coset graph of the amalgam $\mathcal{A}$ in Theorem \ref{t:satamalg}. Thus $\Gamma$ has vertex set the right cosets of $G_1$ and $G_2$ in $G$ and two vertices are incident if and only if their intersection is non-empty  and they are not equal. We know that $G$ acts edge transitively by right multiplication  on $\Gamma$. Since $G$ is the free amalgamated product of $G_1$ and $G_2$, $\Gamma$ is a tree. For more details see \cite[Section 4.1]{ParkerRowleySymplectic}.

Fix $\KK$ of size $q=p^m$ for $p$ an odd prime. Set $V=\Lambda(q)$ and $S=S_{\Lambda}(q)$.  Recall from Notation \ref{n:slandpl} that we write $P_\Lambda^*(q)= D^* \ltimes V$ and $P_{\Lambda}(q)=D\ltimes V$ so that $P_{\Lambda}(q)\le P_{\Lambda}^*(q)$.

For this section, write $P^*:=P_{\Lambda}^*(q)$ and $P:=P_{\Lambda}(q)$.  Then elements of $P^*$ are regarded as tuples and the normaliser of a Sylow $p$-subgroup $S$ of $P^*$ is \[\left\{\left(\phi, \theta, \left(\begin{matrix}
a & 0 \\ c & b
\end{matrix}\right), \sum_{i=0}^p \lambda_i \overline{x^iy^{p-i}} \right)\mid a,b,\theta \in \KK^\times, c,\lambda_i \in \KK, \phi\in O^p(\Aut(\KK)) \right\}\]
with $S$ consisting of the elements in which $\phi=\theta=a=b=1$. Then, by Lemma \ref{(p+1)CUpS}(\ref{cups1}),
 \[Z(S)=\langle \overline{y^p},  \overline{xy^{p-1}} \rangle_{\KK}.\]
 Recall from Notations \ref{n:subsofd*} and \ref{n:bsandcs} that $R=UZ(S)=\left\langle Z(S), (1, 1, \left(\begin{smallmatrix} 1 & 0 \\ \lambda & 1 \end{smallmatrix}\right), 0) \mid \lambda \in \mathbb K\right\rangle\in \mathcal{B}(S)$. Then $N_S(R)= R \langle\overline{x^2y^{p-2}} \rangle_{\KK}=RZ_2(S)$, $R$ is an elementary abelian $p$-group of order $q^3$ and $|N_S(R)|=q^4$ by Lemma \ref{(p+1)CUpS}(\ref{cups2}). 

Define the subgroup $P_R^*$ of $\Gamma\mathrm{L}_4(q)\cong \Aut(\KK)\ltimes \GL_4(\KK)$ as follows \[P_R^*=\left\{\left(\phi, \left(\begin{smallmatrix} x_1 & 0 & 0 & x_2 \\ 0 & x_3 & x_4 & x_5 \\ 0 & x_6 & x_7 & x_8 \\ 0 & 0 & 0 & x_9 \end{smallmatrix} \right)\right)\mid \phi\in O^p(\Aut(\KK));  x_i\in\KK\,\,\text{for}\,\,1\leq i\leq 9\,\,\text{where}\,\,x_1(x_3x_7-x_4x_6)x_9\ne 0 \right\}.\]

\begin{proposition}\label{p:newamalg}
The map $\psi^* : N_{P^*}(R) \rightarrow P_R^*$ given by \[\left(\phi, \theta, \left(\begin{matrix}
a & 0 \\ c & b
\end{matrix}\right), \lambda \overline{y^p} + \mu \overline{xy^{p-1}} + \nu \overline{x^2y^{p-2}} \right) \mapsto \left(\phi, \left(\begin{smallmatrix} \theta ab^p & 0 & 0 & \theta ab^p\lambda \\ 0 & \theta a^2 b^{p-1} & \theta a^2 b^{p-1} \nu & \theta a^2 b^{p-1} \mu \\ 0 & 0 & b & c \\ 0 & 0 & 0 & a  \end{smallmatrix}\right)\right)\] is a monomorphism. Furthermore, the following hold \begin{enumerate}
\item  $R\psi^* = O_p(P_R)$; and
\item $\left(1, \zeta^{p}, \left(\begin{matrix}
1 & 0 \\ 0 & \zeta^{-1}
\end{matrix}\right), 0 \right)\psi^* =\left(1, \diag(1,\zeta^{-1},\zeta,1)\right)$ for $\zeta \in \mathbb K^*$.
\end{enumerate}
\end{proposition}
 
\begin{landscape}
\begin{proof}
We check this directly. Assume that $\Phi\in \Aut(\KK)$ acts by raising elements of $\KK$ to the power $p^l$. Using the description of the action of $B$ on $\{\overline{x^2y^{p-2}},\overline{xy^{p-1}},\overline{y^p}\}$ in the discussion which precedes Lemma \ref{Gammacentraliser} we have,
\[\left(\phi, \theta, \left(\begin{matrix}
a & 0 \\ c & b
\end{matrix}\right), \lambda \overline{y^p} + \mu \overline{xy^{p-1}} + \nu \overline{x^2y^{p-2}} \right) \cdot \left(\Phi, \Pi, \left(\begin{matrix}
\alpha & 0 \\ \kappa & \beta
\end{matrix}\right), \chi \overline{y^p} + \tau \overline{xy^{p-1}} + \eta \overline{x^2y^{p-2}} \right)\]
\[=\left(\phi\Phi, \theta^{p^l} \Pi, \left(\begin{matrix}
a^{p^l} & 0 \\ c^{p^l} & b^{p^l}
\end{matrix}\right)\left(\begin{matrix}
\alpha & 0 \\ \kappa & \beta
\end{matrix}\right), \left(\lambda^{p^l} \overline{y^p} + \mu^{p^l} \overline{xy^{p-1}} + \nu^{p^l} \overline{x^2y^{p-2}} \right)\cdot \left( \Pi, \left(\begin{matrix}
\alpha & 0 \\ \kappa & \beta
\end{matrix}\right)\right) + \chi \overline{y^p} + \tau \overline{xy^{p-1}} + \eta \overline{x^2y^{p-2}}\right)\]
\[=\left(\phi\Phi, \theta^{p^l} \Pi, \left(\begin{matrix}
a^{p^l}\alpha & 0 \\ c^{p^l}\alpha+\kappa b^{p^l} & b^{p^l}\beta
\end{matrix}\right), (\lambda^{p^l} \Pi^{-1} \beta^{-p}+\chi) \overline{y^p} + (\nu^{p^l} \Pi^{-1} \kappa \alpha^{-2} \beta^{1-p}+\mu^{p^l} \Pi^{-1} \alpha^{-1}\beta^{1-p}+\tau)\overline{xy^{p-1}} + (\nu^{p^l} \Pi^{-1} \alpha^{-2}\beta^{2-p}+\eta) \overline{x^2y^{p-2}} \right).\]
On the other hand,
\[\left(\phi, \theta, \left(\begin{matrix}
a & 0 \\ c & b
\end{matrix}\right), \lambda \overline{y^p} + \mu \overline{xy^{p-1}} + \nu \overline{x^2y^{p-2}} \right)\psi^*  \cdot \left(\Phi, \Pi, \left(\begin{matrix}
\alpha & 0 \\ \kappa & \beta
\end{matrix}\right), \chi \overline{y^p} + \tau \overline{xy^{p-1}} + \eta \overline{x^2y^{p-2}} \right)\psi^* \]
\[\left(\phi, \left(\begin{smallmatrix} \theta ab^p & 0 & 0 & \theta ab^p\lambda \\ 0 & \theta a^2 b^{p-1} & \theta a^2 b^{p-1} \nu & \theta a^2 b^{p-1} \mu \\ 0 & 0 & b & c \\ 0 & 0 & 0 & a  \end{smallmatrix}\right)\right) \cdot \left(\Phi, \left(\begin{smallmatrix} \Pi \alpha\beta^p & 0 & 0 & \Pi \alpha\beta^p\chi \\ 0 & \Pi \alpha^2 \beta^{p-1} & \Pi \alpha^2 \beta^{p-1} \eta & \Pi \alpha^2 \beta^{p-1} \tau \\ 0 & 0 & \beta & \kappa \\ 0 & 0 & 0 & \alpha  \end{smallmatrix}\right)\right)\]
\[=\left(\phi\Phi, \left(\begin{smallmatrix} \theta (ab^p)^{p^l} & 0 & 0 & (\theta ab^p\lambda)^{p^l} \\ 0 & (\theta a^2 b^{p-1})^{p^l} & (\theta a^2 b^{p-1} \nu)^{p^l} & (\theta a^2 b^{p-1} \mu)^{p^l} \\ 0 & 0 & b^{p^l} & c^{p^l} \\ 0 & 0 & 0 & a^{p^l}  \end{smallmatrix}\right) \left(\begin{smallmatrix} \Pi \alpha\beta^p & 0 & 0 & \Pi \alpha\beta^p\chi \\ 0 & \Pi \alpha^2 \beta^{p-1} & \Pi \alpha^2 \beta^{p-1} \eta & \Pi \alpha^2 \beta^{p-1} \tau \\ 0 & 0 & \beta & \kappa \\ 0 & 0 & 0 & \alpha  \end{smallmatrix}\right)\right)\]
\[=\left(\phi\Phi, \left(\begin{smallmatrix}
(\theta ab^p)^{p^l}\Pi\alpha\beta^p&0&0&(\theta ab^p)^{p^l}\Pi\alpha\beta^p\chi+(\theta ab^p\lambda)^{p^l}\alpha\\
0&(\theta a^2b^{p-1})^{p^l}\Pi\alpha^2\beta^{p-1}&(\theta a^2 b^{p-1})^{p^l}\Pi\alpha^2\beta^{p-1}\eta+(\theta a^2b^{p-1}\nu)^{p^l}\beta&(\theta a^2b^{p-1})^{p^l}\Pi\alpha^2\beta^{p-1}\tau+(\theta a^2b^{p-1}\nu)^{p^l}\kappa+(\theta a^2b^{p-1}\mu)^{p^l}\alpha\\
0&0&b^{p^l}\beta&b^{p^l}\kappa+\alpha c^{p^l}\\
0&0&0&a^{p^l}\alpha
\end{smallmatrix}\right)\right).\]

Therefore $\psi^*$ is a homomorphism. Since $\ker \psi=1^*$,  $\psi^*$ is a monomorphism.

It is elementary to check that (1) and (2) hold.
\end{proof}
\end{landscape}
 
Define \[P_R=\left\{\left(1, \left(\begin{smallmatrix} x_1 & 0 & 0 & x_2 \\ 0 & x_3 & x_4 & x_5 \\ 0 & x_6 & x_7 & x_8 \\ 0 & 0 & 0 & x_9 \end{smallmatrix} \right)\right)\mid   x_i\in\KK\,\,\text{for}\,\,1\leq i\leq 9\,\,\text{where}\,\,x_1(x_3x_7-x_4x_6)x_9\ne 0 \right\}\le P_R^*\] and observe that the restriction of $\psi^*$ to $N_P(R)$ delivers an injective homomorphism with image contained in $P_R$. Write $\psi=\psi^*|_{N_P(R)}$.

We can now construct two free amalgamated products, using $\psi^*$. Set \[F^*=F^*(\KK)=P^*\ast_{N_{P^*}(R),\psi^*} P^*_R\] and \[F=F(\KK)=P\ast_{N_P(R), \psi} P_R.\] Thus $F^*$ and $F$ are infinite groups with subgroups isomorphic to $P^*$ and $P^*_R$, and $P$ and $P_R$ respectively, which intersect in a subgroup isomorphic to $N_{P^*}(R)$, $N_P(R)$ respectively.  We identify $P^*, P_R^*$ and $P, P_R$ as subgroups of $F^*$ and $F$ respectively. In particular $S$ is identified as a subgroup of both $F^*$ and $F$.

\begin{notation}\label{n:f*lambda}
Define the fusion systems on $S=S_{\Lambda}(q)$:
\[\F_\Lambda^*(q)=\F_S(F^*) \mbox{ and } \F_\Lambda(q)=\F_S(F).\]
\end{notation}

Clearly, $\F_{\Lambda}(q)$ is a fusion subsystem of $\F_{\Lambda}^*(q)$. Our main result of this section is:

\begin{proposition}\label{UPV saturated}
Both $\F_{\Lambda}^*(q)$ and $\F_{\Lambda}(q)$ are saturated fusion systems.
\end{proposition}

Before proving Proposition \ref{UPV saturated} we require a lemma.

\begin{lemma}\label{lem: E cap Ex}
Assume that $T= N_{P^*}(R)$ and $s\in P^*\setminus T$. Then $R\cap R^s\le V$.
\end{lemma}
\begin{proof}
If $s$ does not normalise $S$, then we note that $R \cap R^s \le S \cap S^s=V$. Hence we consider $s \in N_{P^*}(S)$. Assume that $w\in R \cap R^s $ and that $w \not\in V$. Then $w$ centralises $\langle R,R^s\rangle$ as $R$ is abelian.  Since $\gen{R,R^s}$ has order at least $q^3p$ and  $|\gen{R,R^s}: \gen{R,R^s} \cap V|=q^2$,  we deduce that $\langle R,R^s\rangle\cap V > Z(S)$. Thus, $C_V(w)> Z(S)$ which is a contradiction to Lemma \ref{(p+1)CUpS} (1). Hence $R \cap R^s \le V$.
\end{proof}

\begin{proof}[Proof of Proposition \ref{UPV saturated}]
We prove Proposition \ref{UPV saturated} only for $\F_{\Lambda}^*(q)$ but the proof for $\F_{\Lambda}(q)$ is the same. We intend to apply  Theorem \ref{t:satamalg}   with $G=F^*$. For convenience we set $M^*=P^*_R$.  Thus, as above, $\Gamma$  has vertex set the right cosets of $P^*$ and $M^*$ in $F^*$ and two vertices are incident if and only if their intersection is non-empty  and they are not equal. Then $F^*$ acts edge transitively by right multiplication  on $\Gamma$, and $\Gamma$ is a tree.

Suppose that $X\le S$ is an $\F_S(S)$-centric subgroup (there are more of these than there are $\F_{\Lambda}^*(q)$-centric subgroups), and assume that $X\le S$ fixes the path $P^*y ,M^*, P^*,M^*s$ of length $4$ in $\Gamma$ where $s \in P^*\setminus M^*$ and $y \in M^*\setminus P^*$. Since $R\in \syl_p((P^*\cap M^*)^y\cap(P^*\cap M^*))$, we have  $X\le R$. Therefore, as $P^*\cap M^*= N_{P^*}(R)$, $s \not \in N_{P^*}(R)$ and so  $X \le R \cap R^s \le V$ by Lemma~\ref{lem: E cap Ex}. Since  $X < V$, we now have a contradiction as $X$ is $\F_S(S)$-centric and $V$ is abelian.

Now suppose that $X \le S$ is $\F_{\Lambda}^*(q)$-centric  and $\Gamma^X$ is infinite. Then $X$ fixes the vertex $P^*$. Since $\Gamma$ is a tree, $\Gamma^X$ is connected. Hence $\Gamma^X$ fixes an infinite path emanating from $P^*$ starting \[P^*,M^*a,P^*b,M^*c,\dots.\] Since $M^*a\cap P^*b$ is non-empty we may suppose that $a=b$. Thus $X^{a^{-1}}$ fixes the path \[P^*a^{-1},M^*,P^*,M^*ca^{-1},\dots.\] Because $X^{a^{-1}} \le P^* \cap M^*= N_{P^*}(R)$, we see that $X^{a^{-1}}\le  S$ and as  $X$ is $\F_{\Lambda}^*(q)$-centric, we deduce that $X^{a^{-1}}$ is $\F_S(S)$-centric. This therefore contradicts the conclusion of the previous paragraph. Hence, no $\F_{\Lambda}^*(q)$-centric subgroup of $S$ fixes an infinite path in $\Gamma$.

The only $\F_S(P^*)$-essential subgroup is $V$ and the unique $\F_{N_S(R)}(M^*)$-essential subgroup is $R$. As $V$ is weakly $\F_{\Lambda}^*(q)$-closed by Lemma \ref{(p+1)CUpS} (7), $V$ is surely $\F_{\Lambda}^*(q)$-centric. Now suppose that $a \in F^*$ and that $R^a \le S$ is not $\F_S(S)$-centric. Since $R^a\le S$, we have that $R^a$ fixes the vertex of $\Gamma$ corresponding to $P^*$ and since $R\le M^*$, we have that $R^a$ fixes the vertex of $\Gamma$ corresponding to $M^*a$. Since $\Gamma$ is a tree, $R^a$ fixes the unique path with $P^*, M^*a$ as its end points. Assume that $P^*$ and $M^*a$ are adjacent in $\Gamma$. Then $a\in P^*$ and we deduce that $S=R^aV$. Since $R^a$ is abelian of order $q^3$, we conclude that $C_V(R^a)=C_V(S)=R^a\cap V$ and so $R^a$ is $\F_S(S)$-centric, a contradiction. Hence, $R^a$ fixes the path \[P^*,M^*b,P^*c,\dots,M^*a\] of length at least $4$. Conjugating by $b^{-1}$ shows that we may take $M^*b=M^*$ and suppose that $R^a$ fixes \[P^*,M^*,P^*c,\dots,M^*a.\] Then $R^a\le P^*\cap M^*\cap {P^*}^c$ for some $c\in M^*$. But $P^*\cap M^*\cap {P^*}^c=R$ and we conclude that $R^a=R$ is $\F_S(S)$-centric, a contradiction. This completes the proof.
\end{proof}

\begin{proposition}
We have that $\F_{\Lambda}(q)=O^{p'}(\F_{\Lambda}^*(q))$, $|\Out_{\F_{\Lambda}^*(q)}(S)|=(q-1)^2.m_{p'}$ and $|\Out_{\F_{\Lambda}(q)}(S)|=(q-1)^2$.
\end{proposition}
\begin{proof}
We observe that $V$ is a centric, abelian, characteristic subgroup of $S$. Hence, $p'$-order elements of $\Aut(S)$ restrict faithfully to elements of $\Aut(V)$. Therefore, to calculate $|\Out_{\F}(S)|$ for $\F\in\{\F_{\Lambda}^*(q), \F_{\Lambda}(q)\}$ it suffices to compute the order of a Hall $p'$-subgroup of $N_{\Aut_{\F}(V)}(\Aut_S(V))$. For this we observe that $\Aut_{\F_{\Lambda}^*(q)}(V)=\Aut_{P^*}(V)$ and $\Aut_{\F_{\Lambda}(q)}(V)=\Aut_{P}(V)$ so that the orders of the outer automorphism groups follow quickly.

We now use \cite[Theorem I.7.7]{AKO} to show that $O^{p'}(\F_{\Lambda}^*(q))=\F_{\Lambda}(q)$, adopting the notation from \cite{AKO}. We have that $\mathcal E_0$ contains $O^{p'}(\Aut_{P^*}(V))=O^{p'}(\Aut_{P}(V))$ which contains the conjugation map which corresponds to $(1, 1,\left(\begin{smallmatrix} a&0\\0&a^{-1} \end{smallmatrix}\right),0)\in P^*$, where $a \in \mathbb K^*$. Moreover, $\mathcal{E}_0$ also contains the conjugation map determined by $d=\diag(1,\zeta^{-1},\zeta,1)\in P_R$ for all $\zeta \in \KK^*$ on $R$. By Proposition \ref{p:newamalg} such maps correspond to conjugations by elements from $C_{D^*}(V)\left\langle \left(1, \zeta^{p}, \left(\begin{matrix}
1 & 0 \\ 0 & \zeta^{-1}
\end{matrix}\right), 0\right)\mid \zeta \in \mathbb K^*\right\rangle$ where we recognise $C_{D^*}(V) = \left\{ \left(1, a^{-p},
\begin{pmatrix}
a & 0 \\ 0 & a
\end{pmatrix}
, 0\right) \mid  a \in \KK^{*} \right\}$ as a subgroup of $P^*$. Indeed, the conjugation map determined to $d$ lifts to a faithful automorphism of $S$ and generates a subgroup of order $(q-1)$ in $\Aut_\F(S)$. Since this subgroup intersects $\langle (1, 1,\left(\begin{smallmatrix} a&0\\0&a^{-1} \end{smallmatrix}\right),0)\rangle$ trivially, we deduce that $\Aut_{\F_{\Lambda}^*(q)}^0(S)= \Aut_{\F_{\Lambda}(q)}^0(S)=\Aut_{\F_{\Lambda}(q)}(S)$ and now $\F_{\Lambda}(q)=O^{p'}(\F_{\Lambda}^*(q))$ by \cite[Theorem I.7.7 (d)]{AKO}.
\end{proof}

\begin{proposition}
The fusion system $\F_{\Lambda}(q)$ is simple.
\end{proposition}
\begin{proof}
Set $\F=\F_{\Lambda}(q)$ so that $\F=O^{p'}(\F)$. Suppose that $X$ is strongly closed in $S$. Then $X$ is normal in $S$ and $C_V(S)\cap X \ne 1$. Let $z \in C_V(S)\cap X$ be non-trivial. Then $\gen{z^{\Aut_\F(V)}}$ contains the invariant subgroup $W$ of $V$ of order $q^{p-1}$. Thus we may assume that $z$ is in $C_W(S)= \gen{\ov{xy^{p-1}}}_{\KK}$ Lemma \ref{(p+1)CUpS}(1). But then $X \ge \gen{z^{\Aut_\F(R)}}\not \le V$.  Let $w \in X\setminus V$. Then $X \ge [V,X]\ge [V,w]$. As $|V/[V,w]|= q^2$ by Lemma \ref{(p+1)CUpS}(3), and since $[V,w] \not \le W$, we have $\gen{[V,X]^{\Aut_\F(V)}}= V$. Thus $V < X$.  Finally, $X=\gen{X^{\Aut_\F(S)}}= S$. Then \cite[Theorem II.9.8(d)]{AKO} yields that $\F$ is simple.
\end{proof}

\begin{proposition}\label{New exotic}
Every saturated fusion system of $p'$-index in $\F^*_{\Lambda}(q)$ is exotic.
\end{proposition}
\begin{proof}
Let $\F$ be a saturated fusion system of $p'$-index in $\F^*_{\Lambda}(q)$ and suppose that $\F$ is not exotic. Then \cite[Lemma III.6.1]{AKO} implies that $\F$ is realised by an almost simple group. Since $O^{p'}(\F)=\F_{\Lambda}(q)$ is simple, it follows that $S \in \syl_p(F^*(G))$.

Suppose that $F^*(G)$ is a sporadic simple group. Then using that $|S|= q^{p+2}$, quickly rules out almost all cases, except perhaps the remote possibility that $G$ is the Monster with $p=3$ and $q=81$, but then the centre of a Sylow $3$-subgroup of $G$ has order $3$ which is a contradiction.

If $F^*(G)$ is an alternating group $\Alt(n)$, then, as $S$ is not a direct product, we see that $n=p^k$.  But in this case the Sylow $p$-subgroups of $\Alt(n) $ are iterated wreath products which have cyclic centre.

Suppose $F^*(G)$ is a Lie type group defined in characteristic $p$. Then by the Borel--Tits theorem \cite[Theorem 3.13]{GLS3}, unless $R$ is normal in $S$, $R$ is not a $p$-radical subgroup. As $R$ is not normal in $S$, this is impossible.

We have shown that $F^*(G)$ is a Lie type group defined in characteristic $r$, where $r\ne p$. Suppose that $G$ is not an exceptional group with $p\in\{3,5\}$. Then \cite[Theorem 4.10.3(e)]{GLS3} first yields that $V$ is contained in a maximal rank abelian subgroup of $S$, and then that $R$ is conjugate to a subgroup of $V$. This contradicts the fact that $R$ is $\F$-centric. Hence, it remains to consider the exceptional groups when $p\in \{3,5\}$. If $|S/V|=p$ then we appeal to \cite[Theorem 2.8]{p.index2} for a contradiction.

Thus, we may assume that $|S/V|>p$ so that $|V|=q^4\geq p^8$ with equality if and only if $p=3$ and $q=9$. We observe by \cite[Corollary 13.23]{ParkerRowleySymplectic} that the $p$-rank of $S$ is bounded above by the rank of $F^*(G)$. It follows first that $F^*(G)\cong \mathrm E_8(r^c)$ and then that $p=3$. Thus $|V|=3^8$ and $|S|=3^{10}$. Applying \cite[Theorem 4.10.3(a) and (c)]{GLS3}, $V$ is a maximal torus in $F^*(G)$. Since the Weyl group of $\mathrm E_8(r^c)$ is isomorphic to $2.\mathrm{GO}^+_8(2)$ which has Sylow $3$-subgroups of order $3^5$, this implies $|S|=3^{13}$, a contradiction. Hence no such $G$ exists and consequently $\F$ is exotic.
\end{proof}

 Combining the results in this section, we have demonstrated the following theorem.
 
\begin{theorem}\label{t:newexotics}
For $p$ an odd and $q=p^m$, the group $S_\Lambda(q)$ supports a saturated fusion system $\F^*=\F_\Lambda^*(q)$. Moreover,  $O^{p'}(\F^*)$ is simple  and all $p'$-index subsystems of $\F^*$ are exotic.
\end{theorem}

\begin{remark}\label{NotIso}
Assume that $p$ is odd, $q=p^m$, and let $\F$ be a fusion subsystem of $p'$-index in $\F^*_{\Lambda}(q)$ on $S=S_\Lambda(q)$.

When $q=p$, the group $S$ has an elementary abelian subgroup of index $p$, and the fusion systems in Theorem~\ref{t:newexotics} occur among the examples considered in \cite{p.index2} (see the second and seventh rows of \cite[Table~4.1]{p.index2} and \cite[Proposition~4.1]{p.index2}). On the other hand, when $q>p$ these examples appear to be new.

At the time of writing, apart from a handful of further examples mentioned below, the majority of known exotic fusion systems are described in \cite{GPMaxClass}. Broadly speaking, these fall into six families, which we now consider in turn. Note that $O^{p'}(\F)$ is a simple fusion system supported on $S$.

\begin{itemize}
\item[(1)] \emph{$\F$ is not an exotic Benson--Solomon system described in \cite{LO}:}
This is immediate, since $p$ is odd.

\item[(2)] \emph{$\F$ is not a van Beek exotic fusion system on a subgroup of a Sylow $3$-subgroup of the sporadic Thompson group, or of a Sylow $5$-subgroup of the Monster, as described in \cite{vbexotics}:}
Indeed, $|Z(S)|=q^2$ with $q>p$.

\item[(3)] \emph{$\F$ is not an exotic fusion system on a $p$-group with an abelian subgroup of index $p$, as described in \cite{p.index1}, \cite{p.index2}, or \cite{p.index3}:}
Since $q>p$, the group $S$ has no such subgroup.

\item[(4)] \emph{$\F$ is not an exotic fusion system on a $p$-group of maximal class described in \cite{GPMaxClass}:}
As $q>p$ and $|Z(S)|=q^2$, the group $S$ is not of maximal class.

\item[(5)] \emph{$\F$ is not a polynomial fusion system on $S_n(q)$ for $1\le n\le p-1$, as described in \cite{ClellandParker2010}, nor one of its pruned variants described in \cite{henke2023punctured}:}
In contrast to these $p$-groups, we have $|Z(S)|=q^2$ and $S$ has an abelian subgroup of index $q$.

\item[(6)] \emph{$\F$ is not a $p'$-index exotic subsystem of the $p$-fusion system of a finite simple group $H$ of Lie type in characteristic $r\ne p$, as described in \cite{oliver2020simplicity}:}
Suppose otherwise, and let $T\in \Syl_p(H)$. Then $O^{p'}(\F_T(H))$ is an exotic fusion system as in \cite[Theorem~A(c)]{oliver2020simplicity}. In particular, by \cite[Lemma~2.2(a)]{oliver2020simplicity},
\[
\{V,R^S\}\subseteq \F^{rc}
=O^{p'}(\F)^{rc}
=\F_T(H)^{rc}.
\]
Replaying the final two paragraphs of the proof of Proposition~\ref{New exotic}, we conclude that $T\not\cong S$. Hence $O^{p'}(\F)\not\cong O^{p'}(\F_T(H))$, as required.
\end{itemize}

Thus $\F$ is not isomorphic to any of the fusion systems listed above, and these cases exhaust all currently known exotic fusion systems. We conclude that the fusion systems in Theorem~\ref{t:newexotics} are new.
\end{remark}

\section{Polynomial fusion systems}\label{sec:poly}

In this section, we describe the polynomial fusion systems which appear in the conclusions of our main theorems. For this, as in Section \ref{sec : NewExotics}, we require certain extensions of the Clelland-Parker fusion systems by field automorphisms of $p'$-order. We begin by recalling the definitions of these systems, following the presentation in \cite{ClellandParker2010}.

 Let $p$ be an odd prime and for $1\leq n\leq p-1$ set $S=S_n(q)$. Recall the definition of $U$ from Notation \ref{n:subsofd*} and the definition of $R$ and $Q$ from Notation \ref{n:bsandcs}. Thus $R$ is elementary abelian of order $q^2$ and $Q \cong S_1(q)$ by Corollary \ref{prop: Q iso}. In \cite[Section 4]{ClellandParker2010} two groups $P_R$ and $P_Q$ are defined. These groups have a normal subgroup isomorphic to $R$ and $Q$ respectively with $\Aut_{P_R}(R) \cong \GL_2(q)\cong \Out_{P_Q}(Q)$. Recall from Notation \ref{n:snandpn} that $P_n(q) = D\ltimes V_n(q)$.  Then monomorphisms $\Psi_R:N_{P_n(q)}(R) \rightarrow P_R$ and $ \Psi_Q:N_{P_n(q)}(Q)\rightarrow P_Q$ are given such that the free amalgamated products $F(n,\KK,R)= P_n(q)\ast_{N_{P_n(q)}(R)}P_R$ and $F(n,\KK,Q)= P_n(q)\ast_{N_{P_n(q)}(Q)}P_Q$ can be constructed. Then $P_n(q)$ and $P_R$, respectively $P_Q$, are identified as subgroups of these products. In particular, $S$ is a subgroup and the fusion systems $\F(n ,q, R)=\F_S(F(n,\KK,R))$ and $\F(n ,q, Q)=\F_S(F(n,\KK,Q))$ are shown to be saturated and mostly exotic.

 Recall from Notation \ref{n:snandpn} that $P_n^*(q)=D^* \rtimes V_n(q)$, and note that since any element of $\Aut(\KK)$ acts on $P_R$, $P_Q$ by raising matrix entries to some appropriate $p$-power, the maps $\Psi_R$ and $\Psi_Q$ defined above are $\Aut(\KK)$-equivariant. Thus, in parallel to the construction in Section \ref{sec : NewExotics}, one may similarly define \[P^*_R:= O^p(\Aut(\KK)) \ltimes P_R \mbox{ and } P^*_Q := O^p(\Aut(\KK)) \ltimes P_Q\] and maps $\Psi_Q^*: N_{P_n^*(q)}(Q) \rightarrow P_Q^*$ and $\Psi_R^*: N_{P^*_n(q)}(R) \rightarrow P_R^*$ such that the free amalgamated products \[F^*(n,\KK,R):= P_n^*(q) *_{N_{P_n^*(q)}(R)}P^*_R \mbox{ and } F^*(n,\KK,Q):= P_n^*(q)*_{N_{P_n^*(q)}(Q)}P^*_Q\] can be constructed.
 
\begin{notation}\label{n:BigCPsystems}
We write $\F^*(n ,q, R)$, respectively $\F^*(n ,q, Q)$, for the fusion systems associated to $F^*(n,\KK,R)$, respectively $F^*(n,\KK,Q)$, which may be regarded as the largest saturated fusion systems defined on $S_n(q)$, where $q=p^m>p$, among those such that $\mathcal{E}(\F)=\{V\} \cup R^S$, respectively $\mathcal{E}(\F)=\{V\} \cup Q^S$.
\end{notation}
 
It is clear that $\F(n, q, X)$ is a fusion subsystem of $\F^*(n, q, X)$ for $X\in \{R,Q\}$.

\begin{proposition}\label{t:fstarsat1}
The fusion systems $\F^*(n,q,Q)$ and $\F^*(n,q,R)$ are saturated and respectively contain $\F(n, q, Q)$ and $\F(n,q, R)$  as subsystems of $p'$-index. Moreover, \[\mathcal{E}(\F^*_{\Lambda}(q))=\{V\}\cup R^S \mbox{ and } \mathcal{E}(\F^*(n, q, X))=\{V\}\cup X^S \mbox{ for  $X\in \{R,S\},$}\] and for all such systems $\F$, we have $\F^{frc}=\{S\} \cup \E(\F)$.
\end{proposition}
\begin{proof}
Note that neither the isomorphism type of $S$ nor the set of $\F_S(S)$-centric subgroups is affected by appending $p'$-automorphisms to automisers of subgroups in a fusion system. Furthermore, the essential subgroups remain unchanged when passing from $\F_S(P_n(q))$ to $\F_S(P^*_n(q))$, and from $\F_{N_S(X)}(P_X)$ to $\F_{N_S(X)}(P_X^*)$ where $X\in\{R, Q\}$. Therefore, the proof of \cite[Theorem 4.9]{ClellandParker2010}  applies to this situation too. Since $O^{p'}(\Aut_{\F^*(n,q,X)}(X))=O^{p'}(\Aut_{\F(n,q,X)}(X))$ for $X\in\{V, R, Q\}$, it quickly follows from \cite[Lemma I.7.6]{AKO} that $\F(n, q, X)$ has $p'$-index in $\F^*(n, q, X)$ for $X\in\{R, Q\}$. The remaining points follow easily.
\end{proof}

We have that $\F^*(n, p, X)=\F(n,p, X)$ for $X\in\{R, Q\}$. Let $q>p$. One can show that $\F^*(1, q, R)\cong \F_S(O^p(\Gamma\mathrm{L}_3(q)))$, $\F^*(2, q, Q)\cong \F_S(O^p(\Aut(\KK))\ltimes\mathrm{GSp}_4(q))$ and that $\F^*(n, q, X)$ is exotic in every other case, following the same proof as \cite[Section 5]{ClellandParker2010}.

\begin{remark}
The authors expect that one can replace the group $O^p(\Aut(\KK))$ with the full automorphism group $\Aut(\KK)$ in the construction of $F^*(n, \KK,X)$ for $X \in \{R,Q\}$ to make amalgams with a larger Sylow $p$-subgroup. By adapting the methods used in \cite[Theorem 4.9]{ClellandParker2010} and Proposition \ref{UPV saturated} we also expect that one can show these extended systems are saturated (note that they possess a different collection of centric and essential subgroups), but we do not pursue that here.
\end{remark}

 Now observe that by \cite[Lemma 6.4]{parkersemerarocomputing} the fusion systems $\F^*(n ,q, Q)$ and $\F^*(n ,q, R)$ determine further systems via a process known as ``pruning''. More precisely, a subset of the automorphisms of the essential subgroup $V$ can be removed in either case to form, respectively, two new saturated fusion systems.
  
\begin{notation}\label{n:pruned}
Define the saturated fusion systems $\F^*(n ,q, Q)_P$ and $\F^*(n ,q, R)_P$ on $S_n(q)$ to be those obtained by pruning $V$ from $\F^*(n ,q, Q)$ and $\F^*(n ,q, R)$. Then \[\E(\F^*(n ,q, Q)_P)=Q^S \mbox{ and } \E(\F^*(n ,q, R)_P)=R^S\] (note that closely related examples were previously considered in \cite[Lemma 5.10]{henke2023punctured}).

Analogously we write $\F_\Lambda^*(q)_P$ for the system obtained from $\F_\Lambda^*(q)$ by pruning $\Lambda(q)$ (see  Notation \ref{n:f*lambda})  so that \[\E(\F_\Lambda^*(q)_P)=R^S.\] We refer to these systems as \emph{pruned}.
 \end{notation}
 
\begin{proposition}\label{t:fstarsat12}
Suppose that $1<n\leq p-1$ and $\F=\F^*(n,q,R)_P$. Then the following hold:
\begin{itemize}
\item[(1)] $O^{p'}(\F)$ is exotic;
\item[(2)] $|\Out_{\F}^0(S)|=q-1$; and
\item[(3)] $O^{p'}(\F)$ is simple unless $n=p-2$ and $q=p>3$, in which case \[O^p(O^{p'}(\F))=O^p(O^{p'}(\F^*(p-2,p,R)_P))\cong O^{p'}(\F^*(p-3, p, R)_P)\] is simple.
\end{itemize}
\end{proposition}
\begin{proof}
We observe that $R\in \mathcal{E}(\F^*(n, q, R)_P)$ and, as $n>1$, we see that $R\not\normaleq S$. Since $O^{p'}(\Aut_{\F^*(n,q,R)}(R))\cong \SL_2(q)$ acts irreducibly on $R$, \cite[Proposition I.4.5]{AKO} implies that $O_p(\F)=1$. We observe that ${\F^*(n, q, R)_P}^{frc}=\{S\}\cup R^S$ and so, in order to calculate $\Aut_{\F^*(n,q,R)_P}^0(S)$ via \cite[Theorem I.7.7]{AKO}, it suffices to lift elements only from $O^{p'}(\Aut_{\F^*(n,q,R)_P}(R))\cong \SL_2(q)$. Hence, $|\Out_{\F^*(n,q,R)_P}^0(S)|=q-1$.

Let $X$ be a proper non-trivial strongly closed subgroup of $\F^*(n, q, R)_P$. Then $X\normaleq S$ so that $1\ne X\cap Z(S)$ so that $1\ne X\cap R$. Again, as $O^{p'}(\Aut_{\F^*(n,q,R)}(R))$ acts irreducibly on $R$, we deduce that $R\le X$ so that $X[V, S]=\langle X^S\rangle R$. Let $k$ be an element of $\Aut^0_{\F^*(n,q,R)_P}(S)$ of order $q-1$. Then we may arrange that $k$ is the lift of an element of $N_{O^{p'}(\Aut_{\F^*(n,q,R)_P}(R))}(\Aut_S(R))$ and so acts as $\lambda^{-1}$ on $S/V$ and as $\lambda$ on $Z(S)$ for some $\lambda\in \KK^*$.  Recall the map $\delta$ from Notation \ref{n:deltamap}.  Using the injectivity of $\delta$ when restricted to $p'$-elements from Lemma \ref{lem: kerneldescription}, such an element acts in the same manner as an element of $\Aut_{P^*_n(q)}(S)$ which may be represented in the form $\left(1, \lambda, \left(\begin{matrix} 1 & 0 \\ 0 & \lambda \end{matrix}\right) \right)$. Hence, $k$ acts on $V/[V, S]$ as $\lambda^{n+1}$, so that $k$ is irreducible on $V/[V, S]$. Thus, we deduce that $X=R[V, S]$ is the unique proper non-trivial strongly closed subgroup of $\F^*(n, q, R)_P$.

Now if $\mathcal{N}\normaleq O^{p'}(\F^*(n,q, R)_P)$ is supported on $R[V, S]$ then $\Aut_{\mathcal{N}}(R[V, S])\normaleq \Aut_{\F^*(n,q, R)_P}(R[V, S])$. It follows that $\Out_S(R[V, S])$ must commute with $\langle k|_{R[V, S]}\rangle\Inn(R[V, S])/\Inn(R[V, S])$, from which we deduce that $k$ centralises $V/[V, S]$. This occurs only if $q=p>3$ and $n=p-2$ and in this case, $\mathcal{N}=O^p(O^{p'}(\F^*(p-2, p, R)_P))\cong O^{p'}(\F^*(p-3, p, R)_P)$ using that $R[V, S]\cong S_{n-1}(q)$ by Lemma \ref{subgroup similarity}, and then applying Theorem \ref{thm: main}. In every other instance, we have that $O^{p'}(\F^*(n,q,R)_P)$ is simple and no normal subsystem of $\F^*(n, q, R)_P$ is supported on $R[V, S]$. Applying \cite[Theorem 3.12]{vbexotics}, we see that $O^{p'}(\F^*(n,q, R)_P)$ is exotic unless perhaps $q=p>3$ and $n=p-2$. But in this case, if $O^{p'}(\F^*(p-2,p, R)_P)$ was realised by $G$ then $O^p(O^{p'}(\F^*(p-2,p, R)_P))$ is a saturated fusion system supported on $O^p(G)\cap S\in\syl_p(O^p(G))$ by the hyperfocal subgroup theorem of Puig \cite[\S1.1]{PuigHyper}. By Theorem \ref{thm: main}, $\F_{O^p(G)\cap S}(O^p(G))\cong O^{p'}(\F^*(p-3,p, R)_P)$ is exotic, an obvious contradiction. Hence, $O^{p'}(\F^*(p-2,p, R)_P)$ is also exotic.
\end{proof}

\begin{definition}
A saturated fusion system $\F$ is a \emph{polynomial fusion system} if $\F$ is a core-free subsystem of $p'$-index in any of $\F^*(n,q,R)$, $\F^*(n,q,Q)$, $\F^*(n,q,R)_P$ or $\F_\Lambda^*(q)$, where $1\le n<p$ and $q>p$.
\end{definition}

We note that $\F^*(1, q, R)\cong \F_S(G)$ where $G\cong O^p(\mathrm{P}\Gamma \mathrm{L}_3(q))$, $O_p(\F^*(1, q, R)_P)\ne 1$, and $\F^*(2, q, Q)\cong \F_S(G)$ where $G\cong O^p(\Aut(\PSp_4(q)))$.

We close this section with some remarks about quotient systems and subsystems of the pruned fusion systems described above.

Proposition \ref{t:fstarsat2} records the sense in which the classes of pruned fusion systems are ``closed'' under taking certain quotients. The proof utilises Theorem \ref{thm: main}.

\begin{proposition}\label{t:fstarsat2}
The following statements hold:
\begin{itemize}
\item[(1)] $O_p(\F^*(n,q,Q)_P)=Z(S_n(q))$ and $\F^*(n,q,Q)_P/Z(S_n(q)) \cong \F^*(n-1,q,R)_P$.
\item[(2)] $\F^*_\Lambda(q)_P/C_R(O^{p'}(\Aut_{\F^*_\Lambda(q)_P}(R))) \cong \F^*(p-1,q,R)_P$.
\end{itemize}
\end{proposition}
\begin{proof}
Suppose $\F= \F^*(n,q,Q)_P$. Then \cite[Proposition I.4.5]{AKO} implies that $O_p(\F)=Z(S)$ and the resulting quotient fusion system $\F/Z(S)$ is saturated with essential subgroups $(Q/Z(S))^S$. In particular, $O_p(\F/Z(S))=1$. Now by Lemma \ref{Snauto}, $S_n(q)/Z(S_n(q))\cong S_{n-1}(q)$ and so $\F/Z(S) \cong \F^*(n, q, R)_P$ by Theorem \ref{thm: main}, and (1) is proved.

Suppose that $\F= \F^*_\Lambda(q)_P$. Observe that $Z:=\langle \overline{y^p}\rangle_{\KK} = C_R(O^{p'}(\Aut_\F(R)))\le Z(S)$ and so \[C_R(O^{p'}(\Aut_\F(R)))=C_{R^s}(O^{p'}(\Aut_\F(R^s)))\] for all $s\in S$. Applying \cite[Proposition I.4.5]{AKO} again, we conclude that $O_p(\F)=Z$ and $\F/Z$ is saturated with essential subgroups $(R/Z)^S$. Indeed, $O_p(\F/Z)=1$. Now by Lemma \ref{lem: LambdaIso}, $S_{\Lambda}(q)/Z\cong S_{p-1}(q)$ and so $\F/Z \cong \F^*(p-1,q,R)$ by Theorem \ref{thm: main}, proving (2).
\end{proof}

Taking inspiration from the theory of \emph{normaliser towers} (see \cite[Theorem 3.6]{pearls}), we also observe the following subsystem closure property enjoyed by the fusion systems $\F^*(n,q,R)_P$.

\begin{proposition}\label{normalizertower}
Set $\F=\F^*(n, q, R)_P$, and let \[N^i:=RZ_{i+1}(S) \mbox{ and } \mathcal{N}^i=\langle \Aut_{\F}(R), \Inn(N^i)N_{\Aut_{\F}(N^i)}(R)\rangle_{N^i}\] for $1<i\leq n$. Then $\mathcal{N}^i$ is a saturated fusion subsystem of $\F=\mathcal{N}^n$ with $O_p(\mathcal{N}^i)=1$ and $\mathcal{E}(\mathcal{N}^i)=\{R^{N^i}\}$. Moreover, $\mathcal{N}^i$ is isomorphic to $\F^*(i, q, R)_P$.
\end{proposition}
\begin{proof}
To prove that $\mathcal{N}_i$ is saturated we adopt the same method as in the proof of \cite[Lemma 3.7]{pearls}, which ultimately relies on \cite[Theorem 5.1]{BLO}. The equalities  $O_p(\mathcal{N}^i)=1$ and $\mathcal{E}(\mathcal{N}^i)=\{R^{N^i}\}$ are both clear by construction and applying \cite[Proposition I.4.5]{AKO}. All that remains is to show that $\mathcal{N}^i \cong \F^*(i, q, R)_P$, but this is immediate from the fact that $N^i \cong S_i(q)$ (by Lemma \ref{(p+1)CVS}(6)) and our main Theorem \ref{thm: main}.
\end{proof}

\begin{remark}
One can form analogous normaliser towers for the fusion systems $\F^*(n, q, Q)_P$ and $\F^*_{\Lambda}(q)_P$. In these cases, the $p$-core of the largest system is contained in all normaliser systems, and quotienting by this produces the normaliser tower for $\F^*(n, q, R)_P$ described in Proposition \ref{normalizertower} (see Proposition \ref{t:fstarsat2}).
\end{remark}

\begin{remark}
We note that the subsystems of $p'$-index in $\F^*(n, q, R)$ and $\F^*(n, q, R)_P$ have a \textit{punctured group} in the sense of \cite{henke2023punctured}, while the subsystems of $p'$-index in $\F^*(n, q, Q)$ do not have this property when $n\geq 3$. Also since $N_{\F^*_{\Lambda}(q)}(Z)$ is an exotic subsystem of $\F^*_{\Lambda}(q)$, an argument similar to that used to prove \cite[Proposition 5.11]{henke2023punctured} shows that $\F^*_{\Lambda}(q)$ also does \emph{not} have a punctured group.
\end{remark}

\section{Recognition and uniqueness results for groups and modules}\label{sec:reg}

Having described the polynomial $p$-groups and polynomial fusion systems of interest in this paper, we now embark on the proof that, other than the two exotic Henke--Shpectorov systems, these systems are the unique core-free saturated fusion systems supported on $S$, for $S\in\{S_n(q), S_\Lambda(q)\mid 1\leq n\leq p-1, q\ne p\}$. For this, we require a variety of group and module results, and this section is reserved for detailing them.

The following proposition, which depends on the Classification of Finite Simple Groups, will be used to restrict the possibilities for $\Aut_\F(E)$ when $E$ is $\F$-essential in the fusion system $\F$.

\begin{proposition} \label{Strongly p-embedded Sylows}
Suppose that $p$ is a prime, $X$ is a group and $T \in \Syl_p(X)$ is elementary abelian of order $p^m$ with $m\ge 2$. Assume that $X$ has a strongly $p$-embedded subgroup. Then $O_p(X)=1$ and for $K=O^{p'}(X)=\langle T^K\rangle$, $\bar{K}:=K/O_{p'}(K)$ is a non-abelian simple group and one of the following holds:
\begin{enumerate}
\item $p$ is arbitrary, $m \ge 2$, $\bar{K}\cong \PSL_2(p^m)$ and $|T|= p^m$.
\item $p>3$, $\bar{K}\cong \Alt(2p)$ and $T$ is elementary abelian of order $p^2$.
\item $p=3$, $\bar{K} \cong \PSL_3(4)$ and $T$ is elementary abelian of order $3^2$.
\item $p=3$, $\bar{K}\cong \mathrm{M}_{11}$ and $T$ is elementary abelian of order $3^2$.
\item $p=5$, $\bar{K}\cong {}^2\mathrm F_4(2)^\prime$ and $T$ is elementary abelian of order $5^2$.
\item $p=5$, $\bar{K}\cong \mathrm{Fi}_{22}$ and  $T$ is elementary abelian of order $5^2$.
\end{enumerate}
\end{proposition}
\begin{proof}
Since $X$ has a strongly $p$-embedded subgroup, so too does $K=O^{p'}(X)$. Then, as $K$ has a strongly $p$-embedded subgroup, so   does $\bar{K}=K/O_{p'}(K)$. Then $\bar{T}\cong T$ and taking into account that $T$ is abelian and not cyclic, we may apply \cite[Chapter 4, Proposition 10.2 and Lemma 10.3]{GLS4} (but also see \cite[Proposition 4.5]{parkersemerarocomputing}).
\end{proof}

We now record a number of results to aid in identifying modules for groups with strongly $p$-embedded subgroups satisfying certain conditions relevant to the groups $S_n(q)$.

\begin{lemma}\label{lem: Smith's lemma}
Suppose that $G\cong \SL_2(q)$, $q=p^m$, $U \in \syl_p(G)$ and $V$ is an irreducible $\FF_pG$-module. Let $\mathbb{L}$ be the largest field extension of $\FF_p$ over which $V$ is writable as an $\mathbb{L}$-module. Then both $C_V(U)$ and $V/[V, U]$ are irreducible $1$-dimensional $\mathbb{L}N_G(U)$-modules upon restriction.
\end{lemma}
\begin{proof}
See \cite[Theorem 2.8.11]{GLS3}.
\end{proof}

\begin{lemma}\label{lem: SL2split}
Suppose that $G\cong \SL_2(q)$, $q=p^m$, $U \in \syl_p(G)$ and $V$ is an irreducible $\FF_pG$-module. Let $\mathbb{L}$ be the largest field extension of $\FF_p$ for which $V$ is writable as an $\mathbb{L}$-module. Then $|\mathbb{L}|\leq q$. In particular, $\dim_{\FF_p}(C_V(U))=\dim_{\FF_p}(V/[V, U])\leq m$.
\end{lemma}
\begin{proof}
By \cite[pg. 236]{Steinberg}, a splitting field for $C_{p^m-1}$ is a splitting field for $G$. The result then follows from Lemma \ref{lem: Smith's lemma}.
\end{proof}

In the next result we view $V_i(q)$ as a module over $\KK=\FF_q$. Later it will be regarded as a module over $\FF_p$.

\begin{lemma}\label{lem:mod ident}
Suppose that $G\cong \SL_2(q)$, $U \in \syl_p(G)$ and $V$ is an irreducible $\KK G$-module. If $\dim_\KK V=k+1$ and $[V,U;k]\ne 1$, then $V\cong V_k(q)^\theta$ for some $\theta\in \Aut(\KK)$.
\end{lemma}
\begin{proof}
By the Steinberg Tensor Product Theorem \cite[Corollary 2.8.6]{GLS3} we may write \[V=W_0\otimes W_1^\sigma\otimes \dots\otimes W_{a-1}^{\sigma^{a-1}}\] where each $W_i\cong V_j(q)$ for some $j\in\{0, \dots, p-1\}$ and $\sigma$ is the Frobenius automorphism which maps $\lambda$ to $\lambda^p$ for all  $\lambda \in \KK$. Set $d_j=\dim_\KK W_j$. Then \[k+1 =\dim_\KK V= \prod_{j=0}^{a-1}d_j.\]

Define $P(n)= 1+ \dots + z^n \in \mathbb{Z}[z]$ and put $P(W_j)= P(d_j-1)$. Thus $P(V)=\prod_{j=0}^{a-1} P(W_j)$ is a monic polynomial of degree $m = \sum_{j=0}^{a-1}(d_j-1)$. By \cite[Proposition 1 and Corollary 3]{ParkerRowley2005}, $\dim _\KK [V,U;m]=1$ and so we infer that $m=k$. Hence \[k+1=\prod_{j=0}^{a-1}d_j= m+1= \left(\sum_{i=0}^{a-1}(d_j-1)\right) +1.\] It follows that there is a unique $0\leq \ell \leq a-1$ such that $d_j\ne 1$ and so $V=V_{d_j-1}(q)^{\sigma^\ell}$. This proves the result.
\end{proof}

\begin{lemma}\label{lem: SEFF}
Suppose that $G$ has a strongly $p$-embedded subgroup, $p$ is an odd prime, $U\in\syl_p(G)$ and $V$ is an $\FF_pG$-module with $[V, O^p(G)]\ne 1$. If $|U|\geq |V/C_V(U)|$ or $|U|\geq |[V, U]|$ then $G/C_G(V)\cong \SL_2(|U|)$, $C_G(V)$ is a $p'$-group, $V/C_V(G)\cong V_1(q)|_{\FF_p}$ and $|U|=|V/C_V(U)|=|[V, U]|$.
\end{lemma}
\begin{proof}
This follows from \cite[Theorem 5.6]{Henke} (c.f. \cite[Theorem 1]{Henke}).
\end{proof}

We call an $\FF_p\SL_2(q)$-module a \emph{natural} $\SL_2(q)$-module if it isomorphic to $V_1(q)|_{\FF_p}$.

\begin{lemma}\label{lem:NoExt1}
Suppose that $G\cong \SL_2(q)$, $q=p^m>p$, and  $V$ is an indecomposable $\KK G$-module with composition factors $V_k(q)^{\sigma^a}$ and $V_{\ell}(q)^{\sigma^b}$ where $k,\ell\in\{0, \dots, p-1\}$, $\sigma$ is the standard Frobenius automorphism of $\KK$ defined by $s\mapsto s^p$ and $a\leq b$. Then $b=a+1$, $k=p-2$ and $\ell=1$. In particular, $\dim_{\KK} V=p+1$.
\end{lemma}
\begin{proof}
As $q=p^m>p$, we may use \cite[Corollary 4.5]{andersenjorgsenlandrock}. In this context we have $V_k(q)^{\sigma^a}$ has weight $\lambda= kp^a=\sum _{i=0}^{m-1}\lambda_ip^i$ and $V_{\ell}(q)^{\sigma^b}$ has weight $\mu = \ell p^b =\sum _{i=0}^{m-1}\mu_ip^i$, using the conventions that $\lambda_m= \lambda_0$ and $\mu_m=\mu_0$. Then $k=\lambda_a$ and $\lambda_{a+1}=0$, and hence $\mu_{a+1} \in\{1, p-1\}$ which means that $b=a+1$ and $\ell \in\{1, p-1\}$. In addition, we have $\mu_a=0$. Therefore \[k=\lambda_a= p-\mu_{a}-2=p-2\] and this completes the proof.
\end{proof}

\begin{proposition}\label{prop: indiden}
Suppose that $G\cong \SL_2(q)$, $q=p^m>p$, and $V$ is an indecomposable $\FF_pG$-module. Assume the following:
\begin{enumerate}
    \item there is $1\leq k,l\leq p-1$ and a submodule $W$ of $V$ such that $V/W\cong V_k(q)|_{\FF_p}$ and $W\cong V_l(q)|_{\FF_p}$;
    \item $\dim_{\FF_p}(V)\leq m(p+1)$; and
    \item if $p=3$ then for $U\in\syl_3(G)$, $U$ does not act quadratically on $V$ and either $\dim_{\FF_3}(C_V(U))=2m$ or $\dim_{\FF_3}(V/[V, U])=2m$.
\end{enumerate}
Then $V\cong V_p(q)|_{\FF_p}$ or $V\cong \Lambda(q)|_{\FF_p}$.
\end{proposition}
\begin{proof}
Consider the module $Y:=V\otimes_{\FF_p} \KK$. Applying an extension of the arguments in \cite[(26.2)]{AschbacherFG}, we see that there is a $\KK$-submodule $R$ of $Y$ such that $R\cong V_l(q)|_{\FF_p}\otimes_{\FF_p} \KK$ is a direct sum of $m$ Frobenius twists of $V_l(q)$; and $Y/R\cong V_k(q)_{\FF_p}\otimes_{\FF_p} \KK$ is  a direct sum of $m$ Frobenius twists of $V_k(q)$. Indeed $R= W\otimes_{\FF_p} \KK$ may be regarded as a submodule of $Y$.

Let $E$ be the socle of $Y$. Then $E \ge R$. By \cite[(25.7)(1)]{AschbacherFG}, regarded as a $\FF_p$-space, $Y$ admits $\Gamma \times G$ where $\Gamma=\Gal(\KK/\FF_p)$. Assume $Z \le E$ is an  irreducible $\KK G$-submodule of $Y$. Then for $\sigma \in \Gamma$, $Z\sigma$ is also an irreducible $\KK G$-submodule. Set $D= \langle Z^\sigma\mid \sigma \in \Gamma\rangle$. Then $D$ is  $(\Gamma \times G)$-invariant and has dimension $m\dim_{\FF_p} Z < \dim_{\FF_p} Y$. Now \cite[(25.7)(2)]{AschbacherFG} implies that $W= U\otimes_{\FF_p} \KK$ for some submodule  $U$ of $V$. Since the only proper submodule of $V$ is $W$, we conclude that $Z\le R$ and $E=R$. Thus $R$ is the socle of $Y$.

It remains to examine when $\mathrm{Ext}_{\SL_2(q)}^1(V_k(q)^{\sigma^a}, V_l(q)^{\sigma^b})\ne 0$. We have by Lemma \ref{lem:NoExt1} that if $\mathrm{Ext}_{\SL_2(q)}^1(V_k(q)^{\sigma^a}, V_l(q)^{\sigma^b})$ is non-trivial then either $k=p-2, l=1$ and $b=a+1$; or $k=1$, $l=p-2$ and $b=a-1$. Furthermore, \cite[Corollary 4.5]{andersenjorgsenlandrock} implies that $\mathrm{Ext}_{\SL_2(q)}^1(V_k(q)^{\sigma^a}, V_l(q)^{\sigma^b})$ is $1$-dimensional unless $q=9$ and $l=k=1$. If $p \ne 3$, we now have the result.

Assume now that $p=3$. Then if  $\dim_{\FF_3}(C_V(U))=2m$ we apply \cite[Lemma 3.4]{NilesPushingUp} to see that $V$ is determined up to isomorphism. Since $\Lambda(q)|_{\FF_3}$ also satisfies the hypothesis, the result follows in this case. Otherwise we have that $\dim_{\FF_3}(V/[V, U])=2m$ and applying \cite[Lemma 3.4]{NilesPushingUp} to the dual of $V$ implies that $V\cong V_3(q)|_{\FF_3}$.
\end{proof}

\begin{lemma}\label{lem:NoExt}
Suppose that $G\cong \SL_2(q)$ where $q=p^m>p$ and $p$ is an odd prime, and let $S\in \syl_p(G)$ and $1\leq k\leq p-1$. Then
\begin{enumerate}
    \item if $V$ is an $\FF_pG$-module with $V/C_V(G)$ isomorphic to $V_k(q)|_{\FF_p}$ then $C_V(S)/C_V(G)=C_{V/C_V(G)}(S)$; and
    \item if $V$ is an $\FF_pG$-module with $[V, G]$ isomorphic to $V_k(q)|_{\FF_p}$ then $[V, S]=[[V, G], S]$.
\end{enumerate}
\end{lemma}
\begin{proof}
Assume that $G,V$ satisfy the hypothesis of part (1) of the lemma and that $V$ is chosen such that $V$ does not satisfy the outcome of (1) and $|V|$ is minimal. Suppose that there are $W_1, W_2<C_V(G)$ with $|W_1|=p=|W_2|$. Consider $V/W_i$. Then by minimality, $C_{V/C_V(G)}(S)\cong C_{(V/W_i)/C_{V/W_i}(G)}(S)=C_{V/W_i}(S)/(C_V(G)/W_i)$. Hence, writing $R$ for the preimage in $V$ of $C_{V/C_V(G)}(S)$, we have that $[R, S]\le W_1\cap W_2=1$ so that $C_V(S)/C_V(G)=C_{V/C_V(G)}(S)$, a contradiction.

Thus, we may as well assume that $|C_V(G)|=p$. Consider the module $W:=V\otimes_{\FF_p} \KK$. Then by \cite[(26.2)]{AschbacherFG}, $|C_W(G)|=q$ and $W/C_W(G)$ is a direct product of $m$ Frobenius twists of $V_k(q)$. Moreover, by \cite[(27.12)]{AschbacherFG}, $C_W(G)=C_W(S)$. Let $W'$ be the preimage of a summand of $W/C_W(G)$. Then $C_{W'}(G)=C_{W'}(S)$ and so to force a contradiction and prove the result, it suffices to prove that $\mathrm{Ext}_{\SL_2(q)}^1(V_k(q)^{\sigma^a}, V_0(q))$ is trivial for $\sigma$ the Frobenius automorphism on $\KK$ and $a\in \{0, \dots, m-1\}$. But this is true by \cite[Corollary 4.6]{andersenjorgsenlandrock}, and so (1) holds.

The proof of (2) follows by dualisation arguments.
\end{proof}

\begin{lemma}\label{lem:dimV}
Suppose that $T$ is a $p$-group and $V$ is a faithful $\mathbb L T$-module, where $\mathbb L$ has characteristic $p$.  If $2\leq k \leq p$ and \[[V,t;k-1]=  C_V(t)\] for some $t \in T^\#$, then $|V|= |C_V(t)|^k$ and the matrix representing $t$ has exactly $\dim_{\mathbb L} C_V(t)$ Jordan blocks each of dimension $k$.
\end{lemma}
\begin{proof}
Let $t \in T^\#$ be the given element. The fact that $[V,t;k-1]= C_V(t),$ implies that the Jordan form of a matrix representing $t$ has $\dim_{\mathbb L} C_V(t)$ blocks each of size $k$. In particular, $t$ has order $p$, and $|V|= |C_V(t)|^k$.
\end{proof}

\begin{lemma}
Suppose that $T$ is a $p$-group and $V$ is a faithful $\mathbb L T$-module, where $\mathbb L$ has characteristic $p$. If $k\ge 2$ and \[[V,T;k-1]=[V,r;k-1]= C_V(T)=C_V(r)\] for all $r \in T^\#$, then $[V,T;\ell]=[V,r;\ell]$ for all $r \in T^\#$ and $\ell \ge 0$. Furthermore, $[V^*,T;k-1]=[V^*,r;k-1]= C_{V^*}(T)=C_{V^*}(r)$ for all $r \in T^\#$ where $V^*$ is the dual of $V$.
\end{lemma}
\begin{proof}
By Lemma \ref{lem:dimV}, $|V|= |C_V(r)|^k$ and $r$ has $\dim_{\mathbb L} C_V(t)$ Jordan blocks of size $k$. The result is true for $\ell =k$ and $\ell=k-1$ by hypothesis. Suppose that it is true for $0<\ell \leq k-1$. Then $[V,T;\ell]= [V,r;\ell]$ for all $r \in T^\#$. By definition \[[V,T;\ell-1]= \langle [V,r;\ell-1]\mid r \in T^\#\rangle.\]

Aiming for a contradiction, assume that $r, s \in T^\#$ and $[V,r;\ell-1] \ne [V,s;\ell-1]$. Then \[[V,r;\ell-1][V,s;\ell]/[V,s;\ell] \not \le C_{V/[V,s;\ell]}(s)=[V,s;\ell-1]/[V,s;\ell]\] so that $[[V,r;\ell-1],s] \not \le [V,s;\ell]$. Hence \[[V,s;\ell]=[V,T;\ell]\ge [[V,T;\ell-1], s][V,s;\ell] \ge [[V,r,\ell-1],s][V,s;\ell]  > [V,s;\ell],\] which is impossible.
\end{proof}

\begin{lemma}\label{lem: spe modules}
Suppose that $p$ is a prime, $X$ is a group with a strongly $p$-embedded subgroup, $K=O^{p'}(X)$ and $T \in \Syl_p(X)$ is elementary abelian of order at least $p^2$. Assume that  $2 \leq k \leq p$ and $V$ is a faithful $\FF_pX$-module. If, for all $t \in T^\#$, \[[V, T; k-1]=[V, t; k-1]=C_V(T)=C_V(t)\] has order $|T|$, then $K$ is quasisimple and either
\begin{enumerate}
    \item $p$ is arbitrary, $K\cong \SL_2(p^m)$ and if $V$ is irreducible then $V|_K=V_n(p^m)$ considered as a $\FF_pK$-module for some odd $n$ with $1\leq n\leq \max\{1, p-2\}$;
    \item $p$ is arbitrary, $K \cong \PSL_2(p^m)$ and if $V$ is irreducible then $V|_K=V_n(p^m)$ considered as a $\FF_pK$-module for some even $n$ with $2\leq n\leq p-1$; or
    \item $p=3$, $K\cong 2^.\PSL_3(4)$ and $V|_K$ is the unique irreducible $\FF_3K$-module of dimension $6$.
\end{enumerate}
\end{lemma}
\begin{proof}
From \cite[Lemma 2.8]{LSWP} and the hypothesis that $C_V(T)= C_V(t)$ for all $t \in T^\#$, we see that $T$ centralises $O_{p'}(K)$ and therefore so does $K$. It follows from Proposition \ref{Strongly p-embedded Sylows} that $K$ is quasisimple.

Suppose that Proposition \ref{Strongly p-embedded Sylows}(1) holds and therefore $K/Z(K) \cong \PSL_2(p^m)$. Let $\mathbb L= \End_{\FF_p}(V)$ and regard $V$ as an $\mathbb LG$-module so that $V$ is an absolutely irreducible $\mathbb LG$-module. By Lemma \ref{lem: Smith's lemma}, $\dim_{\mathbb L} C_V(T)=1$.  Since $C_V(t)= C_V(T)$ for $t\in T^\#$ we have $\dim_{\mathbb L} C_V(t)=1$ and consequently $V|_{\langle t\rangle}$ is indecomposable. It follows that $t$ is represented on $V$ by a matrix with a single Jordan block.  Hence, writing $\dim_{\mathbb L} V=k+1$, we have $[V,t;k]=[V,T;k]\ne 1$ and $[V, t; k]$ is $1$-dimensional as an $\mathbb L$-module. Now Lemma \ref{lem:mod ident} implies that $\KK=\mathbb L$ and as a $\mathbb{L}G$-module, $V\cong V_k(q)^{\sigma^\ell}$ for some $\ell$ and $\sigma$ the standard Frobenius automorphism. Hence (1) or (2) holds in this case.

In the case of Proposition \ref{Strongly p-embedded Sylows} (2), we have $K/Z(K) \cong \Alt(2p)$ with $p \ge 5$ and $|V|\leq p^{2p}$. Suppose first that $Z(K)=1$. Since $p \ge 5$, \cite{Wagner1} together with \cite[Theorem 7]{James83} (see \cite[Corollary 16.9]{ParkerRowleySymplectic}) imply that $V$ contains a unique non-trivial irreducible constituent isomorphic to the heart of the natural permutation module of dimension $2(p-1)$. Hence, $|V|\geq p^{2p-2}$ and we conclude that every $t\in T^\#$ has two Jordan blocks on $V$ and both are non-trivial. But then an elementary calculation shows that the $p$-cycles and the products of two $p$-cycles have fixed spaces of different dimension.

Now suppose that $Z(K)\ne 1$. Then, as $p \ge 5$, $|Z(K)|= 2$. Write $2p= 2^{w_1}+ \dots +2^{w_s}$ with $w_1> w_2> \dots >w_s=1$. Then
\begin{equation*}
2p \ge \dim V \ge 2^{\lfloor (2p-s-1)/2\rfloor} \tag{$\ast$}
\end{equation*}
by \cite[1.3 Theorem (2)]{Wagner2}. Since $s\leq \log_2(2p)=1+\log_2(p)$, we have \[2p \ge 2^{\lfloor (2p-s-1)/2\rfloor}\ge 2^{\lfloor (2p-\log_2(p)-2)/2\rfloor}\ge 2^{(p-2)/2}\] which is impossible for $p\ge 13$. For $p \in \{7,11\}$, we check that $(\ast)$ does not hold. For $p=5$, $2^.\Alt(10)$ has two $8$-dimensional representations and as $|V|\leq 5^{10}$ and $Z(G)$ acts non-trivially on $V$, we conclude that $V$ is either irreducible or decomposable. Since $|C_V(s)|=5^2$, it is easy to see that $V$ is irreducible. We now verify by MAGMA \cite{magma} that these representations do  not satisfy the hypothesis of the lemma. The remaining faithful $\FF_5$-representations of $2^.\Alt(10)$ have dimensions at least $28$. Thus we have no examples here.

Consider possibilities (3) and (4) so that $p=3$ and $|T|=9$. In these situations we have that $|V|\leq 3^6$ and comparing with \cite{ModAt} we have that $K\cong 2^.\PSL_3(4)$ or $K\cong \mathrm{M}_{11}$. In the former case, $K$ has a unique non-trivial module of dimension at most $6$. This module has dimension exactly $6$ and is listed as outcome (3).

In the latter case, we note that $V$ has a unique non-trivial irreducible constituent, of dimension $5$, and this constituent is either the ``code" and ``cocode" module for $\mathrm{M}_{11}$. We have $J\cong \PSL_2(9)$, regarded as the derived subgroup of $\mathrm M_{10}$, is a subgroup of $K$.  We may assume that $T \le J$ and that $(V,J)$ satisfies the hypothesis of the lemma. But then (2) holds, and we deduce that $V\cong V_2(9)$ which is irreducible of dimension $6$. This is impossible given that $K$ has a composition factor of dimension $5$ on $V$.

Finally  consider cases (5) and (6) of Proposition \ref{Strongly p-embedded Sylows}. Then $\dim V \leq 10$. Hence $K$ must embed into $\SL_{10}(5)$ and, by considering the orders of the groups involved, we have $K \cong {}^2\mathrm F_4(2)'$. However, by \cite{ModAt}, the smallest characteristic $5$ representation of this group has dimension $26$. Hence these cases do not occur.

This concludes the proof.
\end{proof}

\begin{remark}
Note that when $K\cong 2^.\PSL_3(4)$ and $|V|=3^6$, any extension of $K$ over $V$ splits and a Sylow $3$-subgroup of the semidirect product $V\rtimes K$ is isomorphic to a Sylow $3$-subgroup of $\PSp_4(9)$, and therefore to the group $S_2(9)$ defined above.
\end{remark}

\begin{remark}
If $\F$ is a saturated fusion system with a weakly $\F$-closed  elementary abelian subgroup $E$ which, as a module for $O^{3'}(\Out_{\F}(E))\cong \mathrm{M}_{11}$, is isomorphic to either the code or cocode module for $M_{11}$, then $\F$ is known by results in \cite{BobTodd}. It is our intention to treat the analogous situation in which $E$ is a module for $\Alt(2p)$ in a later paper.
\end{remark}

The remainder of this section is concerned with certain uniqueness results needed to show that the polynomial fusion systems are uniquely determined up to isomorphism. The proofs are long and technical. The reader is encouraged to work through them slowly, and may consider skipping some of the details on a first pass.

The irreducibility requirement in parts (1) and (2) of  Lemma \ref{lem: spe modules}  is too limiting for our purposes and this prevents  Lemma \ref{lem: spe modules} from being applied directly. To remedy this we prove following proposition which is specific to the cases we consider, first viewing the modules $V_i(q)$ over $\FF_p$.

\begin{proposition}\label{lem: Vnqrecog}
Suppose that $H$ is a finite group with $S\in\syl_p(H)$. Assume that $S\cong S_n(q)$ for $q=p^m>p$ and $1\leq n<p$, and $H/O_p(H)\cong \SL_2(q)$. Then, writing $V:=O_p(H)$ and recognizing $V$ as an $\FF_p\SL_2(q)$-module, we have that $V\cong V_n(q)|_{\FF_p}$.
\end{proposition}
\begin{proof}
We prove this by induction on $n$. For $n=1$, since $q>p$ we have that $S\le O^p(H)V$ and as $C_V(S)<V$ we have that $[V, O^p(H)]\ne 1$. Observe that $|S/V|=|V/C_V(S)|=|C_V(S)|$ and so that by Lemma \ref{lem: SEFF}, we have that $C_G(V)=V$, $H/V\cong \SL_2(q)$ and $V$ may be viewed as a natural $\SL_2(q)$-module. Thus, the result holds when $n=1$. So assume from now that $1<n<p$.

Recall from Notation \ref{n:subsofd*} $S_n(q)=U\ltimes V_n(q)$, where $U$ is elementary abelian of order $q$ and $V_n(q)$ is elementary abelian, self-centralising in $S_n(q)$ and of order $q^{n+1}$. We aim to prove that the action of $H/V\cong \SL_2(q)$ on $V$ is irreducible, for then the result holds by Lemma \ref{lem: spe modules}. We observe that $C_H(V)\normaleq H$ and that $[S, V]\ne 1$. Since $H/V\cong \SL_2(q)$ and $q>p$, we have that $H/C_H(V)\cong \SL_2(q)$ or $H/C_H(V)\cong \PSL_2(q)$. Set $Z_i=Z_i(S)$ for $0 \le i \le n$ where $Z_0=1$ and $Z_{n+1}=V$.

In pursuit of proving that $H/V$ acts irreducibly on $V$ as an $\FF_p$-module we assume, aiming for a contradiction, that there is $W\normaleq H$ with $1\ne W<V$. Choose $W$ maximal subject to this. We either have that $W\le C_V(S)=Z_1$, or there is some maximal $i\in \{2, \dots, n+1\}$ such that $W\cap Z_i\not\le Z_{i-1}$. Then Lemma \ref{lem: com full} yields $Z_{i-1}\le W < Z_i$.

\begin{claim}
$i=n+1$.
\end{claim}

Suppose otherwise so that $i\leq n$ and $W\le Z_n= [V, S]$. In particular, $[V/W, S]=[V, S]/W$ has index $q$ in $V/W$. If $W=[V,S]$, then $H$ centralises $V/W$ and this is impossible as $q>p$ and $W$ was chosen maximally. Hence $W<[V, S]$. Write $\mathbb{L}$ for the largest extension of $\FF_p$ for which $V/W$ is writable as an $\mathbb{L}\SL_2(q)$-module. Then by Lemma \ref{lem: Smith's lemma}, we have that $(V/W)/[V/W, S]\cong V/[V, S]$ is irreducible as an $\mathbb{L}N_{H/V}(S/V)$-module. Hence, by Lemma \ref{lem: SL2split} we deduce that $\mathbb{L}=\KK:=\FF_q$. But then $|V/W|=q^a$ for some $a\in \N$ and we conclude that $W=Z_i $. Now, by Lemma \ref{Snauto} we have that $S/Z_i \cong S_{n-i}(q)$ and by induction we have that $V/W\cong V_{n-i}(q)$. Since $|C_V(S)|=q$, an application of Lemma \ref{lem:NoExt} yields that $i>1$. Hence $W=Z_i $ for some $i\in\{2,\dots, n-1\}$.

Now, take $W_0<W$ with $W_0\normaleq H$ and $W_0$ maximal subject to this. Then $W/W_0$ is an irreducible module and arguing as above, we either have that $W_0\le C_V(S)=Z_1$, $Z_{i-1} \le W_0$ or $W_0=Z_k $ for some $k\in \{0, 1,\dots, i-2\}$ where $Z_0=1$. If $Z_{i-1}\le W_0$ then $W/W_0$ is a trivial module and, as $q>p$ and $W_0$ is maximal, we have \[Z_{i-1}<W_0<Z_i=W<Z_n=[V,S]\] with $V/W\cong V_{n-i}(q)$ and Lemma \ref{lem:NoExt} implies that $C_{V/W_0}(S)=Z_{i+1}/W_0$, a contradiction to Lemma \ref{lem: com full}. Thus,  $W_0<Z_{i-1}$.

If $W_0\le C_V(S)$, then $W_0\le [W, S]=Z_{i-1}$. In a similar manner to before, we apply Lemma \ref{lem: Smith's lemma} and Lemma \ref{lem: SL2split} to deduce that $|W/W_0|=q^b$ for some $b\in \N$ and we conclude that either $W_0=Z(S)$ or $W_0=1$, and so $=Z_k$ for some $k$. Thus, we lie in the third case:
\[1\le W_0=Z_k \le Z_{i-2}<Z_i=W<Z_n=[V,S].\]
We have that $S/Z_k\cong S_{n-k}(q)$ and by induction, using that $V/W_0$ is reducible, we deduce that $k=0$ and $W$ is irreducible. Applying Lemma \ref{lem:mod ident} we deduce that $W\cong V_{i-1}(q)$. As $W<[V, S]$, $V/W\cong V_{n-i}(q)$ and by Proposition \ref{prop: indiden}, using that $|V|=q^{n+1}<q^{p+1}$, we have a contradiction. This proves the claim.

We have shown that $i=n+1$ so that $[V, S]< W$ and $H$ centralises $V/W$. Then $[W, S]=[V, S]<W$ by Lemma \ref{lem: com full}. Let $k\in N_H(S)$ be of order $q-1$, set $K=\langle k\rangle$ and note that $\Aut_{K}(S)$ has order $q-1$, or $\frac{q-1}{2}$.

Observe that $k$ induces an automorphism on $S$ which induces an action on $S/[V, S, S]$. But by Lemma \ref{Snauto}, $S/[V, S, S]\cong S_1(q)$ is isomorphic to a Sylow $p$-subgroup of $\PSL_3(q)$. Since $q>p$, $k$ acts non-trivially on $S/V$ so acts non-trivially on $S/[V, S, S]$. By Proposition \ref{prop: L3Q}, $S/[V, S, S]$ is a $\KK$-space for $\Aut(S/[V, S, S])$ and we conclude that $[k, V]\le [V, S]$.

 Recall the definition of $P_n^*(q)$ from Notation \ref{n:snandpn}.  We recognise that $\Aut_{P_n^*(q)}(S_n(q))$ contains a Hall $p'$-subgroup of $\Aut(S_n(q))$ and so for $K:=\langle k\rangle$, $\Aut_K(S)$ is $\Aut(S)$-conjugate to an element of order at least $\frac{q-1}{2}$ which acts like an element of $\Aut_{P_n^*(q)}(S_n(q))$. Since $q>p$, Lemma \ref{action on centre} implies that $K$ acts irreducibly on $C_V(S)=Z_1$ as an $\FF_pK$-module.

Let $1\ne V_0\le W < V$ with $V_0$ chosen minimally such that $V_0\normaleq H$. In particular, $V_0$ is an irreducible $H/V$-module. Since $K$ acts irreducibly on $C_V(S)$, we see that $C_V(S)=\langle (C_V(S)\cap V_0)^K\rangle \le V_0$. As before, applying Lemma \ref{lem: Smith's lemma} and Lemma \ref{lem: SL2split} we conclude that $V_0$ is writable as a $\KK$-module and so $V_0=Z_i(S)$ for some $1\leq i\leq n$ by Lemma \ref{lem:mod ident}. If $i< n$, then  $S/V_0\cong S_{n-i}(q)$ and by induction we have that $V/V_0$ is irreducible which is impossible as $V_0<W$. Hence $V_0=[V,S] \cong V_{n-1}(q)|_{\FF_p}$ and $V/V_0$ is centralised by $H$. An application of Lemma \ref{lem:NoExt}(2) then yields the unforgivable $[V,S]=[V,S,S]$. This completes the proof.
\end{proof}

Our next module recognition result concerns $S_\Lambda(q)$ and $V_\Lambda(q)$.

\begin{proposition}\label{lem: upVRecog}
Suppose that $H$ is a finite group with $S\in\syl_p(H)$. Assume that $S\cong S_{\Lambda}(q)$ for $q>p$, where $p$ is an odd prime, and $H/O_p(H)\cong \SL_2(q)$. Then, writing $V:=O_p(H)$ and recognising $V$ as an $\FF_p\SL_2(q)$-module, we have that $V\cong \Lambda(q)|_{\FF_p}$.
\end{proposition}
\begin{proof}
We observe that $Z(S)=C_V(S)$ has order $q^2$ and $V/[V, S]$ has order $q$ by Lemma \ref{(p+1)CUpS}. Let $W\le V$ be such that $W$ corresponds to the $(p-1)$-dimensional $\KK\SL_2(q)$-submodule in the construction of $S_{\Lambda}(q)$. In particular, $S/W\cong S_1(q)$, $[W, S]=[V, S, S]$, $C_W(S)=C_{[V, S, S]}(S)$ and $[V, S]=WC_V(S)$.

As in Proposition \ref{lem: Vnqrecog}, we have that $H/C_H(V)\cong \SL_2(q)$ or $H/C_H(V)\cong \PSL_2(q)$ and we let $k\in N_H(S)$ be of order $q-1$ so that $\Aut_{\langle k\rangle}(S)$ has order $q-1$ or $\frac{q-1}{2}$ respectively. Write $K:=\langle k\rangle$. Then $\Aut_K(S)$ is an abelian subgroup of $\Aut(S)$ of order at least $\frac{q-1}{2}$. Recall the definition of $P_\Lambda^*(q)$ from Notation \ref{n:slandpl}. Since $\Aut_{P_\Lambda^*(q)}(S_\Lambda(q))$ contains a Hall $p'$-subgroup of $\Aut(S_\Lambda(q))$,  $\Aut_K(S)$ is conjugate in $\Aut(S)$ to a subgroup of the unique abelian subgroup of order $(q-1)^2$ in a Hall $p'$-subgroup of $\Aut_{P_\Lambda^*(q)}(S_\Lambda(q))$ by Lemma \ref{Supauto}.

Let $tV\in H/V\cong \SL_2(q)$ with $t$ an involution. In particular, we may arrange that $t$ is the unique involution in $K$. Then $C_V(t) \ge C_V(H)\cap C_W(S)\ne 1$ and $V=[V,t]\times C_V(t)$ is an $H$-invariant decomposition of $V$. If $[V,t] \not \le [V,S]$, then $[V,t]\ge [V,S]$ and if $C_V(t) \not \le [V,S]$, then $C_V(t)\ge [V,S]$. Since $C_V(t)\cap[V,t]=1$, we deduce that either $V=[V,t]$ or $V=C_V(t)$.

We split the analysis into two cases.

\underline{\textbf{Case 1:}} $C_V(H)\cap C_W(S)\ne 1$.

Since $K$ centralises $C_V(H)\cap C_W(S)$, and a Hall $p'$-subgroup of $\Aut_{P_\Lambda^*(q)}(S_\Lambda(q))$ acts irreducibly on $C_W(S)$, we deduce that $K$ centralises $C_W(S)$. It follows that $C_V(K)=C_W(S)$ by Lemma \ref{action on centreLambda}. Then, as $C_V(t) \ne 1$, we also know that $C_V(t)=V$ and $H/C_H(V)\cong \PSL_2(q)$.

Let $V_M\normaleq H$ with $V_M<V$, and choose $V_M$ maximally with respect to these conditions. Then $V/V_M$ is an irreducible module for $H/V$. If $V_M\not< [V, S]$ then $[V, S]\le V_M$ so that $H$ centralises $V/V_M$. Indeed, as $V_M<V$ and $K$ centralises $V/V_M$, we must have that $C_{V}(K)\ne C_W(S)$, a contradiction. Hence, $V_M<[V, S]$ and so $[V/V_M, S]$ has index $q$ in $V/V_M$. Applying Lemma \ref{lem: Smith's lemma}, we see that $|V/V_M|=q^a$, for some $a\in \N$ and $C_{V/V_M}(S)$ has order $q$. Furthermore, both $C_{V/V_M}(S)$ and $V/[V,S]$ are irreducible as $\FF_pK$-modules. Since $C_{V/V_M}(S)$ is irreducible as a $\FF_pK$-module, we must have that either $C_V(S)\le V_M$ or $C_{V/V_M}(S)=C_V(S)V_M/V_M$.

Assume first that $C_V(S)\le V_M$. Then, as $S/C_V(S) \cong S_{p-2}(q)$ by Lemmas \ref{subgroup similarity} and \ref{lem: LambdaIso}, we have that $S/V_M \cong S_{a-1}(q)$. Since $V/V_M$ is irreducible, Proposition \ref{lem: Vnqrecog} implies that $V/V_M \cong V_{a-1}(q)|_{\FF_p}$ as an $H/C_H(V/V_M)$-module. Since $t$ centralises $V$, $a-1$ is even. But then, $K$ centralises $[V/V_M, S; \frac{a-1}{2}]/[V/V_M, S; \frac{a+1}{2}]$, contrary to $C_V(K)=C_W(S)$. Hence $C_V(S)\not \le V_M$.

Now, again applying Lemmas \ref{subgroup similarity} and \ref{lem: LambdaIso}, we have that $S/C_V(S)V_M\cong S_{a-2}(q)$ and it follows that $[V/V_M,S;a]=1$ but $[V/V_M,S;a-1]\ne 1$. Appealing to Lemma \ref{lem:mod ident}, we ascertain that $V/V_M\cong V_{a-1}(q)|_{\FF_p}$ and as $H/C_H(V)\cong \PSL_2(q)$, we again see that $a-1$ is even. As above, this implies that $C_V(K)>C_W(S)$, and we have a contradiction by Lemma \ref{action on centreLambda}.

\underline{\textbf{Case 2:}} $C_V(H)\cap C_W(S)=1$.

Suppose $Y$ satisfies $Y\normaleq H$, $Y<V$ and $[H, Y]\ne 1$ and $Y$ is chosen minimally with respect to these conditions. In particular, $Y/C_Y(H)$ is irreducible. Note that $C_Y(H)\cap C_W(S)=1$. If $Y\not\le C_V(S)[V, S, S]$ then $[Y, S]=[V, S, S]$ so that $C_W(S)\le Y$. On the other hand, if $Y\le C_V(S)[V, S, S]$ then $[Y, S]\le [V, S, S]$. Indeed, since $[Y, S]\ne 1$, we must have that $C_W(S)\le [Y, S]$. In either case, applying Lemma \ref{lem: Smith's lemma} we conclude that $K$ acts irreducibly on $C_W(S)$ and $|Y/C_Y(H)|=q^a$ for some $a\in \N$.

Assume first that  $C_Y(H)\ne 1$ so that $V=C_V(t)$ and $H/C_H(V)\cong \PSL_2(q)$. Since a Hall $p'$-subgroup of $\Aut_{P_\Lambda^*(q)}(S_\Lambda(q))$ acts irreducibly on $C_V(S)[V, S, S]/[V, S, S]$, in combination with Lemma \ref{action on centreLambda} we get that $C_V(S)=C_V(K)\times C_W(S)$ and $|C_V(K)|=q$.

If $Y\not\le C_V(S)[V, S, S]$, then as $K$ acts irreducibly on $Y/C_Y(H)[Y, S]$ and $|Y/C_Y(H)|=q^a$, we deduce that $|V/Y|<q^2$ so that $H$ centralises $V/Y$. Since $|C_V(K)|=q$, the only possibility is that $V=C_V(S)Y$ and as $C_V(S)\le [V, S]$, we conclude that $V=Y$. Then $C_V(K)=C_V(H)$ and applying Lemma \ref{lem:mod ident}, we have that $V/C_V(H)\cong V_{p-1}(q)|_{\FF_p}$. But $K$ centralises $[V/C_V(H), S; \frac{p-1}{2}]/[V/C_V(H), S; \frac{p+1}{2}]$, contrary to $|C_V(K)|=q$.

If $Y\le C_V(S)[V, S, S]$ then since $C_V(H)\cap C_W(S)=1$ and $|C_{Y/C_Y(H)}(S)|=q$, we deduce that $Y=(Y\cap [V, S, S])C_Y(H)$. Since $|Y/[Y, S]C_Y(H)|=q$, we deduce that $Y\cap [V, S, S]=[V, S; i]$ for some $i\in \N$. Then Lemma \ref{lem:mod ident} reveals that $Y/C_Y(H)\cong V_{p-i-1}(q)|_{\FF_p}$ where $i$ is even. But then $K$ centralises $[Y/C_Y(H), S; \frac{p-i-1}{2}]/[Y/C_Y(H), S; \frac{p-i+1}{2}]$, against $|C_V(K)|=q$.

Hence $C_Y(H)=1$ and $Y$ is irreducible as an $\FF_pH$-module. In particular, $|Y| \geq q^2$ and $C_Y(S)=C_W(S)$. Observe that if $Y\not\le [V, S]$ then $[Y, S]=[V, S]$ so that $|C_Y(S)|=q^2$, a contradiction. Furthermore, if $Y\not\le C_V(S)[V, S, S]$, then $[Y, S]=[V, S, S]$ and as $|Y/[Y, S]|=q$ and $C_Y(S)=C_W(S)$, we see by Lemma \ref{lem:mod ident} that $Y\cong V_{p-2}(q)|_{\FF_p}$ and $|V/Y|=q^2$. If $[K, V]\le Y$ then $[K, C_V(S)]\le C_W(S)$, $|C_V(K)|>q$ and Lemma \ref{action on centreLambda} gives a contradiction. Hence, $V/Y\cong V_1(q)|_{\FF_p}$ and Proposition \ref{prop: indiden} implies that $V\cong \Lambda(q)|_{\FF_p}$, as desired.

Thus, for the remainder of the proof we assume, aiming for a contradiction, that $Y\le C_V(S)[V, S, S]$. Let $i\in \N$ be such that $Y\le C_V(S)[V, S; i]$ but $Y\not\le C_V(S)[V, S; i+1]$. Then $[Y, S]=[V, S; i+1]$ and since $|Y|=q^a$ for some $a\in \N$, we deduce that $|Y/[Y, S]|=q$, $a=p-i$ and $Y\cong V_{p-i-1}(q)|_{\FF_p}$. Set $\bar{V}=V/[V, S; i+1]$. Then $K$ acts irreducibly on $\bar{Y}$ and normalises $\bar{[V, S; i]}$ and $\bar{C_V(S)}$. Moreover, $\bar{Y}\cap \bar{C_V(S)}=1$.

Assume that $\bar{Y}\ne \bar{[V, S; i]}$. Then $\bar{[V, S; i]}$ and $\bar{C_V(S)}$ are isomorphic as $\FF_pK$-modules, and as $k$ acts as a scalar on $\bar{Y}$, we see that $k$ acts as a scalar on $\bar{[V, S; i]}$ and $\bar{C_V(S)}$. Since $\Aut_K(S)$ is $\Aut(S)$-conjugate to a subgroup of $\Aut_{P_\Lambda^*(q)}(S_\Lambda(q))$, we deduce that there is  an element $t$ of order at least $\frac{q-1}{2}$ in $\Aut_{P_\Lambda^*(q)}(S_\Lambda(q))$ such that $t$ acts as a scalar on $\bar{[V, S; i]}$ and $\bar{C_V(S)}$. But $\Aut_{P_\Lambda^*(q)}(S_\Lambda(q))$ contains an element $s$ which acts as the same scalar on all of $V$, and so $st$ is an element of order at least $\frac{q-1}{2}$ which centralises $\bar{[V, S; i]}$ and $\bar{C_V(S)}$. Since $\bar{C_V(S)}$ is isomorphic to $C_V(S)[V, S, S]/[V, S, S]$ as an $\FF_p\langle st\rangle$-module, we have a contradiction by Lemma \ref{action on centreLambda}.

Hence, $Y=[V, S; i]$. Let $V_0\normaleq H$ such that $Y<V_0\le V$ and $V_0$ is minimal subject to these conditions. Then $H$ acts irreducibly on $V_0/Y$. If $[V_0, S]\le Y$ then $[V_0, H]\le Y$ and we deduce from Lemma \ref{lem:NoExt} that $[V_0, S]=[Y, S]$ and $V_0\le C_V(S)Y$. Then $[K, V_0]\le Y$, $V=C_V(t)$ and $H/C_H(V)\cong \PSL_2(q)$ so that $a$ is odd. Since a Hall $p'$-subgroup of $\Aut_{P_\Lambda^*(q)}(S_\Lambda(q))$ acts irreducibly on $C_V(S)[V, S, S]/[V, S, S]$, in combination with Lemma \ref{action on centreLambda} we get that $C_V(S)=C_V(K)\times C_W(S)$ and $|C_V(K)|=q$. But $K$ also centralises $[Y, S; \frac{p-i-1}{2}]/[Y, S; \frac{p-i+1}{2}]$, against $|C_V(K)|=q$. Hence, $[V_0, S]\not\le Y$ from which we deduce that $[V, S; i-1]\le V_0$ and $[V, S; i-1]/Y=C_{V_0/Y}(S)$ by Lemmas \ref{lem: Smith's lemma} and \ref{lem: SL2split}.

From Lemma \ref{lem: Smith's lemma} we see that $|V_0/Y|=q^a$ for some $a\in \N$. If $V_0\not\le [V, S]$ then $C_V(S)[V, S; i-1]/Y=C_{V_0/Y}(S)$ has order $q^2$, a contradiction by Lemma \ref{lem: SL2split}. Then $V_0\le [V, S]$. Since $K$ acts irreducibly on  $V_0/[V_0, S]$, we deduce that either $V_0\le [V, S, S]$ or $[V_0, S]=V_0\cap [V, S, S]$. In either instance, $V_0\cap [V, S, S]=[V, S; j]$ for some $j\in \N$ and Lemma \ref{lem:mod ident} implies that $V_0/Y\cong V_{j-i-1}(q)|_{\FF_p}$, against Proposition \ref{prop: indiden}. This completes the proof.
\end{proof}

We close this section with some more uniqueness results for the above actions. We begin first with an elementary lemma.

\begin{lemma}\label{lem: big cyclic}
Suppose that $p$ is odd and $m>1$. Let $H$ be a $p'$-order subgroup of $\GL_m(p)$ which acts transitively on vectors with $|H|\geq p^m-1$ and $|Z(H)|\geq \frac{p^m-1}{p}$. Then $H\cong \GL_1(p^m)$ is cyclic of order $p^m-1$.
\end{lemma}
\begin{proof}
We appeal to the classification of transitive finite linear groups, due to Hering \cite{hering1985transitive}. Then the restrictions on $H$ promise that $H\le \Gamma\mathrm{L}_1(p^m)$. Set $K:=\GL_1(p^m)$ and assume that $|HK/K|=c\leq m$. Then $|C_K(HK)|=p^{m/c}-1$ and $|Z(H)\cap K|\geq \frac{p^m-1}{pc}$. We deduce that $\frac{p^m-1}{pc}\leq p^{m/c}-1$ so that $c\leq 2$. Moreover, if $c=2$ then $m=2$ and $|H\cap K|=\frac{p^2-1}{2}>p-1$ and so $Z(H)\le K$. But now $|Z(H)|\geq \frac{p^2-1}{p}>p-1$ and we obtain a contradiction. Hence, we conclude that $c=1$ and $H\le K$, and as $|H|=p^m-1=|K|$, the result is true.
\end{proof}

\begin{lemma}\label{lem:unique over groups 2dim}
Assume that $q=p^m$, where $p$ is odd, $G= \SL_2(q)$ and $V$ is a natural module for $G$ over $\FF_p$. Set $H=\GL(V)\cong \GL_{2m}(p)$ and associate $G$ with its image in $H$. Let $S \in \syl_p(G)$ and $N_G(S)$. If $L\cong \SL_2(q)$ with $N_G(S)<L\le H$, then $L=G$.
\end{lemma}
\begin{proof}
Let $L\cong \SL_2(q)$ be a subgroup of $H:=\GL_{2m}(p)$ containing $N_G(S)$. Then $S\in \syl_p(G)\cap \syl_p(L)$ and $N_G(S)=N_L(S)$. If $m=1$ then $H=\GL_2(p)$ has a unique subgroup isomorphic to $\SL_2(p)$ and so $L=G$. Note that the action of $S$ and Lemma \ref{lem: SEFF} implies that there is $g\in H$ with $G^g=L$. We can arrange that $G\le Y$ where $Y\cong \GL_2(q)$, and so $L\le X\cong \GL_2(q)$ with $Y^g=X$. Now, $G=O^{p'}(Y)$, $L=O^{p'}(X)$, $Y=C_{H}(Z(Y))$ and $X=C_{H}(Z(X))$, and so to show that $G=L$, it suffices to prove that $Z(Y)=Z(X)$. Let $K$ be a complement to $S$ in $N_G(S)$ so that $K$ is cyclic of order $q-1$. Then $Z(X)\le C_H(K)\ge Z(Y)$.

Now, we have that $V|_{K}=V_1\times V_2$ where $V_1$ and $V_2$ are non-isomorphic irreducible modules of dimension $m$ over $\FF_p$ since $m>1$. Indeed, $C_H(K)$ acts on each of these modules so that every $p$-element of $C_H(K)$ centralises $V_1$ and $V_2$. Thus, such an element is trivial and $C_H(K)$ is a $p'$-group. Since $K$ normalises $C_V(S)$, without loss of generality we may let $V_1=C_V(S)$ and $V_2\cong V/C_V(S)$ as $\KK K$-modules. Since $K$ acts transitively on vectors in $V_1$ and $V_2$, by coprime action and Lemma \ref{lem: big cyclic} we deduce that $|C_H(K)|\leq (q-1)^2$. Since $|Z(X)K|=|Z(Y)K|=(q-1)^2$, we deduce that $Z(X)K=Z(Y)K$. Then $Z(X)=C_{Z(X)K}(S)=C_{Z(Y)K}(S)=Z(Y)$, as desired.
\end{proof}

 Recall the groups $D$, $D^*$ and $D^\dagger$ from Notation \ref{n:DGroups} and the groups $P_n(q)$ and $P_n^*(q)$ from Notation \ref{n:snandpn}. 

\begin{lemma}\label{lem: NormalizerBorel}
Let $V=V_n(p^2)$ and $S=S_n(p^2)$ for $1\leq n\leq p-1$, $G:=\Aut_{D^\dagger}(V)$ and $T:=\Aut_S(V)\in\syl_p(\Aut_{D^*}(V))$. Suppose that $K\le N_G(T)$ is abelian of order $(p-1)(p^2-1)$ such that $|C_K(T)|=p-1$. Assume further that there is $k\in K$ of order $p^2-1$ such that $k$ acts on $T$ as $\lambda^{-1}$ and on $C_V(T)$ as $\lambda$, where $\lambda\in \FF_{p^2}^*$ has order $p^2-1$. Then $K=C_K(T)\times \langle k\rangle$, $KT$ is uniquely determined in $N_{\Aut_{D^*}(V)}(T)$ and $N_{\Aut(V)}(KT)=N_G(KT)$.
\end{lemma}
\begin{proof}
Since $k$ acts as $\lambda^{-1}$ on $T$, where $\lambda\in \FF_{p^2}^*$ has order $p^2-1$, we have that $\langle k\rangle \cap C_K(T)=1$. Then, as $|K|=(p-1)(p^2-1)$, we see that $K=C_K(T)\times \langle k\rangle$. Note that $N_{\Aut_{D^*}(V)}(T)$ normalises $T$ and, along with $\Inn(S)$, induces the group $\Aut_{P_n^*(p^2)}(S)$. Since $\Aut_{P_n^*(p^2)}(S)$ contains a Hall $p'$-subgroup of $\Aut(S)$ by Lemma \ref{Snauto}, applying Lemma \ref{lem: kerneldescription} we conclude that $K$ is uniquely determined in a Hall $p'$-subgroup of $N_{\Aut_{D^*}(V)}(T)$. Hence, $KT$ is uniquely determined in $N_{\Aut_{D^*}(V)}(T)$. It remains to demonstrate that $N_{\Aut(V)}(KT)=N_G(KT)$.

Set $N:=N_{\Aut(V)}(KT)$ so that $T\normaleq N$. Since $T$ is elementary abelian of order $p^2$, we have that $N/C_N(T)$ is isomorphic to a subgroup of $\GL_2(p)$. Moreover, $KC_N(T)/C_N(T)$ has order $p^2-1$ and is normal in $N/C_N(T)$ from which we deduce that $|N/C_N(T)|\leq (p^2-1).2$. Since $|N_{\Aut_{D^\dagger}(V)}(KT)/C_{\Aut_{D^\dagger}(V)}(KT)|=(p^2-1).2$, we have that $N=N_{\Aut_{D^\dagger}(V)}(KT)C_N(T)$.

By a Frattini argument, we have that $C_N(T)=TN_{C_N(T)}(K)$. Since $N_{C_N(T)}(K)$ normalises $T$ and $C_V(T)$ and as $\langle k\rangle$ is the unique subgroup of $K$ of order $q-1$ such that its element's actions on $T$ and $C_V(T)$ are inverse to each other, we have that $N_{C_N(T)}(K)$ normalises $\langle k\rangle$. Then $[N_{C_N(T)}(K), \langle k\rangle]\le C_{\langle k\rangle}(T)=1$. Moreover, $C_K(T)$ acts as scalars on $V$ and has order $p-1$ from which we deduce that $C_K(T)=Z(\Aut(V))$. Hence, $N_{C_N(T)}(K)$ centralises $K=C_K(T)\times \langle k\rangle$.

We have shown that $N=N_{\Aut_{D^\dagger}(V)}(KT)C_N(KT)$. Our aim will be to demonstrate that $C_N(KT)=Z(\Aut_{D^*}(V))$ has order $p^2-1$. Consider the restriction of $V$ to $\langle k\rangle$. Since $k$ acts as $\lambda^{-1}$ on $T$ and as $\lambda$ on $C_V(T)$, we have that $k$ corresponds to the element $\left(1, \lambda,\left(\begin{smallmatrix} 1 & 0 \\0& \lambda \end{smallmatrix} \right)\right)\in \Aut_{D^*}(V).$
Hence, $k$ acts as $\lambda^{n+1-i}$ on $[V, T; i]/[V, T;i+1]$ for $i\in\{0,\dots, n\}$. Since $0<n+1-i<p+1$, it follows that $k$ acts irreducibly on $[V, T; i]/[V, T;i+1]$  for all $i$. Moreover, since $\langle k\rangle$ has order $p^2-1$ and is centralised by $C_N(KT)$ we have that $C_N(KT)$ respects the $\FF_{p^2}$-structure that $\langle k\rangle$ places on $V$.

Since $[V, T; i]/[V, T;i+1]$ is never isomorphic to $[V,T; j]/[V, T; j+1]$ as an $\FF_{p^2} \langle k\rangle$-module for $i\ne j$, we deduce that $C_N(KT)$ normalises the eigenspaces of $k$. Since $k$ acts irreducibly on each of these eigenspaces, we deduce that $p$-elements of $C_N(KT)$ act trivially on $V$ so that $C_N(KT)$ is a $p'$-group. Then $C_N(KT)$ centralises $T$ and so by the A$\times$B-Lemma \cite[(24.2)]{AschbacherFG}, $C_N(KT)$ acts faithfully on $C_V(T)$. Then, as $C_{\Aut_{D^\dagger}(V)}(KT)$ has order $p^2-1$, $|C_V(T)|=p^2$ and $C_N(KT)$ has $p'$-order, we deduce that $|C_N(KT)|\leq (p^2-1).2$. Note that if $|C_N(KT)|=(p^2-1).2$ then there is $c\in C_N(KT)$ such that $c$ fixes some non-trivial vector in $C_V(T)$. Since $c$ commutes with $K$ and $K$ acts irreducibly on $C_V(T)$, this is a contradiction. Hence, $C_N(KT)=Z(\Aut_{D^*}(V))$ has order $p^2-1$ and the result holds.
\end{proof}

\begin{lemma}\label{lem:unique over groups}
Let $q=p^m$ for $p$ an odd prime, $G$ be a finite group, $T\in\syl_p(G)$ and $V$ be an elementary abelian $p$-group on which $G$ acts faithfully. Suppose that $L:=O^{p'}(G)\cong \mathrm{(P)SL}_2(q)$ and $V|_L\cong V_n(q)|_{\FF_p}$, where $2\leq n\leq p-1$, and recognise $G, \Aut_{D^\dagger}(V)\le \Aut(V)$.

Assume that $N_L(T)\le N_{G\cap \Aut_D(V)}(\Aut_S(V))$ and a Hall $p'$-subgroup of $N_{G\cap \Aut_D(V)}(\Aut_S(V))$ has an abelian subgroup $K$ with $|K|\geq \frac{(q-1)^2}{p+1}$ such that $K$ contains a Hall $p'$-subgroup of $N_L(T)$. Furthermore, if $m=2$ and $|C_K(T)|=p-1$ then assume that the conclusion of Lemma \ref{lem: NormalizerBorel} holds. Then $L=O^{p'}(\Aut_D(V))$.
\end{lemma}
\begin{proof}
Since $N_L(T)\le N_{G\cap \Aut_D(V)}(\Aut_S(V))$, we recognise $T=\Aut_S(V)$. Let $X:=O^{p'}(\Aut_D(V))$ and $K$ be an abelian subgroup of order greater than or equal to $\frac{(q-1)^2}{p+1}$ in a Hall $p'$-subgroup of $N_{G\cap \Aut_D(V)}(T)$. Let $\hat{K}$ be an abelian subgroup of order $(q-1)^2$ containing $K$ which lies in a Hall $p'$-subgroup of $N_{\Aut_{D}(V)}(T)$. Since $X\cong L\cong \mathrm{(P)SL}_2(q)$, we have that $C_K(T)=C_K(X)=C_K(L)$. Then $X,L\le C_{\Aut(V)}(C_K(T))$. We have that $\Aut_D(V)$ preserves an $\KK$-structure on $V$, where $\KK:=\FF_q$. Thus, $\Aut_D(V)\le H\cong \GL_{n+1}(q)$ and $C_K(T)\le Z(H)$ may be identified with a subgroup of the scalar matrices of $H$ acting on $V$. Since $H$ contains a Singer cycle and $|C_K(T)|\geq\frac{(q-1)^2}{p+1}$, we immediately see that $C_{\Aut(V)}(C_K(T))=H$ unless $m=2$ and $|C_K(T)|=p-1$. Thus, either $N_{\Aut(V)}(C_K(T))=N_{\Aut(V)}(H)$ or the conclusions of Lemma \ref{lem: NormalizerBorel} are satisfied.

We claim that

\begin{claim}\label{clmx}$N_{\Aut(V)}(X)=\Aut_{D^\dagger}(V)$.\end{claim}

Since $X\cong L\cong \mathrm{(P)SL}_2(q)$, we calculate that $N_{\Aut(V)}(X)=\Aut_{D^\dagger}(V)C_{\Aut(V)}(X)$. Observe that $|C_{\Aut_{D^\dagger}(V)}(X)|=q-1$ and that $C_{\Aut_{D^\dagger}(V)}(X)\le C_{\Aut(V)}(X)$ acts transitively on vectors in $C_V(S)$. Let $\mathcal{Q}$ be a $p$-subgroup of $C_{\Aut(V)}(X)$. Then $[\mathcal{Q}, V]$ is normalised by $X$ and is properly contained in $V$ from which we conclude that $[\mathcal{Q}, V]=1$ so that $\mathcal{Q}=1$. Hence, $C_{\Aut(V)}(X)$ is a $p'$-group. We deduce by Lemma \ref{lem: big cyclic} that $C_{\Aut(V)}(X)(K\cap N_X(T))=C_{C_{\Aut(V)}(X)(K\cap N_X(T))}(C_V(S))C_{\Aut_{D^\dagger}(V)}(X)$. Then, $C_{C_{\Aut(V)}(X)}(C_V(T))=1$ by the A$\times$B-Lemma \cite[(24.2)]{AschbacherFG} and the claim holds.

By Lemma \ref{lem: spe modules}, we have that $X$ and $L$ are $\Aut(V)$-conjugate. Write $L^g=X$. Then $(TK)^g\le N_{\Aut(V)}(L)^g=N_{\Aut(V)}(L^g)=N_{\Aut(V)}(X)=\Aut_{D^\dagger}(V)$. Indeed, since $T$ is elementary abelian of order $q$, it follows that $T^g\le X$ and so by Sylow's theorem there is $a\in X$ such that $T^{ga}=T$. Since $C_K(T)^{ga}$ centralises $T$, we deduce that $C_K(T)^{ga}\le C_{\Aut_{D^\dagger}(V)}(T)$ and as $C_{\Aut_{D^\dagger}(V)}(T)=C_{\hat{K}}(T)\times T$, and $C_{\hat{K}}(T)$ is cyclic and so has a unique subgroup of order $|C_K(T)|$, we have that $C_K(T)^{ga}=C_K(T)$.

If $m=2$ and $|C_K(T)|=p-1$ then we observe that the conclusion of Lemma \ref{lem: NormalizerBorel} holds and for $k\in K$ such that $k$ acts on $T$ as $\lambda^{-1}$ and on $C_V(T)$ as $\lambda$, where $\lambda\in \FF_{p^2}^*$ has order $p^2-1$, the actions of $ k^{ga}$ on $T/V$ and $C_V(T)$ are inverse to each other. Since $KT$ is uniquely determined in $N_{\Aut_D(V)}(\Aut_S(V))$ by Lemma \ref{lem: NormalizerBorel}, we deduce that $(KT)^{ga}=KT$. Since $N_{\Aut(V)}(KT)=N_{\Aut_{D^\dagger}(V)}(KT)$, we have that $ga$ normalises $X$ and as $X=L^{ga}$, we see that $X=L$ as desired.

Hence, we continue under the assumption that $N_{\Aut(V)}(C_K(T))=N_{\Aut(V)}(H)$. By Sylow's theorem, we have that $N_L(T)^{ga}=N_{L^{ga}}(T)=N_X(T)$. Then Lemma \ref{lem: kerneldescription} implies that $N_L(T)=N_X(T)$. Replacing $g$ by $ga$, we deduce that there is $g\in N_{\Aut(V)}(N_X(T))\cap N_{\Aut(V)}(H)$ with $L^g=X$.

We claim that $N_{N_{\Aut(V)}(H)}(N_X(T))=N_{\Aut_{D^\dagger}(V)}(TC_K(T))$ so that $L=X^{g^{-1}}=X$. Note that $$N_{\Aut_{D^\dagger}(V)}(TC_K(T))\le N_{N_{\Aut(V)}(H)}(N_X(T))\mbox{ and } |N_{\Aut_{D^\dagger}(V)}(TC_K(T))|=q(q-1)^2.m.$$ Therefore, it will suffice to show that $|N_H(N_X(T))|=q(q-1)^2$. Write $N_X(T)= TR$, where $R$ is cyclic of order $|Z(X)|\frac{q-1}{2}$, and set $J:=N_H(N_X(T))$. Then, by a Frattini Argument, $J=N_J(R)N_X(T)=N_J(R)T$.

Since the eigenspaces of $R$ are all $1$-spaces over $\KK$, we have $N_J(R)$ permutes these $1$-dimensional subspaces. On the other hand, $N_J(R)$ normalises $T$ and so leaves $[V, T;i]=[V, S; i]$ invariant for $0\leq i\leq n$. As $[V, S; i]$ is a sum of eigenspaces, starting with $[V, S; n]$, we see that $N_J(R)$ preserves each $R$ eigenspace. It follows that $N_J(R)$ has exactly the same submodules as $R$ and, in particular, $N_J(R)$ is abelian of order at most $(q-1)^{n+1}$. In particular, $J/T$ is abelian and has order coprime to $p$.

Now $J/T$ is an abelian group which acts faithfully and irreducibly on $S$. Hence $J/C_J(T)$ has order at most $q-1$ by Schur’s Lemma. Write $C_J(T)=T\times F$ for  some abelian group $p'$-subgroup $F$. Then $F$ acts on $C_V(T)$ and $F/C_F(C_V(T))$ embeds into $\GL_1(q)$ and so has order at most $q-1$. Now $T\times C_F(C_V(T))$ acts on $V$ and we deduce that $C_F(C_V(T))=1$ by the A$\times $B-Lemma \cite[(24.2)]{AschbacherFG}. Hence, $|F|\leq q-1$, $|J/T|\leq (q-1)^2$ and the result holds.
\end{proof}

We now perform the same exercise for the module $\Lambda(q)$.

\begin{lemma}\label{lem:unique over groups lambda}
Let $q=p^m>p$, $G$ be a finite group, $T\in\syl_p(G)$ and $V$  an elementary abelian $p$-group on which $G$ acts faithfully. Suppose that $L:=O^{p'}(G)\cong \SL_2(q)$ and $V|_L\cong \Lambda(q)|_{\FF_p}$, and regard  $G$ and $\Aut_{D^*}(V)$ as subgroups of  $\Aut(V)$. Assume that $N_L(T)\le N_{G\cap \Aut_{D^*}(V)}(T)$ and a Hall $p'$-subgroup of $N_{G\cap \Aut_D(V)}(T)$ has an abelian subgroup of order equal to $(q-1)^2$ containing a Hall $p'$-subgroup of $N_L(T)$. Then $L=O^{p'}(\Aut_D(V))$.
\end{lemma}
\begin{proof}
Let $X:=O^{p'}(\Aut_D(V))$ and $K$ be an abelian subgroup of order $(q-1)^2$ in a Hall $p'$-subgroup of $N_{G\cap \Aut_D(V)}(T)$. We identify $T\in\syl_p(G)$ with $\Aut_S(V)\le \Aut_{D}(V)$. We have by Lemma \ref{lem: spe modules} that $V|_L\cong \Lambda(q)|_{\FF_p}\cong V|_X$. Set $W_L, W_X$ to be the relevant $m(p-1)$-dimensional $\FF_p$-submodules of $V$ for $L$ and $X$ respectively. We claim that $W_L=W_X$. Let $\lambda \in \KK^*$ have order $q-1$,  $r \in X$ be the conjugation map corresponding to  $\left(1,1,
\left(\begin{smallmatrix}
\lambda^{-1} & 0 \\ 0 & \lambda
\end{smallmatrix}\right)
\right)\in D,$
and $R=\langle r \rangle\le X$ so that $|R|=q-1$. Then we may as well assume that $R\le K$ and so normalises $X$ and $L$. Hence, it also normalises $W_X$ and $W_L$. Note that for $\bar{V}:=V/[V, T, T]$ we have that $\bar{[V, T]}=\bar{W_X}\times \bar{C_V(T)}=\bar{W_L}\times \bar{C_V(T)}$. Moreover, we calculate that $r$ acts as $\lambda^{-p}$ on $\bar{C_V(T)}\cong \bar{[V, T]}/\bar{W_X}$ and as $\lambda^{p-2}$ on $\bar{W_X}$. Since $q>p$, $\bar{C_V(T)}$ and $\bar{W_X}$ are non-isomorphic modules for $R$ and so $\bar{[V, T]}$ has precisely two proper non-trivial $R$-submodules. Since $\bar{W_L}\ne \bar{C_V(T)}$, we must have that $\bar{W_X}=\bar{W_L}$ so that $W:=W_X=W_L$, as claimed. Hence, $X,L\le N_{\Aut(V)}(W)$ and we may consider the restriction of $N_{\Aut(V)}(W)$ to $\Aut(W)$. Note that $XK$ and $LK$ restrict faithfully to $W$ and applying Lemma \ref{lem:unique over groups}, we deduce that $XC_{\Aut(V)}(W)=LC_{\Aut(V)}(W)$.

We claim that

\begin{claim}\label{clmy} for $p>3$,  $N_{\Aut(V)}(X)=\Aut_{D^\dagger}(V)$.\end{claim}

Since $X\cong \SL_2(q)$ and $\Aut(X)= \Aut_{D^\dagger}(X)$, we calculate that $N_{\Aut(V)}(X)=\Aut_{D^\dagger}(V)C_{\Aut(V)}(X)$. Then $W$ is the unique proper non-trivial $X$-submodule of $V$ and we have that $C_{\Aut(V)}(X)$ normalises $W$. Indeed, since $XC_{\Aut(V)}(W)=LC_{\Aut(V)}(W)$, it follows from the proof of Lemma \ref{lem:unique over groups} that $N_{\Aut(V)}(X)=\Aut_{D^\dagger}(V)C_{C_{\Aut(V)}(X)}(W)$. Let $Q$ be a $p$-subgroup of $C_{C_{\Aut(V)}(X)}(W)$. Then $[Q, V]$ and $[Q, W]$ are normalised by $X$ and properly contained in $V$, resp. $W$, from which we conclude that $Q\le C_{C_{\Aut(V)}(X)}(W)\cap C_{C_{\Aut(V)}(X)}(V/W)$. Then $C_{\Aut_{D^\dagger}(V)}C_{C_{\Aut(V)}(X)}(W)/C_{\Aut(V)}(V/W)$ is a $p'$-group acting transitively on vectors in $C_{V/W}(T)$. We conclude by Lemma \ref{lem: big cyclic} and the A$\times$B-Lemma \cite[(24.2)]{AschbacherFG} that $$C_{\Aut(V)}(X)=C_{\Aut_{D^\dagger}(V)}(C_{C_{\Aut(V)}(X)}(W)\cap C_{C_{\Aut(V)}(X)}(V/W))=C_{\Aut_{D^\dagger}(V)}(X)H$$ where $H=C_{C_{\Aut(V)}(X)}(W)=C_{C_{\Aut(V)}(X)}(V/W)=O_p(C_{\Aut(V)}(X))$.

Aiming for a contradiction, assume that $H\ne 1$ so that $[H, V]=W$. Note that $[[V, T], T, H]=1=[T, H, [V, T]]$ from which we conclude that $[[V, T], H, T]=1$ and since $[V, H]\le W$, we must have that $[[V, T], H]\le C_W(T)$. On the other hand, $[T, H, V]=1$ and we deduce that $[H, V, T]=[W, T]\le C_W(T)$. Since $p>3$, this is a contradiction.

If $p=3$ then we recreate the proof above to conclude that $C_{\Aut(V)}(X)=C_{\Aut_{D^\dagger}(V)}(X)H$ and $H=O_3(N_{\Aut(V)}(X))$. Furthermore, $H$ acts quadratically on $V$ and so $H$ is an elementary abelian $3$-group. From this, we have that $TH$ is an elementary abelian $3$-subgroup.

Write $\mathbb{T}:=O^p(N_{\Aut(V)}(X))\cap TH$ so that $\mathbb{T}\in\syl_p(O^p(N_{\Aut(V)}(X)))$. Of course, if $p>3$ then $\mathbb{T}=T$. Note that every elementary abelian $p$-subgroup of $N_{\Aut(V)}(X)$ which has the same order as $TH$ is contained in $O^p(N_{\Aut(V)}(X))H$. Then $K\le N_{O^p(N_{\Aut(V)}(X))}(\mathbb{T})$ and $N_{O^p(N_{\Aut(V)}(X))}(\mathbb{T})$ is solvable. From the structure of $O^p(\Aut_{D^\dagger}(V))$, we deduce that $K$ is the unique abelian subgroup of order $(q-1)^2$ in a Hall $p'$-subgroup of $N_{N_{\Aut(V)}(X)}(\mathbb{T})$.

By Lemma \ref{lem: spe modules}, there is $g\in \Aut(V)$ with $L^g=X$. Then since $K\le O^p(N_{\Aut(V)}(L))$, we see that $K^g\le O^p(N_{\Aut(V)}(X))$ and since $K$ normalises $T$, we have that $K^g$ normalises $T^g$. Now, $T^g\in\syl_p(X)$ and so $T^gH$ is an elementary abelian $p$-subgroup of $N_{\Aut(V)}(X)$ which has the same order as $TH$. By Sylow's theorem, there is $c_1\in O^p(N_{\Aut(V)}(X))$ such that $T^{gc_1}\le \mathbb{T}$ and $\mathbb{T}=(\mathbb{T}\cap H)T^{gc_1}$. Since $(\mathbb{T}\cap H)=O_p(O^p(N_{\Aut(V)}(X)))$, we have that $K^{gc_1}$ normalises $\mathbb{T}$. Then, since $K^{gc_1}$ is abelian of order $(q-1)^2$, by Hall's theorem there is $c_2\in N_{O^p(N_{\Aut(V)}(X))}(\mathbb{T})$ such that $K^{gc_1c_2}=K$.

On the other hand, $T^g\in \syl_p(X)$ and since $c:=c_1c_2\in N_{\Aut(V)}(X)$, we deduce that $T^{gc}\in\syl_p(X)$. Hence there is $a\in X$ such that $T^{gca}=T$. Since $K^{gc}$ normalises $T^{gc}$, we ascertain that $K^{gca}\le N_{\Aut_{D^\dagger}(V)}(T)$ and since $K^{gca}$ is elementary abelian of order $(q-1)^2$, we conclude that $(TK)^{gca}=TK^{gca}=TK$. Finally, we have that $L^{gca}=X^{ca}=X$. Replacing $gca$ by $g$, we have demonstrated that there is $g\in N_{\Aut(V)}(TK)$ with $L^g=X$. To complete the proof of the lemma, we shall show that $N_{\Aut(V)}(TK)$ normalises $X$.

From the proof of Lemma \ref{lem:unique over groups}, we have that $N_{\Aut(V)}(TK)\le C_{\Aut(V)}(W)N_{\Aut_{D^*}(V)}(TK)$. It remains to show that $N_{C_{\Aut(V)}(W)}(B)=1$. Indeed, we have that $N_{C_{\Aut(V)}(W)}(B)$ normalises $T$ and $N_{C_{\Aut(V)}(W)}(B)=N_{N_{C_{\Aut(V)}(W)}(TK)}(K)$ by the Frattini argument. Suppose that $\ell\in N_{C_{\Aut(V)}(W)}(T)$ has $p'$-order. Then $[\ell, T]=1$ and $V=[V, \ell]\times C_V(\ell)$ is a $T$-invariant decomposition. Since $\ell\in C_{\Aut(V)}(W)$, we have that $W\le C_V(\ell)$ and $C_V(\ell)\ne 1$. Observe that for $U\in\{[V, \ell], C_V(\ell)\}$, if $U\not\le [V, T]$ then $[V, T]=[U, T][V, T, T]$. Since $T$ acts nilpotently on $V$, we deduce that $C_V(T)\le [V, T]\le U$ and since $V=[V, \ell]\times C_V(\ell)$ is a $T$-invariant decomposition, we conclude that $U=V$. Finally, as $C_V(\ell)\ne 1$, we deduce that $V=C_V(\ell)$ and $\ell=1$.
Therefore, $N_{C_{\Aut(V)}(W)}(TK)$ is a $p$-group.

 To prove the result it suffices to show that $N_{N_{C_{\Aut(V)}(W)}(TK)}(K)$ contains no $p$-elements. Write $K_1=C_K(V/[V, T])$ and $K_2=C_K([V, T]/W)$. Then $|K_1|=|K_2|=q-1$. Let $t\in N_{N_{C_{\Aut(V)}(W)}(TK)}(K)$ be a $p$-element so that $[t, K]=1$. Indeed, $t\in N_{C_{\Aut(V)}(W)}(TK)$ and $[V, t]$ is normalised by $TK$. Since $t$ is a $p$-element, we conclude that $[V, t]\le [V, T]$. By a similar reasoning, we have that $[V, T, t]\le W$. Observe that $[[V, T], K_2, t]=1$ since $t\in C_{\Aut(V)}(W)$ and that $[t, K_2, [V, T]]=1$. The three subgroups lemma then implies that $[[V, T], t, K_2]=1$. Since $TK$ commutes with $t$, we have that $[V, T, t]$ is normalised by $T$ and contained in $W$. If $[V, T, t]$ is non-trivial then $C_W(T)\cap [V, T, t]\ne 1$ and since $TK$ normalises $[V, T, t]$, we deduce that $C_W(T)\le [V, T, t]$ so that $K_2$ centralises $C_W(T)$. That is, $K_2$ centralises both $[V, T]/W$ and $C_W(T)$. But $K_2\le N_{\Aut_{D^*}(V)}(T)$ and computing in the manner following the discussion after Lemma \ref{Vpqstruct}, we observe that no element of order $q-1$ acts like this. Hence, $t$ centralises $[V, T]$.

Now, $[V, K_1, t]=1=[K_1, t, V]$ and we conclude that $[V, t, K_1]=1$. Since $K_1$ acts irreducibly on $[V, T]/W$ we deduce that $[V, t]\le W$ and a similar argument as before implies that $K_1$ centralises $C_W(T)$ if $[V, t]$ is non-trivial. If $[V, t]$ is non-trivial, then as in the above paragraph, we reduce to checking that no abelian subgroup of $N_{\Aut_{D^*}(V)}(T)$ of order $q-1$ centralises both $C_W(T)=C_{[V, T, T]}(T)$ and $V/[V, T]$. Once again, following the discussion after Lemma \ref{Vpqstruct}, there is no such subgroup. Hence, $[V, t]=1$ so that $t=1$. Therefore, we have shown that $N_{\Aut(V)}(KT)=N_{\Aut_{D^\dagger}(V)}(TK)$ normalises $X$ and $L=X$, as desired.
\end{proof}

\section{The potential $\F$-essential subgroups}\label{sec:pot}

By the Alperin--Goldschmidt theorem, any saturated fusion system $\F$ on $S$ is determined by $\Aut_\F(S)$, together with the collection of all $\Aut_\F(P)$ for $P \in \E(\F)$. Here we determine the possibilities for such $P \le S$ when $S \in \{S_n(q), S_\Lambda(q)\}$. Throughout this section let $q=p^m$ be such that $m  > 1$. Let $n \ge 2$, assume $S = S_n(q)$ or $S=S_\Lambda(q)$ and write $V$ for the unique elementary abelian subgroup of $S$ of index $q$ (so that either $V = V_n(q)$ or $V=\Lambda(q)$ respectively). Let $S_0$ be an elementary abelian group of order $q$ such that $S=S_0V$, so that $S_0$ may be identified with the subgroup $U$ described from Notation \ref{n:subsofd*}. Recall $R:=Z(S)S_0$ and $Q:=Z_2(S)S_0$ from Notation \ref{n:bsandcs}.  We prove the following three results:

\begin{proposition}\label{essentials}
Let $\F$ be a saturated fusion system on $S:=S_n(q)$ with $2\leq n\leq p-1$. Then either $\mathcal E(\F) \subseteq \{V\} \cup R^S$ or  $\mathcal E(\F) \subseteq \{V\}\cup Q^S$. If, in addition,  $O_p(\F)=1$ then one of the following holds: \vspace{-4mm}
\begin{enumerate}
    \item $\mathcal{E}(\F)=\{V\} \cup  Q^S$; or
    \item $\mathcal{E}(\F)=\{V\} \cup R^S$; or
    \item $\mathcal{E}(\F)=R^S$.
\end{enumerate}
\end{proposition}

\begin{proposition}\label{upVEssentials}
Let $\F$ be a saturated fusion system on $S:=S_{\Lambda}(q)$. Then $\mathcal{E}(\F)\subseteq \{V\} \cup R^S$. If, in addition, $O_p(\F)=1$ then $\mathcal{E}(\F)=\{V\} \cup R^S$.
\end{proposition}

\begin{proposition}\label{ess.auto}
Let $S \in \{S_\Lambda(q), S_n(q)\}$ and $\F$ be a saturated fusion system on $S$ where $2\leq n\leq p-1$ and let $E \in \mathcal E(\F)$. The following hold: \vspace{-4mm}
\begin{enumerate}
\item If $E \in R^S \cup Q^S$ then $O^{p'}(\Out_\F(E)) \cong \SL_2(q)$.
\item If $E=V$ then either
	\begin{enumerate}
	\item $O^{p'}(\Aut_{\F}(V))\cong \mathrm{(P)SL}_2(q)$ and $V$ is isomorphic to $V_n(q)$ or $\Lambda(q)$, as $\FF_p O^{p'}(\Aut_{\F}(V))$-modules; or
	\item $p=3$, $q=9$, $n=2$, $O^{3'}(\Aut_{\F}(V))\cong 2^.\PSL_3(4)$ and $V$ is isomorphic to the unique irreducible $\FF_3O^{3'}(\Aut_{\F}(V))$-module of dimension $6$.
	\end{enumerate}
\end{enumerate}
\end{proposition}

The primary tool we will use in this section is the following lemma.

\begin{lemma}\label{lem:series}
Suppose that $\F$ is a saturated fusion system on $S$ and $J \le S$. Assume that $J_s < J_{s-1}< \dots <J_0=J$ are $\Aut_\F(J)$-invariant with $J_s \le \Phi(J)$.  If $ A \le \Aut_S(J)$ and $[J_i,A] \le J_{i+1}$ for $0 \leq i \leq s-1$, then $A \le O_p(\Aut_\F(J))$. In particular, if $A \not \le \Inn(J)$, then $J$ is not $\F$-centric-radical.
\end{lemma}
\begin{proof}
See \cite[Lemma 1.1]{oliver2009saturated}.
\end{proof}

We first prove:

\begin{proposition}\label{prop: V_pCon1}
Let $\F$ be a saturated fusion system on $S\cong S_p(q)$. Then $\mathcal{E}(\F)\subseteq \{V\}$ and $V\normaleq \F$.
\end{proposition}
\begin{proof}
Let $E\in\mathcal{E}(\F)$. Let $W=\ker(\psi) \le V_p(q)$ for $\psi$ as in Lemma \ref{Vpqstruct}. From Lemma \ref{(p+1)CVS}, we recall that $C_V(S)=C_W(S)$ has order $q$, $|V/[V, S]|=q^2$ and $|C_V(s)|=q^2$ for all $s\in S\setminus V$. Indeed, $W<Z_2(S)<V$, $|Z_2(S)|=q^3$ and $Z_2(S)=C_V(s)W$ for all $s\in S\setminus V$.

Suppose that $Z(E)\not\le V$. Then $E\cap V\le C_V(Z(E))\le Z_2(S)$ and $E\cap W=C_W(E)=C_V(S)$. We note that $W\le Z_2(S)\le N_S(E)$ and that $W$ centralises a subgroup of $Z(E)$ of index at most $q$. Furthermore, $W$ centralises $E/Z(E)$ and as $|WE/E|=q$, Lemma \ref{lem: SEFF} implies that $N_S(E)=EW$, $|N_S(E)/E|=q$ and $S=Z(E)V$. Thus, $E\cap V=C_V(S)$ and $E=Z(E)$ has order $q^2$, again by Lemma \ref{lem: SEFF}. But then $q^2=|Z_2(S)E/E|\leq |N_S(E)/E|=q$, a contradiction.

Next suppose that $Z(E)\le V$. Then $Z_2(S)$ centralises the chain $1\normaleq Z(E)\normaleq E$ and  $Z_2(S)\le E$ by Lemma \ref{lem:series}. Hence $Z_2(S)\le E\cap V$ and $E\cap V$ is elementary abelian with $|E \cap V| \ge q^3$.  Suppose that $A$ is another elementary abelian subgroup with $ |A| \ge |E\cap V|$ and $A \nleq V$. Then for $a\in A\setminus (A\cap V)$ we have $V\cap A\le C_V(a)$ has order at most $q^2$. Since $|S/V|=q$, $|A|=|E\cap V|=q^3$ and $S=AV$,  but then $|C_V(A)|=q$ and $|A|=q^2$, a contradiction. Thus $E\cap V$ is the unique elementary abelian subgroup of $E$ of maximal order. Hence $N_V(E)$ centralises the chain $1\normaleq E\cap V\normaleq V$ and $V\le E$ by Lemma \ref{lem:series}. If $V<E$ and $A$ is such that $V<A\le E$ then  $[V, A]=[V, S]$ has index $q^2$ in $V$ by Lemma \ref{(p+1)CVS} (\ref{cvs3}). Hence $[V, A; i]=[V, S; i]$ for all $1\leq i\leq p$ and since $V$ is characteristic in $E$,  $S$ centralises the $\Aut_{\F}(E)$-invariant chain $$1\normaleq C_V(S)=[V, S; p-1]\normaleq [V, S; p-2]\normaleq \dots \normaleq [V, S]\normaleq V\normaleq E,$$ a contradiction by Lemma \ref{lem:series}. Hence, $V=E$ and $\mathcal{E}(\F)\subseteq \{V\}$. Since $V$ is characteristic in $S$, $V\normaleq \F$ by \cite[Proposition I.4.5]{AKO}.
\end{proof}

\begin{remark}
The above proof also applies to the case $S_2(2^m)$ for $m>1$.
\end{remark}

We now embark on an initial determination of the essential subgroups of $\F$, where $S\in \{S_n(q), S_{\Lambda}(q)\}$ and $2\leq n\leq p-1$. Recall from Lemmas \ref{Somnibus} and \ref{Snauto} that for $S=S_n(q)$ we have $$[S \colon V] = q, \hspace{2mm} |Z_i(S)|=q^i \mbox{ for all } 1\leq i \leq n \hspace{2mm} \mbox{ and } \hspace{2mm} Z_2(S) \le Z_n(S)=S'=[V,S] < V.$$  Similarly, if $S=S_{\Lambda}(q)$ then $$[S \colon V]=q, \hspace{2mm} |Z_i(S)|=q^{i+1} \mbox{ for all  } 1\leq i\leq p-1, \hspace{2mm} \mbox{ and } \hspace{2mm} Z_2(S)\le Z_{p-1}(S)=S'=[V, S]<V.$$ The proof of the forthcoming lemma follows the same methodology as Proposition \ref{prop: V_pCon1} but reveals the subgroups $R$ and $Q$ as potential candidates for the $\F$-essential subgroups (in (\ref{ab.ess}) and (\ref{nonab.ess}) of Lemma \ref{size.ess} respectively).  Recall the notation $\mathcal{B}(S)$ and $\mathcal{C}(S)$ introduced in Notation \ref{n:bsandcs} and that $d=q$ if $S=S_{\Lambda}(q)$ and $d=1$ if $S=S_n(q)$. 

\begin{lemma}\label{size.ess}
Let $\F$ be a fusion system on $S$ where $S\in \{S_n(q), S_{\Lambda}(q)\}$ and $2\leq n\leq p-1$. If $E\in\mathcal{E}(\F)\setminus \{V\}$, then either
\begin{enumerate}
\item\label{ab.ess} $E\in\mathcal{B}(S)$, $E\cap V=Z(S)$, $N_S(E)=EZ_2(S)$, $|C_E(O^{p'}(\Aut_\F(E)))|=d$ and $E/C_E(O^{p'}(\Aut_\F(E)))$ may be regarded as a natural $\SL_2(q)$-module for $O^{p'}(\Aut_\F(E))\cong \SL_2(q)$; or
\item\label{nonab.ess} $E\in\mathcal{C}(S)$, $E\cap V=Z_2(S)$, $Z(E)=Z(S)$, $N_S(E)=EZ_3(S)$ and $E/Z(E)$ may be regarded as a natural $\SL_2(q)$-module for $O^{p'}(\Out_\F(E)) \cong \SL_2(q)$.
\end{enumerate}
\end{lemma}
\begin{proof}
Since $E$ is $\F$-centric, we have $E \nleq V$ and $[Z_2(S), E]\le [Z_2(S), S]\le Z(S)\le E$, so $Z_2(S)\le N_S(E)$. We will freely use the results from Lemma \ref{Somnibus} and Lemma \ref{(p+1)CUpS} throughout this proof.

Suppose first that $Z(E)\not\le V$. If $z \in Z(E) \setminus V$ then $C_V(z) = C_V(S)$ and so $$Z(S)\le E\cap V\le C_V(z)=C_V(S)\le Z(S).$$ Thus $|E\cap V| = |Z(S)| = qd$ and $|E/E\cap V| \leq |S/V| = q$ which implies that $|E|\leq q^2d$. Since $q=|Z_2(S)E/E|\leq |N_S(E)/E|$ and $Z_2(S)$ centralises $E\cap V$, we deduce from Lemma \ref{lem: SEFF} that $E$ is elementary abelian of order $q^2d$, $Z_2(S)E=N_S(E)$, $S=EV$, $|C_E(O^{p'}(\Aut_\F(E)))|=d$ and that we may regard $E/C_E(O^{p'}(\Aut_\F(E)))$ as a natural $\SL_2(q)$-module for $O^{p'}(\Aut_\F(E)) \cong \SL_2(q)$. Moreover, by Lemma \ref{lem:series}, we must have that $C_E(O^{p'}(\Aut_\F(E)))\cap [Z_2(S), S]=1$ and this proves (1).

Now suppose that $Z(E)\le V$. If $e \in E\setminus V$ then $$Z(S)\le Z(E)\le C_V(e)=C_V(S) \le Z(S),$$ and so $Z(E)=Z(S)$. Since $[Z_2(S), E]\le Z(E)$, $Z_2(S)$ centralises the sequence $1 \normaleq Z(E) \normaleq E$, and we conclude from Lemma \ref{lem:series} that $Z_2(S)\le E$. In particular, $[Z_3(S), E] \le Z_2(S) \le E$ and $Z_3(S)\le N_S(E)$. Since $E\cap V$ is an elementary abelian subgroup of $E$ containing $Z_2(S)$, $|E\cap V|\geq q^2d$. Moreover $N_V(E)$ centralises the normal chain $1 \normaleq E \cap V \normaleq E$ and so Lemma \ref{lem:series} implies that either $V\le E$, or $E\cap V$ is not $\Aut_\F(E)$-invariant.

If $V \le E$ then $V$ is the unique elementary abelian subgroup of $E$ of maximal order. Moreover, if $A$ is such that $V<A\le E$, then $Z_i(A)=Z_i(S)$ for $1\leq i\leq m:=\mathrm{min}\{n, p-1\}$ and $[S, A]\le Z_m(S)$. Then $S$ centralises the series $1\normaleq Z(E)\normaleq Z_2(E)\normaleq \dots \normaleq Z_m(E)\normaleq E$ and Lemma \ref{lem:series} yields a contradiction.

If $E \cap V$ is not $\Aut_\F(E)$-invariant then there exists an elementary abelian subgroup $A$ of $E$ distinct from $E\cap V$  with $|A|\geq |E \cap V| \geq q^2d$. On the other hand if $a \in A \setminus V$ then $A\cap V\le C_V(a)=C_V(S)\le Z(S)$ which implies that $|A\cap V| \leq qd$ so $|A|=|AV/V||A\cap V| \leq q^2d$. Therefore $|E\cap V|=|A|=q^2d$ and we must have $S=AV=EV$ and $Z_2(S)=E\cap V$. It follows that $E=AZ_2(S)$ has order $q^3d$ and $\Phi(E)=[A, Z_2(S)]\le Z(S)=Z(E)$. Moreover $O^{p'}(\Aut_{\F}(E))$ centralises $Z(E)$ and $N_S(E)$ centralises $Z_2(S)/Z(E)$. Since $|E/Z(E): Z_2(S)/Z(E)|=q$, we obtain $q=|Z_3(S)E/E|\leq |N_S(E)/E|$ and Lemma \ref{lem: SEFF} implies that $N_S(E)=Z_3(S)E$ and $E/Z(E)$ may be regarded as a natural $\SL_2(q)$-module for $O^{p'}(\Out_\F(E)) \cong \SL_2(q)$. This completes the proof of (2).
\end{proof}

\begin{lemma}\label{lem:intersec}
Let $\F$ be a fusion system on $S$. Suppose that $E,F \in \mathcal E(\F)\setminus \{V\}$ are such that $E \ne F$ and $E$ is abelian. Then
\begin{enumerate}
\item\label{ea1} $E \nleq F$.
\item\label{ea2} $E \cap F = Z(S)$, and $E \cap E^s = Z(S)$ for all $s \in S \backslash N_S(E)$.
\item\label{FFx1} If $F$ is also non-abelian and $P \in \mathcal E(\F)\setminus \{V,F\}$ is non-abelian, then $F \cap P = Z_2(S)$ and $F \cap F^s= Z_2(S)$ for each $s \in S \backslash N_S(F)$.
\end{enumerate}
\end{lemma}
\begin{proof}
By Lemma \ref{size.ess} we have that $|E|=q^2d$ and $|F|\in\{q^2d, q^3d\}$. Aiming for a contradiction, suppose that $E < F$. Then $|F|=q^3d$ and $F$ is as described in Lemma~\ref{size.ess}(\ref{nonab.ess}). Since $Z(S) < E < F$ and $O^{p'}(\Aut_\F(F))$ acts transitively on subgroups of order $q$ in $F/Z(S)$, there exists a morphism $\alpha \in O^{p'}(\Aut_\F(F))$ such that $E\alpha \cap Z_2(S)> Z(S)$. Then, as $E$ is abelian, $C_{V}(E)> Z(S)$ and we conclude that $E\alpha \le V$, contradicting the fact that $E$ is fully $\F$-normalised. This proves (\ref{ea1}). Parts (\ref{ea2}) and (\ref{FFx1}) follow from Lemma \ref{lem:intersec1}.
\end{proof}

\begin{lemma}\label{lem:S-conj22}
Suppose that $E, F \in \mathcal E(\F)$ and $|E|=|F|$. If $E[V,S] \cap F[V,S] > [V,S]$, then $E$ and $F$ are $S$-conjugate and $E[V,S]=F[V,S]$.
\end{lemma}
\begin{proof}
This follows from Lemma \ref{lem:S-conj2}.
\end{proof}

\begin{lemma}\label{lem:extends}
Suppose that $E\in \mathcal E(\F) \setminus \{V\}$. If $\theta \in N_{\Aut_\F(E)}(\Aut_S(E))$, then there exists $\ov{\theta} \in \Aut_\F(S)$ such that $\theta =\ov \theta|_E$.
\end{lemma}
\begin{proof}
Suppose that $\theta\in N_{\Aut_\F(E)}(\Aut_{S}(E))$. Then there exists $\theta^* \in \Aut_\F(N_S(E))$ such that $\theta^*|_E= \theta$. By Lemmas \ref{size.ess} and \ref{lem:intersec} no essential subgroup of $\F$ properly contains $E$ and the Alperin-Goldschmidt Theorem yields $\theta^*=\ov \theta|_E$ for some $\ov \theta \in \Aut_\F(S)$. This establishes the claim.
\end{proof}

\begin{lemma}\label{lem:TE}
Suppose that $E \in \mathcal E(\F)$. Then $N_{O^{p'}(\Aut_\F(E))}(\Aut_S(E))$ has cyclic Hall $p'$-subgroups of order $q-1$. Fix such a subgroup $T_E$, and let $T\le \Aut_{\F}(S)$ be a subgroup of order $q-1$ such that $T|_E=T_E$. Then
\begin{enumerate}
\item $T$ acts regularly on $(S/V)^\#$;
\item if $E$ is abelian, $T_E$ acts regularly on $(Z(S)/C_{E}(O^{p'}(\Aut_{\F}(E))))^\#$; and
\item if $E$ is non-abelian, then $T_E$ centralises $Z(S)$.
\end{enumerate}
\end{lemma}
\begin{proof}
Let $t\in T^\#$. By Lemma \ref{size.ess}, $E\cap V\in\{Z(S), Z_2(S)\}$ and so $E\cap V$ is normalised by $t$. Since $S=EV$, the elements of $(S/V)^\#$ are of the form $eV$ where $e\in E\setminus (E\cap V)$. Suppose $[t, eV]=1$ for some such $e$. Then $[t, e]\le E\cap V$, and since $E/C_E(O^{p'}(\Aut_{\F}(E)))$ is a natural $\SL_2(q)$-module with $$[E/C_E(O^{p'}(\Aut_{\F}(E))), N_S(E)]=(E\cap V)/C_E(O^{p'}(\Aut_{\F}(E))),$$ $t$ acts non-trivially on $E/E\cap V$, a contradiction. Thus since $|(S/V)^\#|=q-1=|T|$, $T$ acts regularly on $(S/V)^\#$ and (1) is proved.

If $E$ is abelian then by Lemma \ref{size.ess}(\ref{ab.ess}), $E/C_{E}(O^{p'}(\Aut_{\F}(E)))$ can be viewed as a natural module for $O^{p'}(\Aut_{\F}(E))\cong \SL_2(q)$. In particular, any non-trivial element of $T_E$ acts fixed point freely on $E/C_{E}(O^{p'}(\Aut_{\F}(E)))$. As $T_E$ is the restriction of $T$ to $E$, $T_E$ normalises $Z(S)$ and so normalises $Z(S)/C_{E}(O^{p'}(\Aut_{\F}(E)))$. Since $|(Z(S)/C_{E}(O^{p'}(\Aut_{\F}(E))))^\#|=q-1=|T|$, (2) holds.

Finally, if $E$ is non-abelian, then $Z(E)=Z(S)$ and $O^{p'}(\Aut_{\F}(E))$ acts trivially on $Z(S)$ by Lemma \ref{size.ess}(\ref{nonab.ess}).
\end{proof}

We can now prove the three main results of this section.

\begin{proof}[Proof of Proposition \ref{essentials}]
Let $E \in \mathcal E(\F) \setminus \{V\}$. We first show that $[V,S]E= [V,S]S_0$ and $E$ is $S$-conjugate to either $R$ or $Q$.

Let $T_E$ and $T$ be as in Lemma \ref{lem:TE} and  recall the definitions of $\Sigma$ from Notation \ref{n:subsofd*}, $P_n^*(q)$ from Notation \ref{n:snandpn} and $\delta$ from Notation \ref{n:deltamap}.  Choose $\tau_E\in T$ an element of order $q-1$ and observe that $\tau_E$ is conjugate in $\Aut(S)$ to an element of $\Aut_{P_n^*(q)}(S)$. Since the isomorphism type of $\F$ is unchanged by conjugating by elements of $\Aut(S)$, we may assume that $\tau_E \in \Aut_{P_n^*(q)}(S)$ and ${\tau_E}|_E$ is a generator of $T_E$.

Suppose first that $E$ is abelian. We see that $\tau_E\delta = (\tau,\tau^{-1})$ where $\tau$ is some primitive element of $\KK$. We know that $\tau_E$ normalises $V$, $[V,S]S_0$ and $[V,S]E$.

We may regard $\Sigma$ as a subgroup of $P_n^*(q)$ and consider the set of elements of $\Aut_{P_n^*(q)}(S)$ induced by $(\phi, \lambda, \begin{pmatrix}\mu&0\\0&\nu \end{pmatrix})\in \Sigma$, and which are mapped by $\delta$ to $(\tau^{-1},\tau)$ satisfy $\phi=1$, $\lambda\mu^n=\tau$ and $\mu\nu^{-1}= \tau^{-1}$. In other words,
\[\left((\tau^{-1},\tau){\delta|_{\Aut_{\Sigma}(S)}}\right)^{-1} = \{(1, \mu^{-n}\tau, \begin{pmatrix}\mu&0\\0&\tau\mu \end{pmatrix})\mid \mu \in \mathbb K\}.\] These elements act on $V/[V,S]$ as multiplication by $\mu^{-n}\tau(\tau \mu)^n= \tau^{n+1}$.

Now, $\tau_E$ acts on $S/[V,S]$ as $\diag (\tau^{-1}, \tau^{n+1})$ preserving the subgroups $V/[V,S]$ and $S_0[V,S]/[V,S]$. Assume that $E[V,S]\ne S_0[V,S]$. Then, as $\FF_p\langle \tau_E\rangle$-modules, we must have that  $V/[V,S]$ and $[V,S]S_0/[V,S]$ are isomorphic. This means that \[(\tau^{n+1})^{p^k}= \tau^{-1}\] for some $1\leq k \leq m$. Hence $np^k+p^k+1$ must be divisible by $p^m-1$. Since $m\geq 2$ and $n\leq p-1$, this has no solutions and we infer that $E[V,S] = S_0[V,S]$. From this it follows by Lemma \ref{lem:S-conj22} that $E \in R^S$.

The argument when $E$ is non-abelian is similar. In this case we obtain $\tau_E\in \Aut_{P_n^*(q)}(S)$ with $\tau_E\delta= (\tau^{-1},1)$ and this corresponds to the following subset of $\Sigma$:
\[\left((\tau^{-1},1){\delta|_{\Aut_{\Sigma}(S)}}\right)^{-1} = \{ (1, \mu^{-n}, \begin{pmatrix}\mu&0\\0&\tau\mu \end{pmatrix})\mid \mu \in \mathbb K\}.\]
We deduce that
\[\mathcal E(\F) \subseteq \{V \} \cup R^S \cup Q^S.\]
 Assume that some $S$-conjugate of $R$ is $\F$-essential. Then all $S$-conjugates of $R$ are $\F$-essential. In particular, $R\in\mathcal{E}(\F)$. Similarly, if some $S$-conjugate of $Q$ is $\F$-essential then $Q\in\mathcal{E}(\F)$. Note that $R < Q$. Hence by Lemma \ref{lem:intersec} we ascertain that either $\mathcal{E}(\F)\subseteq\{V \} \cup R^S$ or else $\mathcal{E}(\F)\subseteq \{V\} \cup Q^S$.

Finally, suppose that $O_p(\F)=1$. Applying \cite[Proposition I.4.5]{AKO} we see that $\mathcal{E}(\F)\ne \{V\}$ since $V$ is characteristic in $S$. If $\mathcal{E}(\F) = \{Q^S\}$ then  $Z(S)$  is left invariant by  $\Aut_\F(S)$ and $\Aut_{\F}(F)$ for any $F\in Q^S$ and so $Z(S)\normaleq \F$ by \cite[Proposition I.4.5]{AKO}, a contradiction. Thus $\mathcal{E}(\F) \neq \{Q^S\}$, as required.
\end{proof}

\begin{proof}[Proof of Proposition \ref{upVEssentials}]
If $E\in\mathcal{E}(\F)$ is abelian and $E \ne V$ then $E$ is elementary abelian and by Lemma \ref{lem: CharSub}, we have $E\le R[V, S]$ so $E[V, S]=R[V, S]$ and $E$ is $S$-conjugate to $R$ by Lemma \ref{lem:S-conj22}. Hence, either $R^S\subseteq \mathcal{E}(\F)$ or $R^S\cap \mathcal{E}(\F)=\emptyset$.

Similarly if $E\in\mathcal{E}(\F)$ is non-abelian, $E$ has exponent $p$ so $E\le Q[V, S]$ and $E$ is $S$-conjugate to $Q$ by Lemma \ref{lem:S-conj22}. Hence, either $Q^S\subseteq \mathcal{E}(\F)$ or $Q^S\cap \mathcal{E}(\F)=\emptyset$. Moreover, by Lemma \ref{lem:intersec}, we have  $R^S \cup Q^S \not\subseteq \mathcal{E}(\F)$.

Suppose that $E\in\mathcal{E}(\F)$ is non-abelian. Let $T_E$ be a Hall $p'$-subgroup of $N_{O^{p'}(\Aut_\F(E))}(\Aut_S(E))$ and  $T$ be its lift to $\Aut_{\F}(S)$. Suppose that $V$ is essential in $\F$. Then $T$ normalises $V$ and centralises $C_V(S)$.  Recall the definition of $P_\Lambda^*(q)$ from Notation \ref{n:slandpl}.  Since a Hall $p'$-subgroup of $\Aut_{\F}(S)$ is $\Aut(S)$-conjugate to a Hall $p'$-subgroup of $\Aut_{P_\Lambda^*(q)}(S)$, it suffices to analyse the elements of $\Sigma$ which act trivially on $C_{\Lambda(q)}(S_{\Lambda}(q))$. Recall that $C_{\Lambda(q)}(S_{\Lambda}(q))=\langle \overline{xy^{p-1}}, \overline{y^p}\rangle_{\KK}$. Calculating, we see that $$\overline{xy^{p-1}} \cdot \left( 1, \theta, \left(\begin{matrix}
a & 0 \\ 0 & b
\end{matrix}\right)\right) = \theta^{-1}a^{-1}b^{1-p} \overline{xy^{p-1}} \hspace{2mm} \mbox{ and } \hspace{2mm} \overline{y^{p}} \cdot \left( 1, \theta, \left(\begin{matrix}
a & 0 \\ 0 & b
\end{matrix}\right)\right) = \theta^{-1}b^{-p} \overline{y^{p}}. $$
Hence, any element which centralises $C_{\Lambda(q)}(S_{\Lambda}(q))$ satisfies $a=b$ and $\theta=a^{-p}$ and applying Lemma \ref{Gammacentraliser}, we see that such an element must centralise $\Lambda(q)$. This implies that $T$ centralises $V$, a contradiction, and we must have $\mathcal{E}(\F)\subseteq \{V\} \cup R^S$, as required.

If $O_p(\F)=1$ then since $V$ is characteristic in $S$ we must have $\mathcal{E}(\F)\ne \{V\}$  by \cite[Proposition I.4.5]{AKO}. Suppose that $\mathcal{E}(\F)=R^S$. Write $Z:=C_R(O^{p'}(\Aut_{\F}(R)))$. Then for $s\in S$, $Z=C_{R^s}(O^{p'}(\Aut_{\F}(R^s))) \le Z(S)$ is normalised by $\Aut_{\F}(S)$ and $\Aut_\F(P)$ for $P \in \E(\F)$, and so $$1\ne Z=\bigcap_{P\in\mathcal{E}(\F)} C_P(O^{p'}(\Aut_{\F}(P))) \le O_p(\F),$$ a contradiction. We conclude that $\mathcal{E}(\F)=\{V\} \cup R^S$, as desired.
\end{proof}

\begin{proof}[Proof of Proposition \ref{ess.auto}]
Suppose that $E \in \mathcal E(\F)$. If $E \in R^S \cup Q^S$ then $O^{p'}(\Out_\F(E)) \cong \SL_2(q)$ by Lemma \ref{size.ess} and (1) holds. If $E=V$ then $O^{p'}(\Aut_{\F}(V))$ is determined by Lemma \ref{lem: spe modules}, Proposition \ref{lem: Vnqrecog} and Proposition \ref{lem: upVRecog} and (2) holds.
\end{proof}

We end this section with two consequences of these results.

\begin{corollary}\label{cor: Opiff}
If $\F$ is a saturated fusion system on $S$ with $\{V\}\subsetneq\mathcal{E}(\F)$ then $O_p(\F)=1$.
\end{corollary}

\begin{proof}
Assume $\{V\}\subsetneq\mathcal{E}(\F)$. If $S\cong S_n(q)$ then Proposition \ref{ess.auto} shows that the action of $O^{p'}(\Aut_{\F}(V))$ is irreducible and $O_p(\F)\in\{1, V\}$ by \cite[Proposition I.4.5]{AKO}. Since  $R<Q$ and $\E(\F) \cap \{Q,R\}$ is non-empty by Proposition \ref{essentials}, it follows that $O_p(\F)\le Q$  so that $O_p(\F)=1$. Similarly if $S\cong S_{\Lambda}(q)$ then $\mathcal{E}(\F)=\{R^S, V\}$ and $O_p(\F)\le R\cap V=Z(S)$ by \cite[Proposition I.4.5]{AKO}. Since $C_V(O^{p'}(\Aut_{\F}(V)))=1$ by Proposition \ref{ess.auto} we conclude that $O_p(\F)=1$, as desired.
\end{proof}

\begin{corollary}\label{cor: frcdet}
If $\F$ is a saturated fusion system on $S$ then $\F^{frc}=\mathcal{E}(\F)\cup \{S\}$.
\end{corollary}
\begin{proof}
Plainly $\mathcal{E}(\F)\cup \{S\} \subseteq \F^{frc}$, so it remains to prove the converse. Suppose $P\in \F^{frc}$. If $P<S$ then the Alperin--Goldschmidt theorem implies that an $\F$-conjugate of $P$ is contained in an $\F$-essential subgroup. If all such subgroups are abelian then $\E(\F) \subseteq \{V\} \cup R^S$ and $P$ is $\F$-conjugate either to $\{V\}$ or to $R$ since $P$ is $\F$-centric, $V^\F=\{V\}$ and $R^\F=R^S$. We may thus  assume that $S=S_n(q)$, $\mathcal{E}(\F)\subseteq \{Q^S\}\cup \{V\}$ and some $\F$-conjugate $\hat{P}$ of $P$ is contained in $Q$. In particular note that $n>1$ and $C_S(\hat{P})\le \hat{P}$. Since $Q/Z(Q)$ is a natural module for $O^{p'}(\Out_{\F}(Q))\cong \SL_2(q)$, $O^{p'}(\Out_{\F}(Q))$ acts transitively on subgroups of order $q$ in $Q/Z(Q)$. If $|P|\leq q^2$ then let $Y$ be such that $Z(Q) \le \hat{P} \le Y \le Q$ and $|Y|=q^2$; if $|P|>q^2$ then let $Y$ be such that  $Z(Q) \le Y \le \hat{P}$  with $|Y|=q^2$. In either case, $Y$ is $\F$-conjugate to $Z_2(S)$. If $\hat{P}\le Y$, then $P$ is $\F$-conjugate to a subgroup of $Z_2(S)<V$, a contradiction since $V$ centralises $Z_2(S)$. Thus, $Z_2(S)=Q\cap V$ is properly contained in some $\F$-conjugate of $P$, $P$ is non-abelian and the Alperin--Goldschmidt theorem implies that $P\le Q^s$ for some $s\in S$. If $X \le P$ is any $\Aut_{\F}(Q^s)$-conjugate of $Z_2(S)$ then $Z(P)=C_X(P)=Z(S)=Z(Q^s)$ so $Q^s$ centralises the chain $1\normaleq Z(P)\normaleq P$ and we conclude by Lemma \ref{lem:series} that $Q^s=P$, as desired.
\end{proof}

\section{Uniqueness of the saturated fusion systems on $S$}\label{sec:tsfs}

As in Section \ref{sec:pot}, here we assume that $p$ is odd, $q=p^m$ for some $m > 1$, and $\F$ is a core-free saturated fusion system on $S \in \{S_n(q),S_{\Lambda}(q)\}$ for some $1\leq n\leq p-1$. By applying Propositions \ref{essentials} and \ref{upVEssentials}, we will determine $\F$ up to isomorphism, thereby completing the proofs of Theorem \ref{thm: main} and Theorem \ref{thm: upVMain}. We start by analysing some small cases.

\begin{proposition}\label{prop:murray}
If $S=S_1(q)$ then $\F=\F_S(G)$ for some group $G$ with $F^*(G)\cong \PSL_3(q)$.
\end{proposition}
\begin{proof}
Lemma \ref{Snauto}\ref{sn1} asserts that $S_1(q)$ isomorphic to a Sylow $p$-subgroup of $\PSL_3(q)$. The result then follows from \cite[Theorem 4.5.1]{Clelland}.
\end{proof}

\begin{proposition}\label{prop:HS}
If $S=S_2(9)$, $V\in\mathcal{E}(\F)$ and $O^{3'}(\Aut_{\F}(V))\cong 2^.\PSL_3(4)$ then $\F$ is one of the two Henke--Shpectorov exotic fusion system.
\end{proposition}

\begin{proof}
This can be verified using the MAGMA package \cite{parkersemerarocomputing}, and is only a moderate extension of \cite[Theorem 5.8]{parkersemerarocomputing}, which also assumes that $O^p(\F)=\F$. These fusion systems have the property that $\mathcal{E}(\F)=\{V\} \cup R^S$ and were first discovered in unpublished work of Henke--Shpectorov \cite{HS}. One can prove that they are exotic with an argument analogous to those in \cite[Lemma 5.9-Lemma 5.12]{ClellandParker2010} for the fusion system $\F(2, 9, R)$.
\end{proof}

For the remainder of this section, we suppose that $n\geq 2$ and that $O^{p'}(\Aut_{\F}(V))\cong \mathrm{(P)SL}_2(q)$ whenever $V\in\mathcal{E}(\F)$. The method we use to prove uniqueness is broadly the same for all the saturated fusion systems under consideration. Recall the fusion systems $\F^*_\Lambda(q)$ from Notation \ref{n:f*lambda}, the systems $\F^*(n,q,R)$ and $\F^*(n, q, Q)$ from Notation \ref{n:BigCPsystems} and the systems $\F^*(n,q,R)_P$ from Notation \ref{n:pruned}.

\begin{proposition}\label{aa}
Suppose that $\F$ is a saturated fusion system on a $p$-group $S$.
\begin{enumerate}
    \item If $S\cong S_n(q)$ and $\mathcal{E}(\F)=\{V\} \cup Q^S$ then $\Aut_{\F}^0(S)$ is $\Aut(S)$-conjugate to $\Aut_{\F^*(n, q, Q)}^0(S)$ and $|\Out_{\F}^0(S)|=\frac{(q-1)^2}{(n, q-1)}$.
    \item If $S\cong S_n(q)$ and $\mathcal{E}(\F)=\{V\} \cup R^S$ then $\Aut_{\F}^0(S)$ is $\Aut(S)$-conjugate to $\Aut_{\F^*(n, q, R)}^0(S)$ and $|\Out_{\F}^0(S)|=\frac{(q-1)^2}{(n+2, q-1)}$.
    \item If $S\cong S_n(q)$ and $\mathcal{E}(\F)=R^S$ then $\Aut_{\F}^0(S)$ is $\Aut(S)$-conjugate to $\Aut_{\F^*(n, q, R)_P}^0(S)$ and $|\Out_{\F}^0(S)|=q-1$.
    \item If $S\cong S_\Lambda(q)$ and $\mathcal{E}(\F)=\{V\}\cup R^S$ then $\Aut_{\F}^0(S)$ is $\Aut(S)$-conjugate to $\Aut_{\F^*_{\Lambda}(q)}^0(S)$ and $|\Out_{\F}^0(S)|=(q-1)^2$.
\end{enumerate}
\end{proposition}
\begin{proof}
We intend to apply \cite[Theorem I.7.7]{AKO} which shows that $\Aut_\F^0(S)$ is determined by lifts of elements in $O^{p'}(\Aut_{\F}(E))$ to $\Aut_\F(S)$ for $E \in\F^{frc}$. Observe that $\F^{frc}$ is equal to $\{S\} \cup \mathcal{E}(\F)$ by Corollary \ref{cor: frcdet}.

Suppose first that $S\cong S_n(q)$ and $\mathcal{E}(\F)=\{V\}\cup Q^S$. Let $k_Q$ be the lift of the element of order $q-1$ in $N_{O^{p'}(\Aut_{\F}(Q))}(\Aut_S(Q))$ to $\Aut_{\F}(S)$ provided by Lemma \ref{lem:extends}. Then $k_Q$ acts on $V$ and $k_Q \in C_{\Aut(S)}(Z(S))$. Choose $k$ to be an element of order $\frac{q-1}{(n,2)}$ in $N_{\Aut_{\F}(V)}(\Aut_S(V))$ which is in the same Hall $p'$-subgroup of $N_{\Aut_{\F}(V)}(\Aut_S(V))$ as $k_Q|_V$, and chosen such that the lift of $k$ to $\Aut_{\F}(S)$ normalises $Q=S_0Z_2(S)$. Note that  when $n$ is even, $O^{p'}(\Aut_{\F}(V))\cong \PSL_2(q)$ so $|C_{O^{p'}(\Aut_{\F}(V))}(C_V(S))|=q.\frac{(n, q-1)}{(n,2)}$ by Lemma \ref{l:centslem}. Write $k_V$ for the lift of $k$ to $\Aut_{\F}(S)$. Observe that $K_Q:=\langle k_Q\rangle$ normalises $K_V:=\langle k_V \rangle$ and so we need only show that $|K_Q\cap K_V|=\frac{(n, q-1)}{(n,2)}$ to demonstrate that $|\Out_{\F}^0(S)|=\frac{(q-1)^2}{(n, q-1)}$ as in (1). Since elements of $K_Q\cap K_V$ act trivially on $C_V(S)$, we must have $|K_Q\cap K_V|\leq \frac{(n, q-1)}{(n,2)}$. On the other hand, $C_{K_V}(Z(S))$ has order exactly $\frac{(n, q-1)}{(n,2)}$ and normalises $Q$. Since $Q=Z_2(S)S_0\cong S_1(q)$ by Corollary \ref{prop: Q iso}, Proposition \ref{prop: L3Q} implies that $C_{\Aut_{\F}(Q)}(Z(S))=O^{p'}(\Aut_{\F}(Q))$ so that $C_{K_V}(Z(S))|_{Q}\le O^{p'}(\Aut_{\F}(Q))$ and then $C_{K_V}(Z(S))|_{Q}\le \langle k_Q|_{Q}\rangle$. We obtain $\frac{(n, q-1)}{(n,2)}=|C_{K_V}(Z(S))|\leq |K_Q\cap K_V|\leq \frac{(n, q-1)}{(n,2)}$.

Next assume that $S\cong S_n(q)$ and $\mathcal{E}(\F)=\{V\}\cup R^S$. We intend to show that $|\Out_{\F}^0(S)|=\frac{(q-1)^2}{(n+2, q-1)}$ as in (2). Let $k_R$ be the lift of the element of order $q-1$ in $O^{p'}(\Aut_{\F}(R))$ to $\Aut_{\F}(S)$ so that $k_R$ acts on $V$. Let $k$ be an element of order $\frac{q-1}{(n,2)}$ in $N_{O^{p'}(\Aut_{\F}(V))}(\Aut_S(V))$ chosen so that the lift of $k$ to $\Aut_{\F}(S)$ normalises $R=S_0Z(S)$. We may arrange that $k_R|_V$ normalises $\langle k\rangle$ and write $k_V$ for the lift of $k$ to $\Aut_{\F}(S)$. The groups $K_R:=\langle k_R\rangle$ and $K_V:=\langle k_V \rangle$ can be arranged to commute, so it suffices to prove that $|K_R\cap K_V|=\frac{(n+2, q-1)}{(n,2)}$. If $1 \neq t\in K_V$ then $t|_R\in K_R|_R$ if and only if $t_R$ acts as $\lambda^{-1}$ on $S/V=RV/V\cong R/Z(S)$ and as $\lambda$ on $Z(S)$ for some $\lambda\in \KK^*$. Now, $K_V$ is the lift of an element of $O^{p'}(\Aut_{\F}(V))$ and since $V\cong V_n(q)|_{\FF_p}$, we may arrange for $t\in K_V$ to act as $\mu^2$ on $S/V$ and $\mu^n$ on $C_V(S)$ for some $\mu\in \KK^*$. Thus $\lambda^{-1}=\mu^2$ and $\lambda=\mu^n$ so $\mu^{n+2}=1$ which has exactly $(n+2, q-1)$ solutions, potentially including $\mu=-1$. It follows that $|K_R\cap K_V|=\frac{(n+2, q-1)}{(n,2)}$, as desired.

Next suppose that $S=S_n(q)$ and $\mathcal{E}(\F)=R^S$. In this case $|N_{O^{p'}(\Aut_{\F}(R))}(\Aut_S(R))|=q.(q-1)$ and so $|\Out_{\F}^0(S)|=q-1$, as in (3).

Lastly suppose that $S\cong S_{\Lambda}(q)$ and $\mathcal{E}(\F)=\{V\}\cup R^S$. Let $k_R$ be the lift  to $\Aut_{\F}(S)$ of an element of order $q-1$ in $O^{p'}(\Aut_{\F}(R))$ so that $k_R$ acts on $V$. Let $k$ to be any element of order $q-1$ in $N_{\Aut_{\F}(V)}(\Aut_S(V))$ in the same Hall $p'$-subgroup of $N_{\Aut_{\F}(V)}(\Aut_S(V))$ as $k_R|_V$, whose lift $k_V$ to $\Aut_{\F}(S)$ normalises $R=S_0Z(S)$. Then $K_R:=\langle k_R\rangle$ normalises $K_V:=\langle k_V\rangle$ and so to prove that $|\Out_{\F}^0(S)|=(q-1)^2$ as in (4), it suffices to verify $K_R\cap K_V=1$.  By Proposition \ref{ess.auto}, $V$ is isomorphic to $\Lambda(q)$ as an $\FF_pO^{p'}(\Aut_{\F}(V))$-module. Let $W$  be the unique proper non-trivial $O^{p'}(\Aut_{\F}(V))$-submodule of $V$. Then  $\Aut_{\F}(S)$ normalises $W$ since $V$ is characteristic in $S$ and $\Aut_\F(S)|_V$ normalises $O^{p'}(\Aut_\F(V))$. Let $t\in K_R\cap K_V$. Then since $t$ is the lift of an element of $O^{p'}(\Aut_{\F}(R))$, $t$ acts trivially on $C_V(S)W/W=C_R(O^{p'}(\Aut_{\F}(R)))W/W\cong C_R(O^{p'}(\Aut_{\F}(R)))$. On the other hand,  since $V/W$ is isomorphic as a $\FF_pO^{p'}(\Aut_{\F}(V))$-module to a natural $\SL_2(q)$-module, every non-trivial $p'$-element in $O^{p'}(\Aut_{\F}(V))$ acts fixed point freely on $V/W$. This implies that $t=1$ so that $K_R\cap K_V=1$, as desired.

To complete the proof, it remains to show that in all cases $\Aut_{\F}^0(S)$ is uniquely determined up to $\Aut(S)$-conjugacy.

Let $\mathfrak{D}$ be a Hall $p'$-subgroup of $\Aut(S)$ containing a Hall $p'$-subgroup $K$ of $\Aut_{\F}^0(S)$. Conjugating by $\Aut(S)$ if necessary we may assume that $\Aut_{\F}^0(S)\le \Aut_{\F^*(n ,q, R)}(S)$ and that $\mathfrak{D}$ is a Hall $p'$-subgroup of $\Aut_{\F^*(n ,q, R)}(S)$.  Recall the map $\delta:\Aut(S)\to \Delta=(\Aut(S, S/V), \Aut(S, Z(S)))$ from Notation \ref{n:deltamap}. By Lemma \ref{lem: kerneldescription}, $\delta$ is injective upon restriction to $\mathfrak{D}$. Thus in each case, there is $k\in \Aut_{\F}^0(S)$ cyclic of order $q-1$ determined by the lift of an element in $N_{O^{p'}(\Aut_{\F}(X))}(\Aut_S(X))$, for $X\in\{Q, R\}$ which determines a unique subgroup of $\mathfrak{D}$. Similarly if $V\in\mathcal{E}(\F)$, then the image under $\delta$ of lift to $S$ of an element of $\frac{q-1}{(n,2)}$ in $N_{O^{p'}(\Aut_{\F}(V))}(\Aut_S(V))$ uniquely determines the actions on $S/V$ and $Z(S)$. We conclude that $\Aut_{\F}^0(S)$ is uniquely determined.
\end{proof}

Suppose $\F^*$ is the largest polynomial fusion system on $S$ with a fixed nomination of essential subgroups. That is, $$\F^*\in\{\F^*(n, q, R), \F^*(n, q, Q), \F^*(n, q, R)_P, \F^*_{\Lambda}(q)\}$$ for $n,q$ associated to $S$ in the obvious way. Then since $\Aut_{\F}^0(S)$ is uniquely determined in $\Aut_{\F}(S)$ by Lemma \ref{lem: kerneldescription} and $\Aut_{\F^*}(S)$ contains a Hall $p'$-subgroup of $\Aut(S)$, we can arrange that $\Aut_{\F^*}^0(S)=\Aut_{\F}^0(S)\le \Aut_{\F}(S)\le \Aut_{\F^*}(S)$.

We begin with the case $\mathcal{E}(\F)=R^S$.
\begin{theorem}\label{justR}
If $S=S_n(q)$ for some $2\leq n\leq p-1$ and $\mathcal{E}(\F)=R^S$ then $\F$ is isomorphic to a $p'$-index subsystem of $\F^*(n ,q, R)_P$.
\end{theorem}
\begin{proof}
Set $\F^*=\F^*(n ,q, R)_P$ and assume that $\Aut_{\F}^0(S)=\Aut_{\F^*}^0(S)$ as above. Let $X:=O^{p'}(\Aut_{\F}(R))$ and $X^*=O^{p'}(\Aut_{\F^*}(R))$ so that \[N_X(\Aut_S(R))=\Aut_{\F}^0(S)|_R=\Aut_{\F^*}^0(S)|_R=N_{X^*}(R).\] By Lemma \ref{lem:unique over groups 2dim}, $O^{p'}(\Aut_{\F}(R))=O^{p'}(\Aut_{\F^*}(R))$ so $\F=\langle O^{p'}(\Aut_{\F}(R)), \Aut_{\F}(S)\rangle_S$ is a saturated subsystem of $\G$  containing $O^{p'}(\F^*)$ by the Alperin--Goldschmidt theorem, and the result follows from  \cite[Theorem I.7.7]{AKO}.
\end{proof}

In all remaining cases, we may assume that $V\in\mathcal{E}(\F)$ by Propositions \ref{essentials} and \ref{upVEssentials}.  Recall the group $D^\dagger$ from Notation \ref{n:DGroups}. 

\begin{proposition}\label{someprop}
Suppose that $V\in\mathcal{E}(\F)$, and $\F^*$ is a fusion system on $S$ equal to $\F^*(n, q, R)$, $\F^*(n, q, Q)$ or $\F^*_{\Lambda}(q)$. If $O^{p'}(\F)=O^{p'}(\F^*)$ then $\F$ is isomorphic to a $p'$-index subsystem of $\F^*$.
\end{proposition}
\begin{proof}
By \cite[Lemma 3.4]{oliver2020simplicity} and using  \cite[Theorem I.7.7]{AKO}, it suffices to prove that \begin{equation}\label {e:p'cond}O^p(N_{\Aut(V)}(\Aut_{O^{p'}(\F)}(V)))\le \Aut_{\F^*}(V). \end{equation} Observe first that in all cases, $$O^p(N_{\Aut(V)}(\Aut_{O^{p'}(\F)}(V)))\le O^p(N_{\Aut(V)}(O^{p'}(\Aut_{\F}(V)))).$$

If $S=S_n(q)$ and $\F^*=\F^*(n, q, X)$ for $X\in\{Q, R\}$, then \ref{clmx} in the proof of Lemma \ref{lem:unique over groups} implies that $N_{\Aut(V)}(O^{p'}(\Aut_{\F}(V)))=\Aut_{D^\dagger}(V)$ and (\ref{e:p'cond}) follows quickly  from this. Similarly if $S=S_\Lambda(q)$ and $\F^*=\F^*_\Lambda(q)$ then \ref{clmy} in the proof of Lemma \ref{lem:unique over groups lambda} implies that $N_{\Aut(V)}(O^{p'}(\Aut_{\F}(V)))=\Aut_{\F_\Lambda^*(q)}(V)$ when $p\geq 5$.

It thus remains to treat the case when $p=3$ and $S=S_\Lambda(q)$. Observe that $\Aut_{O^{p'}(\F)}(V)\cong \GL_2(q)$ and so $O^p(N_{\Aut(V)}(\Aut_{O^{p'}(\F)}(V)))\le \Aut_{\F_\Lambda^*(q)}(V)C_{\Aut(V)}(\Aut_{O^{p'}(\F)}(V))$. it thus suffices to show that $C:=C_{\Aut(V)}(\Aut_{O^{p'}(\F)}(V))=Z(\Aut_{O^{p'}(\F)}(V))$ has order $q-1$. Denote by $W$ the unique proper non-trivial $\FF_3\Aut_{O^{p'}(\F)}(V)$-submodule of $V$. Let $s$ be a $3$-element in $C$. Then $[V, s]$ and $[W, s]$ are both normalised by $\Aut_{O^{p'}(\F)}(V)$ and since $[V, s]W<V$ and $[W, s]<W$, we must have $[V, s]\le W$ and $[W, s]=1$. Let $K$ be a Hall $3'$-subgroup of $N_{\Aut_{O^{p'}(\F)}(V)}(\Aut_S(V))$ and define $K_1=C_K([V, S]/W)$, a group of order $q-1$. The Three Subgroups Lemma implies that $[s, [V, S], K_1]=1$ so that $[s, [V, S]]\le C_W(K_1)=1$, where the last equality follows from Lemma \ref{action on centreLambda}. Hence, $[V, S]\le C_V(s)$. But $C_V(s)$ is clearly normalised by $\Aut_{O^{p'}(\F)}(V)$, and we deduce that $V=C_V(s)$ and $s=1$. Hence, $C$ is a $3'$-group. But $Z(\Aut_{O^{p'}(\F)}(V))\le C$ and $Z(\Aut_{O^{p'}(\F)}(V))$ acts transitively on non-trivial vectors in $C_{V/W}(S)$. Using Lemma \ref{lem: big cyclic} we conclude that $C=C_{C}(C_{V/W}(S))Z(\Aut_{O^{p'}(\F)}(V))$. An application of the A$\times$B-Lemma \cite[(24.2)]{AschbacherFG} gives that $C_{C}(C_{V/W}(S))$ centralises $V/W$. But then $C_{C}(C_{V/W}(S))$ centralises $\Aut_{O^{p'}(\F)}(V)$ and so $C_V(C_{C}(C_{V/W}(S)))$ is a $\FF_3\Aut_{O^{p'}(\F)}(V)$-submodule of $V$. Since $V$ is indecomposable, $C_{C}(C_{V/W}(S))$ centralises $V$ so $C_{C}(C_{V/W}(S))=1$, and  $C=Z(\Aut_{O^{p'}(\F)}(V))$ has order $q-1$, as desired.
\end{proof}

By the Alperin--Goldschmidt theorem, for $\G=\{\F,\F^*\}$ and $X\in \mathcal{E}(\F)\setminus \{V\}$ we have,
\[O^{p'}(\G)=\langle \Aut_{\G}^0(S), O^{p'}(\Aut_{\G}(V)), O^{p'}(\Aut_{\G}(X))\rangle\] so to complete the proofs of Theorem \ref{thm: main} and Theorem \ref{thm: upVMain}, by Proposition \ref{someprop} and the fact that $\Aut_{\F}^0(S)=\Aut_{\F^*}^0(S)$,  it suffices to show that  $\Aut_{\F}^0(S)$ uniquely determines $O^{p'}(\Aut_{\F}(V))$ and each $O^{p'}(\Aut_{\F}(X))$.

\begin{proposition}\label{uniqueV}
Suppose that $V\in\mathcal{E}(\F)$, and $\F^*$ is a fusion system on $S$ equal to $\F^*(n, q, R)$, $\F^*(n, q, Q)$ or $\F^*_{\Lambda}(q)$.  If $\Aut_{\F}^0(S)=\Aut_{\F^*}^0(S)$, then $O^{p'}(\Aut_{\F}(V))=O^{p'}(\Aut_{\F^*}(V))$.
\end{proposition}
\begin{proof}
We observe that $O^{p'}(\Aut_{\F}(V))\cong \mathrm{(P)SL}_2(q)$ and $N_{O^{p'}(\Aut_{\F}(V))}(\Aut_S(V))$ is contained in the restriction of $\Aut_{\F}^0(S)$ to $V$. Set $G:=\langle O^{p'}(\Aut_{\F}(V)), \Aut_{\F}^0(S)|_V\rangle\le \Aut_{\F}(V)$.

If $S=S_n(q)$ then $|\Out_{\F}^0(S)|\geq \frac{(q-1)^2}{(a, q-1)}$ for $a\in\{n, n+2\}$ by Proposition \ref{aa}, and so upon restriction to $V$ a Hall $p'$-subgroup of $N_G(\Aut_S(V))$ has an abelian subgroup of order at least $\frac{(q-1)^2}{p+1}$. Furthermore, if $m=2$ and $|C_G(\Aut_S(V))|=p-1$ (so $a=n+2$ and $R\in\mathcal{E}(\F)$) then since there is $k\in N_{O^{p'}(\Aut_{\F}(R))}(\Aut_S(R))$ which lifts to $\Aut_{\F}(S)$ and restricts faithfully to $N_G(\Aut_S(V))$, the hypotheses of Lemma \ref{lem: NormalizerBorel} are satisfied. Hence the hypotheses of Lemma \ref{lem:unique over groups} are also satisfied and $O^{p'}(\Aut_{\F}(V))=O^{p'}(G)=O^{p'}(\Aut_{\F^*(n, q, X)}(V))$ for $X\in\{Q, R\}$, as desired.

If $S=S_{\Lambda}(q)$ then by Proposition \ref{aa} we have $|\Out_{\F^0}(S)|\geq (q-1)^2$ and upon restriction to $V$, a Hall $p'$-subgroup of $N_G(\Aut_S(V))$ has an abelian subgroup of order $(q-1)^2$. Now $G$ satisfies the hypotheses of Lemma \ref{lem:unique over groups lambda} and so $O^{p'}(\Aut_{\F}(V))=O^{p'}(G)=O^{p'}(\Aut_{\F^*_{\Lambda}(q)}(V))$, as we hoped.
\end{proof}

\begin{theorem}\label{thm1Q}
If $S=S_n(q)$ and $\mathcal{E}(\F)=\{V\}\cup Q^S$ then $O^{p'}(\F)\cong O^{p'}(\F^*(n ,q, Q))$ and $\F$ is isomorphic to a $p'$-index subsystem of $\F^*(n ,q, Q)$.
\end{theorem}
\begin{proof}
Applying Propositions \ref{someprop} and \ref{uniqueV} it suffices to show that $\Aut_{\F}^0(S)=\Aut_{\F^*(n ,q, Q)}^0(S)$ uniquely determines $O^{p'}(\Aut_{\F}(Q))$. Now, $O^{p'}(\Out_{\F}(Q))\cong \SL_2(q)$ by Proposition \ref{ess.auto} and $N_{O^{p'}(\Aut_{\F}(Q))}(\Aut_S(Q))$ is contained in the restriction of $\Aut_{\F}^0(S)$ to $Q$. Note that $|C_{\Aut_{\F}^0(S)}(Z(S))|=|\Inn(S)|.(q-1)$ and $|N_{C_{\Aut_{\F}^0(S)}(Z(S))}(Q)| = q^3.(q-1)$. Hence $$N_{O^{p'}(\Aut_{\F^*(n ,q, Q)}(Q))}(\Aut_S(Q))=N_{C_{\Aut_{\F}^0(S)}(Z(S))}(Q)|_{Q}=N_{O^{p'}(\Aut_{\F}(Q))}(\Aut_S(Q)).$$  Write $X:=O^{p'}(\Aut_{\F}(Q))$ and $X^*:=O^{p'}(\Aut_{\F^*(n ,q, Q)}(Q))$ so that $N_X(\Aut_S(Q))=N_{X^*}(\Aut_S(Q))$. By Proposition \ref{prop: L3Q} and Corollary \ref{prop: Q iso}, $\Aut(Q)=CH$ where $H\cong\Gamma\mathrm{L}_2(q)$ and $C$ is a normal $p$-subgroup of central automorphisms of $Q$. It follows immediately that $XC=X^*C$. Let $\tau$ be an involution in $N_X(\Aut_S(Q))$, so $X=\Inn(Q) C_X(\tau)$ and $X^*=\Inn(Q) C_{X^*}(\tau)$. Since $[C_{C}(\tau), Q, \tau]=1$, the Three Subgroups Lemma implies that $[Q, C_{C}(\tau)]=[[\tau, Q]Z(Q), C_{C}(\tau)]=1$, and so $C_{C}(\tau)=1$. We deduce that $C_{X}(\tau)=C_{XC}(\tau)=C_{X^*C}(\tau)=C_{X^*}(\tau)$, so $X=X^*$ and $O^{p'}(\Aut_{\F}(Q))=O^{p'}(\Aut_{\F^*(n, q, Q)}(Q))$, as desired.
\end{proof}

\begin{theorem}\label{thm2Q}
If $S=S_n(q)$ and $\mathcal{E}(\F)=\{V\}\cup R^S$ then $O^{p'}(\F)\cong O^{p'}(\F^*(n ,q, R))$ and $\F$ is isomorphic to a $p'$-index subsystem of $\F^*(n ,q, R)$.
\end{theorem}
\begin{proof}
By Propositions \ref{someprop} and \ref{uniqueV} it suffices to show that $\Aut_{\F}^0(S)$ uniquely determines $O^{p'}(\Aut_{\F}(R))$. Set $X:=O^{p'}(\Aut_{\F}(R))$ and $X^*=O^{p'}(\Aut_{\F^*(n ,q, R)}(R))$ and observe that $N_{X}(\Aut_S(R))$ is contained in the restriction of $\Aut_{\F}^0(S)$ to $R$. Upon demonstrating that $N_{X}(\Aut_S(R))=N_{X^*}(\Aut_S(R))$, the result will follow from Lemma \ref{lem:unique over groups 2dim}. Recall that $\Aut_{\F}^0(S)=\Aut_{\F^*(n ,q, R)}^0(S)$ and let $k$ be the lift of an element of order $q-1$ in $N_{X}(\Aut_S(R))$ to $\Aut_{\F}^0(S)$. Then $k$ acts as $\mu^{-1}$ on $S/V$ and as $\mu$ on $C_V(S)$. By Lemma \ref{lem: kerneldescription}, $\delta$ is injective upon restricting to a Hall $p'$-subgroup of $\Aut_{\F^*(n ,q, R)}^0(S)$, so the restriction of $k$ to $R$ lies in $N_{X^*}(\Aut_S(R))$. But then $N_{X}(\Aut_S(R))=\langle \Aut_S(R), k\rangle=N_{X^*}(\Aut_S(R))$ and so  Lemma \ref{lem:unique over groups 2dim} applies. This completes the proof.
\end{proof}

\begin{theorem}\label{upVUnique}
If  $S=S_\Lambda(q)$ and $\mathcal{E}(\F)=\{V\}\cup R^S$ then $O^{p'}(\F)\cong O^{p'}(\F^*_\Lambda(q))$ and $\F$ is isomorphic to a $p'$-index subsystem of $\F^*_\Lambda(q)$.
\end{theorem}
\begin{proof}
By Propositions \ref{someprop} and \ref{uniqueV} it suffices to show that $\Aut_{\F}^0(S)$ uniquely determines $O^{p'}(\Aut_{\F}(R))$. Let $k$ be the lift to $\Aut_{\F}(S)$ of an element of $N_{O^{p'}(\Aut_{\F}(R))}(\Aut_S(R))$ of order $q-1$. Since $\Aut_{\F}^0(S)=\Aut_{\F^*_{\Lambda}(q)}^0(S)$, $k$ lies in a Hall $p'$-subgroup of $N_{\Aut_{\F^*_{\Lambda}(q)}^0(S)}(R)$ and there is a subgroup $K_V$ of order $q-1$ satisfying $$K_V|_V\in O^{p'}(\Aut_{\F^*_{\Lambda}(q)}(V)), \hspace{2mm} K_V\le N_{\Aut_{\F^*_{\Lambda}(q)}(S)}(R), \hspace{2mm} \mbox{ and }\hspace{2mm} \langle k\rangle\cap K_V=1.$$ Thus $K_V$ restricts faithfully to a group of automorphisms of $R$ and as $K_V\le \Aut_{\F^*_{\Lambda}(q)}^0(S)=\Aut_{\F}^0(S)$, we deduce that ${K_V}|_R\in \Aut_{\F}(R)$. It follows that $\Aut_{O^{p'}(\F)}(R)\cong \GL_2(q)$. By similar arguments, we see that $N_{\Aut_{O^{p'}(\F)}(R)}(\Aut_S(R))=N_{\Aut_{O^{p'}(\F^*_{\Lambda}(q))}(R)}(\Aut_S(R))$. Set $\mathbb{B}:=N_{\Aut_{O^{p'}(\F)}(R)}(\Aut_S(R))$ and $$Y:=\Aut_{O^{p'}(\F)}(R), \hspace{2mm} X:=O^{p'}(Y), \hspace{2mm} Y^*:=\Aut_{O^{p'}(\F^*_{\Lambda}(q))}(R) \hspace{2mm} \mbox{ and } \hspace{2mm} X^*=O^{p'}(Y^*).$$ Then $Z(Y)=Z(Y^*)=Z(\mathbb{B})$ has order $q-1$ and $X, X^*\le C_{\Aut(R)}(Z(Y))=:C$. Let $t$ be the unique involution in $Z(Y^*)$ so that $t\in Z(X^*)$ and $R=[R, t]\times C_R(t)$. Then $C$ normalises $[R, t]$ and $C_R(t)$. Moreover, both $X$ and $X^*$ act faithfully on $[R, t]$ and centralise $C_R(t)$. An application of Lemma \ref{lem:unique over groups 2dim} reveals that $X\hat{C}=X^*\hat{C}$ where $\hat{C}:=C_C([R, t])$.

Let $c$ be a $p$-element in $\hat{C}$. Then $c$ acts trivially on $[R, t]$ and $[c, C_R(t)]<C_R(t)$. Since $Z(Y)$ acts irreducibly on $C_R(t)$ and commutes with $c$, we conclude that $[c, C_R(t)]=1$ so that $c=1$. Hence, $\hat{C}$ is a $p'$-group. We claim that $\hat{C}$ is cyclic of order at most $q-1$. Note that $\hat{C}$ acts faithfully on $C_R(t)$, $Z(Y)\cap \hat{C}=\langle t\rangle$ and $\Aut_{Z(Y)}(C_R(t))$ is cyclic of order $\frac{q-1}{2}$. Moreover, $\Aut_{Y}(C_R(t))$ is cyclic of order $q-1$, so contains a Singer cycle of $\GL_m(p)$, and centralises $\Aut_{Z(Y)}(C_R(t))$.  From these points we conclude, using $q>p$, that $C_{\GL_m(p)}(\Aut_{Z(Y)}(C_R(t)))=\Aut_{Y}(C_R(t))$ has order $q-1$. But $\Aut_{\hat{C}}(C_R(t))$ centralises $\Aut_{Z(Y)}(C_R(t))$ and we therefore $C_R(t)$ is cyclic of order at most $q-1$. Now, $X$ acts on $\hat{C}$ and as $X\cong \SL_2(q)$, we deduce that $X$ centralises $\hat{C}$. But then $X=O^{p'}(X\hat{C})$ and so $X=X^*$ as desired.
\end{proof}

Finally, we stitch all the pieces together to prove our main theorems.

\begin{proof}[Proofs of Theorems \ref{thm: main} and \ref{thm: upVMain}]
Suppose that $S= S_n(q)$ for some $1\le n \le p-1$ and $q$ is a power of $p$. Assume that $\F$ is a saturated fusion system on $S$ with $O_p(\F)=1$. If $q=p$ then $S$ has an elementary abelian subgroup of index $p$ and so $\F$ is known by \cite{p.index1, p.index2}. Hence we may assume that $q> p$. If $n=1$, then Proposition \ref{prop:murray} shows that $\F$ is realised by a group $G$ with $F^*(G) \cong \PSL_3(q)$.  In particular, $\F$ is a polynomial fusion system in this case. If $p=2$ then $S_{\Lambda}(q)$ has nilpotency class two and Theorem \ref{thm: upVMain}(1) follows from \cite[Proposition 4.2]{vB24}

So assume that $p$ is odd and $n\geq 2$. Then Propositions \ref{essentials} and \ref{ess.auto} give the possibilities for the configuration of essential subgroups and their automisers. In the exceptional case that $S=S_2(9)$ and $O^{3'}(\Aut_\F(V))\cong 2^.\PSL_3(4)$, Proposition \ref{prop:HS} shows that $\F$ is one of the two exotic Henke--Shpectorov systems. Thus if $V$ is $\F$-essential then $O^{p'}(\Aut_\F(V))\cong \mathrm{(P)SL}_2(q)$ and $V \cong V_n(q)$. Furthermore, in this situation either $Q$ or $R$ is also $\F$-essential.  The uniqueness of these fusion systems is then addressed in Theorem \ref{thm1Q} and Theorem \ref{thm2Q}. The only remaining case from Proposition \ref{essentials} is that $\E(\F)=R^S$, and this is the subject of Theorem \ref{justR}.

Next suppose that $S=S_\Lambda(q)$. Once again, if $q=p$ then $S$ has an elementary abelian subgroup of index $p$ and so $\F$ is known by \cite{p.index1, p.index2}. Hence we may assume that $q>p$ whence Proposition \ref{upVEssentials} implies that $\{V,R\} \subseteq \E(\F)$. The uniqueness of $\F$ is addressed in Theorem \ref{upVUnique}.
\end{proof}

\begin{appendices}

\section{Appendix}

 We begin this appendix with some group and module recognition results which, while not needed in the proofs of our main results, may still be of interest.

The following proposition is an analogue of \cref{lem: Vnqrecog} for the module $V_p(q)$.

\begin{proposition}\label{lem: VpRecog}
Suppose that $H$ is a finite group with $S\in\syl_p(H)$. Assume that $S\cong S_p(q)$ for $q>p$, where $p$ is an odd prime, and $H/O_p(H)\cong \SL_2(q)$. Then, writing $V:=O_p(H)$ and recognising $V$ as an $\FF_p\SL_2(q)$-module, we have that $V\cong V_p(q)$.
\end{proposition}
\begin{proof}
Write $V^*$ for the dual of the $\FF_p[H/V]$-module $V \cong V_p(q)$. Thus $V^*\cong \Lambda(q)$ as $\FF_p[S/V]$-modules and therefore as $\FF_p[H/V]$-modules by \cref{lem: upVRecog}. Finally, $V\cong (V^*)^*\cong V_p(q)$ as $\FF_p[H/V]$-modules, as desired.
\end{proof}

Suppose $q=3^m$, $G\cong \SL_2(q)$ and $S_0\in\syl_3(G)$. In \cref{prop: indiden} we recognised $\Lambda(q)$ from a putative indecomposable $\FF_3$-module $V$ with two composition factors each isomorphic to $V_1(q)$, on which $S=VS_0$ does not act quadratically and for which $|C_V(S)|=q^2$. These latter two hypotheses are justified by the following result, which shows that they arise when a fusion system on $S$ has additional essential subgroups.

\begin{proposition}\label{NilesRecog}
Suppose that $\F$ is a saturated fusion system on a $3$-group $S$, and let $q=3^m>3$. Assume that there is $V\normaleq S$ satisfying the following:
\begin{enumerate}
    \item $V$ elementary abelian;
    \item $V\in\mathcal{E}(\F)$ with $O^{3'}(\Out_{\F}(V))\cong \SL_2(q)$; and
    \item there is $W<V$ such that $W\cong V/W\cong V_1(q)$ as $\FF_3O^{3'}(\Aut_{\F}(V))$-modules.
\end{enumerate}
If $V\not\normaleq \F$ then $|C_V(S)|=q^2$ and $S$ does not act quadratically on $V$. In particular, $S\cong S_{\Lambda}(q)$.
\end{proposition}
\begin{proof}
Let $S$ and $V$ be as described in the proposition. Note that each $s\in S\setminus V$ acts cubically on $V$. Indeed, we have that $[V, s]W=[V, S]W$ and $[[V, S]W, s]=C_W(S)$ for all such $s$. It follows that $[V, S]W=Z_2(S)$ has order $q^3$.

Assume that $E\in\mathcal{E}(\F)\setminus \{V\}$ so that $Z_2(S)\le N_S(E)$. Suppose that $\Omega(Z(E))\not\le V$. Then $E\cap V\le C_V(\Omega(Z(E)))$ has order at most $q^2$ so that $|Z_2(S)E/E|\geq q$. An application of \cref{lem: SEFF} reveals that $S=\Omega(Z(E))V$, $E=\Omega(Z(E))$ and $Z_2(S)E=N_S(E)$. Moreover, $E\cap V=Z(S)$ has order $q^2$. If $S$ was quadratic on $V$ then since $S=EV$, we must have that $[E, V]\le C_V(E)=E\cap V=Z(S)$ from which we conclude that $E\normaleq S$, a contradiction. Hence, the proposition holds in this case.

Suppose now that $\Omega(Z(E))\le V$. Then $[Z_2(S), E]\le Z(S)\le \Omega(Z(E))$ and an application of \cref{lem:series} yields that $Z_2(S)\le E$. Then $E\normaleq S$ and $Z_2(S)$ is an elementary abelian subgroup of $E$ of order $q^3$. Let $A$ be an elementary abelian subgroup of $E$ of order at least $q^3$. Assume that $A\not\le V$. Then $q^3\leq |A| \leq |AV/V||C_V(A)|\leq q.q^2$ and so $|A|=q^3$. Suppose that every elementary abelian subgroup of $E$ of maximal order has order $q^3$, and there exists such an $A$ lying outside of $V$. In particular, $E\cap V=Z_2(S)$. Then $A\cap V$ is centralised by $S=AV$ and so $Z(S)=Z(E)$ has order $q^2$. If $S$ acted quadratically on $V$, $S'=[A, V]\le C_A(V)=Z(S)$ and \cref{lem:series} implies that $S\le E$, a contradiction.

Assume now that $E\cap V$ is the unique elementary abelian subgroup of $E$ of maximal order. Then \cref{lem:series} implies that $V\le E$ and so $V$ is characteristic in $E$. Further assuming that the conclusion of the proposition is false, we must have that $V$ is characteristic in every essential subgroup of $\F$. Moreover, $V$ is the unique maximal elementary abelian subgroup of $S$ and applying \cite[Proposition I.4.5]{AKO} we see that $V\normaleq \F$, a contradiction.
\end{proof}

We end the appendix with a general saturation result for fusion systems.

\begin{lemma}\label{l:fzssat}
Suppose $\F$ is an $\F^{cr}$-generated fusion system on $S$ such that $\F=C_\F(Z)$ for some $Z\le Z(S)$. If $\F/Z$ is saturated then $\F$ is saturated.
\end{lemma}
\begin{proof}
Write $\overline{\F}:=\F/Z$. By \cite[Theorem I.3.10]{AKO} it suffices to show that $\F$ is $\F^{cr}$-saturated. If $P \in \F^{cr}$ then $\overline{P} \in \overline{\F}^{cr}$ by \cite[Proposition 5.60]{Craven} and since $\overline{\F}$ is $\overline{\F}^{cr}$-generated, $\overline{P}$ is conjugate to a fully $\overline{\F}$-automised, $\overline{\F}$-receptive subgroup $\overline{Y}$. Then $\overline{Y}$ is fully $\F$-normalised by \cite[Lemma I.2.6]{AKO} so that $|N_{\overline S}(\overline Y)|\geq |N_{\overline S}(\overline A)|$ for all $\overline A\in \overline{Y}^{\overline{\F}}$. Then for any $B\le S$ with $Z\le B$, we have that $|N_S(B)/Z|=|N_{\overline S}(\overline B)|$ and we deduce that $|N_S(Y)/Z|\geq |N_S(A)/Z|$. Therefore, $|N_S(Y)|\geq |N_S(A)|$ and $Y$ is fully $\F$-normalised. Applying \cite[Lemma I.2.5]{AKO} yields that $Y$ is fully $\F$-automised.

If $\varphi \in \Iso_\F(X,Y)$ then since $\overline{Y}$ is $\overline{\F}$-receptive there exists $\overline{\psi} \in \Hom_{\overline{\F}}(N_{\overline{\varphi}}, S)$ extending $\overline{\varphi} \in \Hom_{\overline{\F}}(\overline{X},\overline{Y})$. Since $\overline{N_\varphi} \subseteq N_{\overline{\varphi}}$, we infer that there exists $\chi \in \Hom_\F(N_\varphi,S)$ extending $\varphi$ and $Y$ is $\F$-receptive. This shows that $\F$ is $\F^{cr}$-saturated and so $\F$ is saturated.
\end{proof}
\end{appendices}

\bibliographystyle{alpha}
\bibliography{my.books}

\end{document}